\documentclass[11pt,a4paper]{article}

\usepackage[T1]{fontenc}
\usepackage[utf8]{inputenc}
\usepackage{lmodern}
\usepackage[english]{babel}

\usepackage[margin=1in]{geometry}
\usepackage{microtype}
\usepackage{setspace}

\usepackage{amsmath}
\usepackage{amsfonts}
\usepackage{amssymb}
\usepackage{amsthm}
\usepackage{mathtools}
\usepackage{bm}
\usepackage{dsfont}

\allowdisplaybreaks
\mathtoolsset{showonlyrefs=true}
\numberwithin{equation}{section}

\usepackage{graphicx}

\usepackage{epstopdf}

\usepackage{float}
\usepackage{multicol}
\usepackage{multirow}
\usepackage{booktabs} 
\usepackage{tabularx} 
\usepackage{mathrsfs}
\usepackage{caption}
\usepackage{subcaption}
\usepackage{algorithm}
\usepackage{algpseudocode}
\usepackage{listings}

\usepackage{comment}
\usepackage{enumitem}
\usepackage{textcomp}
\usepackage{manyfoot}

\usepackage[dvipsnames]{xcolor}

\definecolor{NoteYellow}{RGB}{255,250,200}
\definecolor{NoteBlue}{RGB}{240,248,255}
\definecolor{NoteGreen}{RGB}{240,255,240}

\usepackage[authoryear,round]{natbib}

\newcommand{\cF}{\mathcal{F}}

\newcommand{\cK}{\mathcal{K}}

\newcommand{\cU}{\mathcal{U}}

\newcommand{\EE}{\mathbb{E}}
\newcommand{\FF}{\mathbb{F}}

\newcommand{\NN}{\mathbb{N}}
\newcommand{\PP}{\mathbb{P}}
\newcommand{\QQ}{\mathbb{Q}}
\newcommand{\R}{\mathbb{R}}

\newcommand{\drm}{\mathrm{d}}

\DeclareMathOperator{\tr}{tr}

\DeclareMathOperator{\Int}{Int}

\newcommand{\ep}{\epsilon}

\newcommand{\ti}{\widetilde}

\newcommand{\Uad}{\cU_{ad}}
\newcommand{\dd}{\,\drm}
\newcommand{\ds}{\,\drm s}
\newcommand{\ind}{\mathds{1}}
\newcommand{\Nt}{\widetilde{N}}
\newcommand{\Yt}{\ti{Y}}
\newcommand{\CM}{C_M}

\newcommand{\Sp}{\mathcal S}
\newcommand{\HW}{\mathcal H_W}
\newcommand{\JN}{\mathcal J_N}
\newcommand{\JY}{\mathcal J_Y}
\newcommand{\Bspace}{\mathcal B}
\newcommand{\bX}{\bar X}
\newcommand{\bM}{\bar M}
\newcommand{\bmu}{\bar\mu}
\newcommand{\bu}{\bar u}
\newcommand{\bN}{\bar N}
\DeclareMathOperator{\Lip}{Lip}
\DeclareMathOperator{\Pred}{Pred}
\DeclareMathOperator{\sech}{sech}

\newcommand{\Dmu}{D_\mu}

\newcommand{\doi}[1]{\href{https://doi.org/#1}{\nolinkurl{#1}}}
\DeclareMathOperator{\Law}{Law}
\newcolumntype{Y}{>{\raggedright\arraybackslash}X}
\newcommand{\norm}[1]{\left\lVert #1\right\rVert}

\newcommand{\inner}[2]{\left\langle #1,#2\right\rangle}

\theoremstyle{plain}

\newtheorem{theorem}{Theorem}[section]
\newtheorem{proposition}[theorem]{Proposition}
\newtheorem{lemma}[theorem]{Lemma}
\newtheorem{corollary}[theorem]{Corollary}
\newtheorem{assumption}[theorem]{Assumption}

\theoremstyle{definition}

\theoremstyle{remark}

\newtheorem{remark}[theorem]{Remark}

\usepackage[
unicode=true,
colorlinks=true,
linkcolor=NavyBlue,
citecolor=BrickRed,
urlcolor=NavyBlue,
filecolor=NavyBlue,
bookmarksopen=true,
bookmarksnumbered=true
]{hyperref}

\hypersetup{pdftitle={Partially Observed McKean--Vlasov Control with Vanishing Observation Intensities},pdfauthor={Jie Xiong and Ying Yang}}

\title{Partially Observed McKean--Vlasov Control with Vanishing Observation Intensities\thanks{This work is supported by the National Key R\&D Program of China under Grant No.~2022YFA1006102, and the National Natural Science Foundation of China under Grant No.~12471418.}}
\author{ 
	Jie Xiong\thanks{
		Department of Mathematics and Shenzhen International Center for Mathematics,
		Southern University of Science and Technology, Shenzhen, China.
		Email: \texttt{xiongj@sustech.edu.cn}.
	}
	\and
	Ying Yang\thanks{
		Department of Mathematics,
		Southern University of Science and Technology, Shenzhen, China.
		Email: \texttt{12331007@mail.sustech.edu.cn}.
	}
}
\date{}
\begin{document}
	\maketitle
\begin{abstract}
	We study partially observed McKean--Vlasov jump-diffusions with controlled counting intensities that may vanish and depend on the unconditional physical state distribution. A deterministic upper bound on the likelihood supports one-sided state--likelihood variations, a smooth-test-function expansion of the physical state marginal, and a necessary maximum principle without inverse intensity or likelihood terms. We also establish a one-sided total-variation derivative of the physical probability on the fixed noise space, including a nonnegative singular component on baseline-null histories. A quadratic example shows, within a specified policy family and under explicit conditions, that positive variance feedback raises the first-order activation fee threshold while negative feedback lowers it.
\end{abstract}
	
	\section{Introduction} 
	
Counting processes provide a natural model for observations consisting of the times and types of discrete events. When their intensities are controlled, decisions affect the information available for subsequent actions. A difficulty arises when an observation channel is inactive under a baseline policy: histories containing an arrival in that channel have zero baseline probability but may acquire probability of first order under activation. The marginal cost of activation may therefore depend on continuation actions that do not affect the baseline cost.

In mathematical finance, \citet{zeng2003partially} models transaction prices through counting observations of a hidden value process and develops filtering and Bayesian estimation methods. Subsequently, \citet{xiong2011branching} construct branching particle approximations for the corresponding filters. \citet{xiong2020mean} study mean–variance portfolio selection under counting observations, establishing a separation principle and deriving optimal strategies and the efficient frontier. However, these works impose a positive lower bound on the total trading intensity, rather than on each price-level intensity. Therefore, merely allowing vanishing component intensities does not set the present work apart.

Maximum principles for partially observed stochastic control and for systems with jumps also have an extensive literature; see, for example, \citet{tang1994necessary,tang1998maximum,baghery2007maximum}. Mean-field control under partial observation has been studied by \citet{buckdahn2017mean,nie2022extended,li2023stochastic}. In particular, \citet{nie2022extended} allow coefficients to depend on the joint distribution of the state and control in a Brownian observation framework. \citet{fuhrman2026optimal} develop a change-of-probability and dynamic programming approach to distribution-dependent partial observations with hidden Markov switching. Related progressive-control formulations treat partially observed mean-field systems with random jumps \citep{chen2023jumps}, mean-field systems with Markov chains \citep{chen2023markov}, and forward--backward systems with random jumps \citep{zheng2023global}.

Weighted mean-field control and the associated scalar and vector adjoints are studied in \citet{tang2023weighted,tang2026stochastic}. The analysis of \citet{tang2026stochastic} assumes that the weight jump multipliers are uniformly bounded away from zero and uses logarithmic estimates involving inverse jump multipliers. That work also permits distribution dependence in the jump coefficients and Brownian terms in the weight equation. For counting observations, \citet{xiong2026stochastic} derive coupled state--likelihood variations and a maximum principle under uniformly positive state-dependent intensities. \citet{confortola2013backward} treat nonnegative controlled intensity multipliers, including zero, using policies specified on complete jump histories.

In this paper, we study partially observed McKean--Vlasov jump-diffusions with controlled counting intensities that may vanish and depend on the hidden state, its unconditional physical distribution, and the control. We consider observation-predictable policies specified on complete counting histories. Under a fixed dominating reference probability $\QQ$, the likelihood $M^u$ and state $X^u$ determine the physical probability and its state marginal through
\[
\frac{d\PP^u}{d\QQ}\bigg|_{\mathcal F_T}
=M_T^u,
\qquad
\langle\mu_t^u,\varphi\rangle
=\EE^{\QQ}\!\left[M_t^u\varphi(X_t^u)\right].
\]
Although the likelihood may vanish, bounded nonnegative intensities yield a deterministic upper bound. This property enables a joint analysis of the state, likelihood, and endogenous physical distribution without requiring inverse intensities or likelihoods. 
The reference-measure formulation does not remove the partial-observation constraint: controls remain observation-predictable, and the maximum principle involves the observation-predictable projection of the Hamiltonian derivative.

Under the stated regularity and moment assumptions, we establish well-posedness and stability in Theorem~\ref{thm:forward}, followed by one-sided state--likelihood differentiability and a first-order expansion of the physical state marginal against specified smooth test functions in Theorem~\ref{thm:differentiability}. Using the cost expansion and adjoint duality, we derive a necessary maximum principle for a convex control set in Theorem~\ref{thm:SMP}. In Proposition~\ref{prop:singular}, we also obtain a one-sided total-variation derivative of the physical probability on the fixed noise space, with decomposition
\[
\dot{\PP}
=S_T\overline{\PP}+\nu^\perp,
\qquad
\nu^\perp\ge0,
\qquad
\nu^\perp\perp\overline{\PP}.
\]
The singular component captures first-order probability mass on baseline-null histories. This result is distinct from the test-function expansion of the moving state marginal.

Finally, we consider a quadratic example in which early observation activation changes the physical variance of accumulated actions and thereby affects later observation intensity. We identify the optimal benchmark among all admissible policies keeping the early observation interval inactive. For deterministic monitoring schedules, bounded continuous first-arrival responses, and later count-based decisions, Proposition~\ref{prop:feedbackactivation} gives a uniform $O(\varepsilon^2)$ cost expansion, both with the baseline final decision and after optimizing the final count-based decision. Under coercivity and nonbinding response bounds, Corollary~\ref{cor:feedbackthreshold} identifies the optimal first-order response and shows, within this policy family, that positive variance feedback raises the first-order activation fee threshold relative to zero feedback, while negative feedback lowers it.

The remainder of the paper is organized as follows. Section~\ref{sec:model} introduces the model and establishes well-posedness and stability. Section~\ref{sec:variation} develops the first-order variations and the physical-measure decomposition. Section~\ref{sec:smp} derives the adjoint equations and the necessary maximum principle. Section~\ref{sec:example} presents the quadratic example and the activation threshold comparison.

	\section{Model and well-posedness}\label{sec:model}
	\subsection{Preliminaries}
	Fix $T>0$. Under a reference probability $\QQ$, let $W$ be a
	$d_W$-dimensional Brownian motion, let $\bN$ be a Poisson random measure
	on a standard Borel mark space $E$ with compensator
	$\nu(\mathrm d\zeta)\mathrm dt$, and let $Y_1,\ldots,Y_n$ be Poisson
	processes with common intensity $c>0$. These noises are mutually
	independent under $\QQ$. Throughout the paper,
	$\FF=(\cF_t)_{0\leq t\leq T}$ is the usual $\QQ$-augmentation of their
	joint natural filtration, and we work on
	$(\Omega,\cF_T,\FF,\QQ)$. In particular, the adjoint equations use this
	same natural filtration. Let $\tilde N(\drm\zeta,\drm t) = \bar N(\drm\zeta,\drm t) - \nu(\drm\zeta)\drm t$ and $\tilde Y_k(t) = Y_k(t) - ct$ denote the corresponding compensated martingale noises, and let $\mathbb F^Y$ be the $\QQ$-augmented filtration generated by $Y=(Y_1,\cdots,Y_n)$. 
	
	Let $U\subset\R^m$ be nonempty, closed, and convex. The admissible controls
	are
	\[
	\Uad=\left\{u\in U: 
		u\text{ is }\FF^Y\text{-predictable}, 
		\displaystyle \EE\int_0^T|u_t|^r\dd t<\infty
		\text{ for every }r\geq1 \right\}.
	\]

We then define the following spaces. 
\begin{enumerate}[label=\arabic*.,leftmargin=*]
	\item For $1\leq j<\infty$, let $L^j(E,\nu;\R^q)$ be the space of
	$\mathcal B(E)$-measurable functions $h:E\to\R^q$, identified up
	to $\nu$-almost-everywhere equality, such that
	\[
	\|h\|_{L^j(E,\nu)}
	=\left(\int_E |h(\zeta)|^j\,\nu(\mathrm d\zeta)\right)^{1/j}
	<\infty.
	\]
We omit the target space when its dimension is clear and write  simply $L^j(\nu)$. For matrix-valued functions, $|\cdot|$ denotes the Frobenius norm.
	\item For $p\geq2$, let $\Sp^p$ be the space of $\FF$-adapted
	c\`adl\`ag processes $Z$, identified up to indistinguishability,
	such that $\EE\sup_{t\leq T}|Z_t|^p<\infty$. Its norm is given by
	\[
	\|Z\|_{\Sp^p}^p=\EE\sup_{t\leq T}|Z_t|^p.
	\]
	We also use $\Sp^\infty$ for $\FF$-adapted c\`adl\`ag processes
	with finite norm
	\[
	\|Z\|_{\Sp^\infty}
	=\left\|\sup_{t\leq T}|Z_t|\right\|_{L^\infty(\QQ)}.
	\]
	
	\item For $p\geq1$ and each fixed $t\in[0,T]$, we use the
	$\FF$-predictable subspace of
	$L^p([0,t]\times\Omega,\mathrm ds\otimes\mathrm d\QQ)$, with norm
	\[
	\|u\|_{p,t}^p=\EE\int_0^t|u_s|^p\dd s.
	\]
	
		\item For $r\geq1$, the $r$-Wasserstein distance on
	$\mathcal P_r(\R^d)$ is
	\[
	W_r(\mu,\widehat\mu)
	=\left(\inf_{\pi\in\Pi(\mu,\widehat\mu)}
	\int_{\R^d\times\R^d}|x-y|^r\pi(\mathrm dx,\mathrm dy)
	\right)^{1/r},
	\]
	where $\Pi(\mu,\widehat\mu)$ is the set of Borel probability
	measures on $\R^d\times\R^d$ with first marginal $\mu$ and second
	marginal $\widehat\mu$.
	
	\item For $p\geq1$, let $\mathcal P_p(\R^d)$ be the set of Borel
	probability measures on $\R^d$ with finite $p$th moment. For
	$\mu\in\mathcal P_p(\R^d)$, write
	\[
	m_p(\mu)=\int_{\R^d}|x|^p\mu(\mathrm dx).
	\]
	In particular,  for $\mu,\widehat\mu\in\mathcal P_1(\R^d)$, define
	\begin{equation}\label{eq:rho}
		\rho(\mu,\widehat\mu)
		=\sup_{\substack{\varphi:\R^d\to\R,\ |\varphi(0)|\leq1\\
				\Lip(\varphi)\leq1}}
		\left|\int_{\R^d}\varphi(x)(\mu-\widehat\mu)(\mathrm dx)\right|.
	\end{equation}
	By the Kantorovich--Rubinstein representation, $\rho=W_1$.

	\item For $p\geq1$, the space $C([0,T];\mathcal P_p(\R^d))$
	consists of $W_p$-continuous curves and is equipped with the uniform
	$W_p$ distance
	\[
	\sup_{t\leq T}W_p(\mu_t,\widehat\mu_t).
	\]
	
	\item For $I=[s_0,s_1]\subset[0,T]$, define $\Sp^2(I)$ by
	restricting the processes and the supremum in the definition of
	$\Sp^2$ to $I$. The spaces $\HW^2(I)$ and $\JY^2(I)$ consist of
	$\FF$-predictable processes, and $\JN^2(I)$ consists of
	$\Pred(\FF)\otimes\mathcal B(E)$-measurable kernels, with finite
	norms given by
	\begin{align*}
		\|Z\|_{\HW^2(I)}^2=\EE\int_I|Z_t|^2\dd t,\quad	\|R\|_{\JN^2(I)}^2=\EE\int_I\int_E|R_t(\zeta)|^2
		\nu(\mathrm d\zeta)\dd t,\quad	\|L\|_{\JY^2(I)}^2=\EE\int_I c\sum_{k=1}^n|L_{k,t}|^2\dd t.
	\end{align*}
 
	\item With the appropriate Euclidean dimensions, set
	\[
	\Bspace^2(I)
	=\Sp^2(I)\times\HW^2(I)\times\JN^2(I)\times\JY^2(I).
	\]
	The square of its norm is the sum of the squared component norms.
	We omit $I$ from these spaces when $I=[0,T]$.
\end{enumerate}
 
 	For $u\in\cU_{ad}$, we consider the following coupled system:
 \begin{equation}\label{eq:weightedlaw}
 	\left\{
 	\begin{aligned}
 		\drm X_t^u= &b(t,X_{t-}^u,\mu_t^u,u_t)\drm t+\sigma(t,X_{t-}^u,\mu_t^u,u_t)\drm W_t+\int_E\eta(t,X_{t-}^u,\mu_t^u,u_t,\zeta)\tilde{N}(\drm \zeta,\drm t),\\
 		\drm M_t^u= &M_{t-}^u\sum_{k=1}^n\left(\frac{\lambda_k(t,X_{t-}^u,\mu_t^u,u_t)}{c}-1\right)\drm \ti{Y}_k(t),\\
 		\langle \mu_t^u,\varphi\rangle= &\EE[M_t^u\varphi(X_t^u)],\quad
 		X_0^u=x_0\in\R^d,\quad M_0^u=1.
 	\end{aligned}
 	\right.
 \end{equation}
 The coefficients $b,\sigma,\eta$ take values respectively in $\R^d$, $\R^{d\times d_W}$, and $\R^d$. The intensities $\lambda_k$ are nonnegative.
 For brevity, write $\lambda_k^u(t)=\lambda_k(t,X_{t-}^u,\mu_t^u,u_t)$.
 
 Moreover, the cost functional reads
 \begin{equation}\label{eq:cost}
 	J(u)=\EE\left[M_T^u\Phi(X_T^u,\mu_T^u)+\int_0^T M_t^u f(t,X_{t-}^u,\mu_t^u,u_t)\drm t\right].
 \end{equation}
 The control problem consists of minimizing $J(u)$ over all admissible controls in $\cU_{ad}$.
 
 Now, we introduce the following assumptions.
	\begin{assumption}\label{ass:state}
		All coefficients are deterministic and jointly Borel measurable, where the measure argument is equipped with the
		corresponding Wasserstein Borel $\sigma$-field. For every finite $p\geq2$
		there is $C_p<\infty$ such that, for $z=(x,\mu,u)$ and
		$\widehat z=(\widehat x,\widehat\mu,\widehat u)$,
		\begin{align*}
			&|b(t,z)-b(t,\widehat z)|+|\sigma(t,z)-\sigma(t,\widehat z)|+\|\eta(t,z,\cdot)-\eta(t,\widehat z,\cdot)\|_{L^2(\nu)}
			+\|\eta(t,z,\cdot)-\eta(t,\widehat z,\cdot)\|_{L^p(\nu)}\\
			&\qquad\leq C_p\bigl(|x-\widehat x|+
			W_1(\mu,\widehat\mu)+|u-\widehat u|\bigr),\\
			&|b(t,z)|+|\sigma(t,z)|+\|\eta(t,z,\cdot)\|_{L^2(\nu)}
			+\|\eta(t,z,\cdot)\|_{L^p(\nu)}\leq C_p\bigl(1+|x|+m_1(\mu)+|u|\bigr).
		\end{align*}
		These bounds hold on $[0,T]\times\R^d\times\mathcal P_1(\R^d)\times U$.
	\end{assumption}
	
	\begin{assumption}\label{ass:intensity}
		The intensities are deterministic and jointly Borel measurable. There exist 
		$\lambda_*\in(0,\infty)$ and $L<\infty$ such that $
		0\leq\lambda_k(t,x,\mu,u)\leq\lambda_*,$ 
		and $
		|\lambda_k(t,x,\mu,u)-\lambda_k(t,\widehat x,\widehat\mu,\widehat u)|
		\leq L\bigl(|x-\widehat x|+W_1(\mu,\widehat\mu)+|u-\widehat u|\bigr).$ 
	\end{assumption}
	Throughout the paper, choose $c\geq\max\{1,\lambda_*\}$ and define
	$C_M=e^{ncT}$.  
	
\begin{assumption}\label{ass:regularity}
	The state coefficients and intensities are defined on
	$\mathcal P_1(\mathbb R^d)$, while
	$f:[0,T]\times\mathbb R^d\times\mathcal P_2(\mathbb R^d)\times U\to\mathbb R$ and
	$\Phi:\mathbb R^d\times\mathcal P_2(\mathbb R^d)\to\mathbb R$.
	All measure arguments appearing below belong to $\mathcal P_2(\mathbb R^d)$.
	
	Define $\mathcal C=\{b,\sigma,\lambda_1,\ldots,\lambda_n\}$, and $\mathcal D=\mathcal C\cup\{\eta,f,\Phi\}$.  Set $
	A_2(x,\mu,u)=1+|x|^2+m_2(\mu)+|u|^2$ and $ A_2^0(x,\mu)=1+|x|^2+m_2(\mu).$ 
	 
	\begin{enumerate}[label=(\roman*),leftmargin=*]
		\item Each $\varphi\in\mathcal D$ is continuously differentiable in $(x,u)$, with a $C^1$ extension in $u$ to an open neighbourhood of $U$. The terminal cost $\Phi$ is continuously differentiable with respect to $x$. There exists a linear functional derivative $\varphi_\mu$ such that, for all $\mu,\widehat\mu\in\mathcal P_2(\mathbb R^d)$,
		\begin{equation}\label{eq:flat}
			\varphi(t,x,\widehat\mu,u)-\varphi(t,x,\mu,u)
			=\int_0^1\int_{\mathbb R^d}
			\varphi_\mu(t,x,\mu_\theta,u;y)\,
			(\widehat\mu-\mu)(\mathrm dy)\,\mathrm d\theta,
		\end{equation}
		where $\mu_\theta=(1-\theta)\mu+\theta\widehat\mu$.
		We impose the normalisation $\varphi_\mu(t,x,\mu,u;0)=0$.
		The map $y\mapsto\varphi_\mu(t,x,\mu,u;y)$ is continuously differentiable, and we write $
		\varphi_{\mu,1}(t,x,\mu,u;y)
		=\partial_y\varphi_\mu(t,x,\mu,u;y).$
		
		\item For $\varphi\in\mathcal C$,
		\begin{equation}\label{eq:boundedderivatives}
			|\varphi_x|+|\varphi_u|
			+|\varphi_{\mu,1}(\cdot;y)|\leq C,
			\qquad
			|\varphi_\mu(\cdot;y)|\leq C(1+|y|).
		\end{equation}
		For every finite $p\geq2$ and $j\in\{2,p\}$,
		\begin{equation}\label{eq:etaderivatives}
				\|\eta_x\|_{L^j(\nu)}
				+\|\eta_u\|_{L^j(\nu)}
				+\sup_y\|\eta_{\mu,1}(\cdot;y)\|_{L^j(\nu)}
				\leq C_j,\qquad
				\|\eta_\mu(\cdot;y)\|_{L^j(\nu)}
				\leq C_j(1+|y|).
		\end{equation}
	Moreover,	the cost functions satisfy
		\begin{equation}\label{eq:costderivatives}
			\begin{aligned}
				|f|+|f_x|+|f_u|\leq &CA_2,&\qquad |f_\mu(\cdot;y)|+|f_{\mu,1}(\cdot;y)|
				\leq & CA_2(1+|y|^2),\\
				|\Phi|+|\Phi_x|
				\leq& CA_2^0,&\qquad |\Phi_\mu(\cdot;y)|+|\Phi_{\mu,1}(\cdot;y)|
				\leq & CA_2^0(1+|y|^2).
			\end{aligned}
		\end{equation}
		
		\item For every $R>0$, the derivatives in (i) are uniformly continuous in $(x,u,y,\mu)$ on the set
		\[
		\bigl\{(x,u,y,\mu):
		\mu\in\mathcal P_4(\mathbb R^d),\
		|x|+|u|+|y|+m_4(\mu)\leq R\bigr\}.
		\]
 
	\end{enumerate}
\end{assumption}
Here the measure variable is equipped with the \(W_2\) metric, and the remaining variables carry their standard Euclidean distances.
For \(\eta\), differentiability in \((x,u)\) and differentiability of \(\eta_\mu\) in its evaluation variable \(y\) in (i) refer to Fréchet differentiability of \(B_p\)-valued maps, where \(B_p=L^2(\nu)\cap L^p(\nu)\).

Our moment assumptions are chosen to yield estimates for every finite \(p\), to accommodate the polynomial-growth test classes used below, and to preserve admissibility when the dominating reference rate is changed.
We do not seek minimal moment conditions.
Furthermore, under the above assumptions, any argument absent from a coefficient is omitted for notational simplicity.

	\subsection{Well-posedness of the state equation}
	We first establish a deterministic likelihood bound. Its validity at zero
	intensity is the basic estimate used in the weighted-law construction.
	
	\begin{lemma}\label{lem:boundedlikelihood}
		For predictable intensities $0\leq\lambda_k\leq\lambda_*\leq c$, the
		solution of the likelihood equation in \eqref{eq:weightedlaw} is a
		nonnegative uniformly integrable $\QQ$-martingale, and
		\begin{equation}\label{eq:product}
			M_t=\exp\left\{\sum_{k=1}^n\int_0^t(c-\lambda_k(s))\dd s\right\}
			\prod_{k=1}^n\prod_{\substack{0<s\leq t\\\Delta Y_k(s)=1}}
			\frac{\lambda_k(s)}c.
		\end{equation}
		In particular,
		\begin{equation}\label{eq:pathbound}
			0\leq M_t\leq e^{nct}\leq C_M,
			\qquad \EE M_t=1.
		\end{equation}
	\end{lemma}
	
	\begin{proof}
		Between successive jumps of $Y$, we have
		\[
		\drm M_t
		=M_t\sum_{k=1}^n(c-\lambda_k(t))\,\drm t,
		\]
		while at a jump time $t$ of $Y_k$, $
		M_t=M_{t-}\frac{\lambda_k(t)}{c}.$ 
		
		Since the observation processes have finitely many jumps on $[0,T]$ and no simultaneous jumps, then by iterating, we can derive   \eqref{eq:product}. Moreover, since 	$0\le\lambda_k/c\le1$ and $0\le c-\lambda_k\le c$, then  \eqref{eq:pathbound} holds.

		Furthermore, $$\EE\int_0^T\sum_{k=1}^n cM_{t-}^2\left(\frac{\lambda_k(t)}{c}-1\right)^2	\,\drm t
		\le ncC_M^2T<\infty.$$	Thus the stochastic integrals of $M$ in \eqref{eq:weightedlaw} are square-integrable martingales, and $\EE[M_t]=M_0=1$.  
	\end{proof}
	
	We next record the estimates used in the forward and variational
	constructions. Define $\|a_t\|_j=\|a_t\|_{L^j(\nu)}$.
	
	\begin{lemma}\label{lem:maximal}
		For $p\geq2$ and a predictable kernel $a$,
		\begin{equation}\label{eq:poissonmaximal}
			\EE\sup_{s\leq t}\left|\int_0^s\int_E a_r(\zeta)
			\Nt(\mathrm d\zeta,\mathrm dr)\right|^p
			\leq C_p\left\{\EE\left(\int_0^t\|a_r\|_2^2\dd r\right)^{p/2}
			+\EE\int_0^t\|a_r\|_p^p\dd r\right\}.
		\end{equation}
		For a Brownian integrand $B$,
		\begin{equation}\label{eq:brownianmaximal}
			\EE\sup_{s\leq t}\left|\int_0^s B_r\dd W_r\right|^p
			\leq C_p\EE\left(\int_0^t|B_r|^2\dd r\right)^{p/2}.
		\end{equation}
		Therefore, if $ 
		Z_t=Z_0+\int_0^t b_r\dd r+\int_0^t B_r\dd W_r
		+\int_0^t\int_E a_r(\zeta)\Nt(\mathrm d\zeta,\mathrm dr)
		+\sum_{k=1}^n\int_0^t l_{k,r}\dd\Yt_k(r),$ 
		then
		\begin{equation}\label{eq:combinedmaximal}
			\EE\sup_{s\leq t}|Z_s|^p
			\leq C_{p,T,c,n}\left\{\EE|Z_0|^p+
			\EE\int_0^t\bigl(|b_r|^p+|B_r|^p+
			\|a_r\|_2^p+\|a_r\|_p^p+|l_r|^p\bigr)\dd r\right\}.
		\end{equation}
	 
	\end{lemma}
	\begin{proof}
	We first consider a bounded elementary integrand $a$ supported on a set of finite $\dd t\otimes\nu$ measure.
	Set $L_t=\int_0^t\int_Ea_s(\zeta)\Nt(d\zeta,ds)$, and $ L_t^*=\sup_{s\le t}|L_s|.$ 
	
	Since for $p\ge2$, by Taylor's formula, $0\le |x+y|^p-|x|^p-p|x|^{p-2}x\cdot y\le C_p\bigl(|x|^{p-2}|y|^2+|y|^p\bigr).$ 
	Then, by applying It\^o's formula to $|L_t|^p$ and taking expectations, we obtain
	\[
	\begin{aligned}
		\EE|L_t|^p&=\EE\int_0^t\int_E\bigl[|L_{s-}+a_s(\zeta)|^p-|L_{s-}|^p-p|L_{s-}|^{p-2}L_{s-}\cdot a_s(\zeta)\bigr]\nu(d\zeta)ds\\
		&\le C_p\EE\int_0^t\left(|L_{s-}|^{p-2}\|a_s\|_2^2+\|a_s\|_p^p	\right)ds.
	\end{aligned}
	\]
	Therefore, by Doob's inequality, $
		\EE(L_t^*)^p\le C_p\EE\left[	(L_t^*)^{p-2}	\int_0^t\|a_s\|_2^2ds	\right]	+C_p\EE\int_0^t\|a_s\|_p^pds.$
		
	For $p>2$, set $q = \frac{p}{p-2}$ and $q'=\frac{p}{2}$. Then,   by using Young's inequality, we have 
	\begin{align}
		(L_t^*)^{p-2}\int_0^t \|a_s\|^2_2\drm s \le
		\frac1q \varepsilon \cdot (L_t^*)^{p}+\frac1{q'} \varepsilon^{-q'/q} \left(\int_0^t \|a_s\|^2_2\drm s\right)^{p/2}.
	\end{align}
	Choosing $\varepsilon>0$ sufficiently small,	we obtain
	\[\EE(L_t^*)^p	\le C_p\left[	\EE\left(\int_0^t\|a_s\|_2^2ds\right)^{p/2}	+\EE\int_0^t\|a_s\|_p^pds	\right].
	\]
	Thus, \eqref{eq:poissonmaximal} holds for $p>2$. For $p=2$, it follows directly from Doob's inequality and It\^o's isometry. Approximating a predictable integrand with finite right-hand side by bounded elementary integrands proves \eqref{eq:poissonmaximal} in the stated generality.
	
	For the Brownian integral, by using the	Burkholder--Davis--Gundy inequality,
	\[\EE\sup_{s\le t}\left|\int_0^s B_r\,dW_r\right|^p\le C_p\EE\left[\left\langle\int_0^\cdot B_r\,dW_r\right\rangle_t^{p/2}	\right]=C_p\EE
	\left(\int_0^t|B_r|^2dr\right)^{p/2},
	\]
	which proves \eqref{eq:brownianmaximal}.
	
	Next, applying \eqref{eq:poissonmaximal} to the compensated counting processes, we obtain
	\[\EE\sup_{s\le t}	\left|\sum_{k=1}^n	\int_0^s l_{k,r}\,d\Yt_k(r)\right|^p\le C_p\bigg[	c^{p/2}\EE	\left(\int_0^t|l_r|^2dr\right)^{p/2}	+c\EE\int_0^t	\sum_{k=1}^n|l_{k,r}|^pdr	\bigg]\le C_{p,T,c,n}\EE\int_0^t|l_r|^pdr.
	\]
	
	Finally, since for every nonnegative measurable $g$,	by using H\"older's inequality, we have  $\left(\int_0^t g_r\,dr\right)^p\le t^{p-1}\int_0^t g_r^p\,dr$, and $	\left(\int_0^t g_r^2\,dr\right)^{p/2}\le t^{p/2-1}\int_0^t g_r^p\,dr$. 	Combining these inequalities with \eqref{eq:poissonmaximal} and \eqref{eq:brownianmaximal} proves \eqref{eq:combinedmaximal}.
\end{proof}
	
The following estimate bounds the weighted marginal in terms of the state and likelihood. It is used both in the forward stability argument and in the proof of convergence for difference quotients.

	\begin{lemma}\label{lem:distance}
		Suppose $\mu_t=\EE[M_t\delta_{X_t}]$ and
		$\widehat\mu_t=\EE[\widehat M_t\delta_{\widehat X_t}]$ are probability
		measures, $0\leq M,\widehat M\leq C_M$, and
		$\EE\sup_{t\leq T}|X_t|^2\leq R$. Then
		\begin{equation}\label{eq:weighteddistance}
			\rho(\mu_t,\widehat\mu_t)
			\leq\|M_t-\widehat M_t\|_{L^2}\|1+|X_t|\|_{L^2}
			+C_M\|X_t-\widehat X_t\|_{L^2}.
		\end{equation}
		For every $p\geq2$,
		\begin{equation}\label{eq:weighteddistancep}
			\sup_{s\leq t}\rho(\mu_s,\widehat\mu_s)^p
			\leq C_{p,R}\left(\EE\sup_{s\leq t}|M_s-\widehat M_s|^p
			+\EE\sup_{s\leq t}|X_s-\widehat X_s|^p\right).
		\end{equation}
	\end{lemma}
	\begin{proof}
	For any test function $\varphi$ satisfying $|\varphi(0)|\le 1$ and $\text{Lip}(\varphi)\le 1$, we have $|\varphi(x)|\le 1+|x|$ and
	$|\varphi(x)-\varphi(y)|\le |x-y|$. Hence,
	\begin{align*}
		|\inner{\mu_t-\widehat\mu_t}{\varphi}|&=\left|\EE[(M_t-\widehat M_t)\varphi(X_t)]+\EE[\widehat M_t(\varphi(X_t)
		-\varphi(\widehat X_t))]\right|\\
		&\le \EE[|M_t-\widehat M_t|(1+|X_t|)]
		+\EE[\widehat M_t|X_t-\widehat X_t|]\\
		&\le \norm{M_t-\widehat M_t}_{L^2}
		\norm{1+|X_t|}_{L^2}
		+C_M\norm{X_t-\widehat X_t}_{L^2}.
	\end{align*}
	Thus, based on \eqref{eq:rho}, by taking the supremum over all such $\varphi$, we obtain \eqref{eq:weighteddistance}.
	
	We proceed to proving \eqref{eq:weighteddistancep}. Since $\EE\sup_{t\le T}|X_t|^2\le R$, $\sup_{s\le T}\norm{1+|X_s|}_{L^2}\le 1+R^{1/2}$. Moreover, since for any $s\le t$,   $|M_s-\widehat M_s|
	\le \sup_{0\le r\le t}|M_r-\widehat M_r|$ $\QQ$-a.s.. Then  by using H\"{o}lder's inequality, for any $p\ge 2$,
	\[
		\|M_s-\widehat M_s\|_{L^2}=\left(\EE|M_s-\widehat M_s|^2\right)^{1/2}\le\left(\EE\sup_{r\le t}|M_r-\widehat M_r|^2\right)^{1/2}\le\left(\EE\sup_{r\le t}|M_r-\widehat M_r|^p\right)^{1/p}.
	\]
	Similarly,  $\norm{X_s-\widehat X_s}_{L^2}\le\left(\EE\sup_{r\le t}|X_r-\widehat X_r|^p\right)^{1/p}$. 
	
	Therefore, 
	\begin{align*}
		\sup_{s\le t}\rho(\mu_s,\widehat\mu_s)\le (1+R^{1/2})\left(\EE\sup_{s\le t}|M_s-\widehat M_s|^p\right)^{1/p}+C_M
		\left(\EE\sup_{s\le t}|X_s-\widehat X_s|^p\right)^{1/p}.
	\end{align*}
	Raising both sides to the power $p$ and applying the estimate $(a+b)^p\le 2^{p-1}(a^p+b^p)$ yields \eqref{eq:weighteddistancep}.
\end{proof}
		Recall that $m_p(\mu)=\int |x|^p \mu(\drm x)$.
	\begin{lemma}\label{lem:wasserstein}
		For $1<r<q<\infty$ and $\mu,\widehat\mu\in\mathcal P_q(\R^d)$,
		\begin{equation}\label{eq:wasserstein}
			W_r(\mu,\widehat\mu)^r
			\leq\rho(\mu,\widehat\mu)^{(q-r)/(q-1)}
			\bigl[2^{q-1}\{m_q(\mu)+m_q(\widehat\mu)\}\bigr]^{(r-1)/(q-1)}.
		\end{equation}
		Thus, if the flows $\mu_t^j,\mu_t$ have uniformly bounded $q$th moments
		and $\sup_{t\leq T}\rho(\mu_t^j,\mu_t)\to0$, then
		$\sup_{t\leq T}W_r(\mu_t^j,\mu_t)\to0$ for every $1\leq r<q$.
	\end{lemma}
	\begin{proof}
	Let $\pi\in\Pi(\mu,\widehat\mu)$ be an optimal coupling for $W_1(\mu,\widehat\mu)= \rho(\mu,\widehat\mu)$. Define  $
	\alpha=\frac{q-r}{q-1}$, and $ \beta=\frac{r-1}{q-1}$.  Then $\alpha+\beta=1$ and $r=\alpha+q\beta$. By H\"older's inequality,
	\begin{align*}
		W_r(\mu,\widehat\mu)^r	 \le\int_{\R^d\times\R^d}
		|x-y|^r\,\pi(\drm x,\drm y) \le	\left(\int_{\R^d\times\R^d}	|x-y|\,\pi(\drm x,\drm y)\right)^\alpha\cdot	\left(\int_{\R^d\times\R^d}	|x-y|^q\,\pi(\drm x,\drm y)\right)^\beta.
	\end{align*}
	Moreover, by using	$|x-y|^q\le 2^{q-1}(|x|^q+|y|^q)$,	we have
	\[ \int_{\R^d\times\R^d}	|x-y|^q\,\pi(\drm x,\drm y)\le	2^{q-1}\int_{\R^d\times\R^d}	(|x|^q+|y|^q)\,\pi(\drm x,\drm y)=
	2^{q-1}\bigl(m_q(\mu)+m_q(\widehat\mu)\bigr). 
	\]
	Since $\int_{\R^d\times\R^d}|x-y|\,\pi(\drm x,\drm y)=\rho(\mu,\widehat\mu)$, we obtain \eqref{eq:wasserstein}.
	
	The last claim follows from \eqref{eq:wasserstein} and the stated uniform $q$-th moment bound. For $r=1$, it is exactly the assumed convergence in $\rho=W_1$.
\end{proof}
	
	Fix $u\in\Uad$. For any$\vartheta\in C([0,T];\mathcal P_1(\R^d))$, let$(X^\vartheta,M^\vartheta)$ solve the first two equations of
	\eqref{eq:weightedlaw} with $\mu^u$ replaced by $\vartheta$. Define $\Gamma(\vartheta)$ by
	\[
	\langle\Gamma(\vartheta)_t,\phi\rangle
	=\EE[M_t^\vartheta\phi(X_t^\vartheta)], 
	\]
	with $\phi\in B_b(\R^d).$ The dependence on the fixed control $u$ is suppressed for notational simplicity.

	\begin{lemma}\label{lem:decoupled}
		Under Assumptions~\ref{ass:state}--\ref{ass:intensity}, fix $u\in\Uad$
		and $\vartheta\in C([0,T];\mathcal P_1(\R^d))$. The decoupled equations
		have a unique strong solution $(X^\vartheta,M^\vartheta)$, and, for every
		finite $p\geq2$,
		\begin{equation}\label{eq:decoupledmoment}
			\EE\sup_{s\leq t}|X_s^\vartheta|^p
			\leq C_p\left(1+|x_0|^p+\|u\|_{p,t}^p+
			\int_0^tm_1(\vartheta_s)^p\dd s\right),
		\quad \text{ and }0\leq M_t^\vartheta\leq C_M.
		\end{equation}
		If $\vartheta,\widehat\vartheta$ have uniformly bounded second moments,
		then, for the same control and every $p\geq2$,
		\begin{equation}\label{eq:decoupleddifference} 
			\EE\sup_{s\leq t}|X_s^\vartheta-X_s^{\widehat\vartheta}|^p
				+\EE\sup_{s\leq t}|M_s^\vartheta-M_s^{\widehat\vartheta}|^p
				+\sup_{s\leq t}\rho(\Gamma(\vartheta)_s,
				\Gamma(\widehat\vartheta)_s)^p\leq C_{p,R}\int_0^t
				\sup_{r\leq s}\rho(\vartheta_r,\widehat\vartheta_r)^p\dd s. 
		\end{equation}
		where $R$ includes bounds on the input second moments and the control's second moment. 
	
	Furthermore, $\Gamma(\vartheta)$ belongs to $C([0,T];\mathcal P_1(\R^d))$.
	\end{lemma}
	\begin{proof}
	For a given $\vartheta$, let $X^{(0)}_t=x_0$. For any $j\in \NN$, define $X^{(j+1)}$ by
	\begin{align*}
		X_t^{(j+1)}=x_0+
		\int_0^tb(s,X_{s-}^{(j)},\vartheta_s,u_s)\,\drm s+\int_0^t\sigma(s,X_{s-}^{(j)},\vartheta_s,u_s)\,\drm W_s+\int_0^t\int_E\eta(s,X_{s-}^{(j)},\vartheta_s,u_s,\zeta)
		\widetilde N(\drm\zeta,\drm s).
	\end{align*}
	
	Since $\vartheta\in C([0,T];\mathcal P_1(\R^d))$, we have $\sup_{t\le T}m_1(\vartheta_t)<\infty.$  By Assumption~\ref{ass:state} and Lemma~\ref{lem:maximal}, for every $p\ge2$,
	\[ \EE\sup_{r\le t}|X_r^{(j+1)}|^p\le C_p\bigg[|x_0|^p+\int_0^t\EE\sup_{r\le s}|X_r^{(j)}|^p\,\drm s+\int_0^t
	\bigl(1+m_1(\vartheta_s)^p+\EE|u_s|^p\bigr)\drm s\bigg].
	\]
	Since $X^{(0)}=x_0$ and $u\in\cU_{ad}$, by induction, $X^{(j)}\in\mathcal S^p$ for every $j\ge0$. In particular,
	\[a_0(T):=\EE\sup_{t\le T}|X_t^{(1)}-X_t^{(0)}|^p<\infty.
	\]
	Moreover, based on Assumption \ref{ass:state} and Lemma~\ref{lem:maximal}, we derive that there exists a constant $C_p$ such that  $a_{j+1}(t)\le C_p\int_0^t a_j(s)\,\drm s,$ where $a_j(t)=\EE\sup_{s\le t}|X_s^{(j+1)}-X_s^{(j)}|^p.$ By induction, $
	a_j(T)\le\frac{(C_pT)^j}{j!}a_0(T)$. Therefore, 
	\[\sum_{j=0}^\infty\|X^{(j+1)}-X^{(j)}\|_{\mathcal S^p}\le a_0(T)^{1/p}\sum_{j=0}^\infty\left(\frac{(C_pT)^j}{j!}\right)^{1/p}<\infty.
	\]
	Thus $(X^{(j)})_{j\ge0}$ is a Cauchy sequence in $\mathcal S^p$ and converges to some $X^\vartheta\in\mathcal S^p$.
	Passing to the limit in the Picard equation shows that $X^\vartheta$ is a solution. Moreover, by using the same Lipschitz estimate and Gronwall's inequality, we obtain the uniqueness of this solution. 
	
	Furthermore, based on the growth condition and Lemma~\ref{lem:maximal}, we have 
	\[ \EE\sup_{s\le t}|X_s^\vartheta|^p\le C_p\bigg(|x_0|^p+\int_0^t
	\EE\sup_{r\le s}|X_r^\vartheta|^p\,\drm s+\int_0^t
	\bigl(1+m_1(\vartheta_s)^p
	+\EE|u_s|^p\bigr)\,\drm s
	\bigg).
	\]
	By	applying Gronwall's inequality, we have	\eqref{eq:decoupledmoment} holds for $X^\vartheta$. Based on Assumption \ref{ass:intensity}, combining with  Lemma~\ref{lem:boundedlikelihood},  $M^\vartheta$ satisfying $0\le M_t^\vartheta\le C_M$.
	
	Let $\vartheta,\widehat\vartheta$ be two probability flows and  let $	\Delta X=X^\vartheta-X^{\widehat\vartheta}$, $ 
	\Delta M=M^\vartheta-M^{\widehat\vartheta}$, and $d(t)=\rho(\vartheta_t,\widehat\vartheta_t).$ By the Lipschitz condition,
	\begin{equation*} 
		A(t):=\EE\sup_{s\le t}|\Delta X_s|^p\le C_p\int_0^t\bigl(A(s)+d(s)^p\bigr)\,\drm s.
	\end{equation*}
	Let  $\lambda_k^\vartheta(t)=\lambda_k(t,X_{t-}^\vartheta,\vartheta_t,u_t)$. Then, 
	\begin{equation}\label{eq:Mdifference} 
		\drm\Delta M_t=\sum_{k=1}^n\bigg[\left(\frac{\lambda_k^\vartheta(t)}{c}-1\right)\Delta M_{t-}+\frac{M_{t-}^{\widehat\vartheta}}{c}\left(\lambda_k^\vartheta(t)-\lambda_k^{\widehat\vartheta}(t)\right)\bigg]\drm\widetilde Y_k(t).
	\end{equation}
	Since $0\le\lambda_k^\vartheta\le c$ and $0\le M^{\widehat\vartheta}\le C_M$, based on  Assumption~\ref{ass:intensity}, we have 
	\[ B(t):=\EE\sup_{s\le t}|\Delta M_s|^p\le C_p\int_0^t\bigl(B(s)+A(s)+d(s)^p\bigr)\,\drm s.
	\]
	Adding the above two inequalities and applying Gronwall's inequality, we obtain $A(t)+B(t)\le C_p\int_0^t d(s)^p\,\drm s.$
	By \eqref{eq:decoupledmoment} with $p=2$, $\sup_{s\le T}\EE|X_s^\vartheta|^2\le C_R$.
	
	Thus, based on Lemma~\ref{lem:distance},
	\[ \sup_{s\le t}\rho\bigl(\Gamma(\vartheta)_s,\Gamma(\widehat\vartheta)_s\bigr)^p\le C_{p,R}\left(A(t)+B(t)\right)\le C_{p,R}
	\int_0^t\sup_{r\le s}\rho(\vartheta_r,\widehat\vartheta_r)^p\,\drm s,
	\]
	which proves \eqref{eq:decoupleddifference}.
	
	Finally, by using Lemma~\ref{lem:boundedlikelihood} , $\Gamma(\vartheta)_t(\R^d)=\EE M_t^\vartheta=1$, and $
	m_p(\Gamma(\vartheta)_t)\le C_M\EE|X_t^\vartheta|^p$. Then, for $0\le s<t\le T$,  
	\[  \EE|X_t^\vartheta-X_s^\vartheta|^2\le 3\bigg[ (t-s)\EE\int_s^t|b_r|^2\,\drm r+\EE\int_s^t|\sigma_r|^2\,\drm r
	+\EE\int_s^t\|\eta_r\|_{L^2(\nu)}^2 \,\drm r \bigg], 
	\]
	where the coefficients are evaluated at $(r,X_{r-}^\vartheta,\vartheta_r,u_r)$. Similarly,
	\[
	\EE|M_t^\vartheta-M_s^\vartheta|^2 \le ncC_M^2(t-s).
	\]
	Thus,  we have that both quantities converge to zero as $t-s\to0$. Hence, by Lemma~\ref{lem:distance},
	\[
	\rho\bigl(\Gamma(\vartheta)_t, 	\Gamma(\vartheta)_s\bigr)\to 0, \qquad\text{as }t-s\to0.
	\] 
	Since $\Gamma(\vartheta)_t$ is a probability measure with finite first moment, we conclude that $\Gamma(\vartheta)\in C([0,T];\mathcal P_1(\R^d))$.
\end{proof}

	Now, we introduce the following theorem to state the solution and the control stability required for convex variations.
	
	\begin{theorem}\label{thm:forward}
	Suppose that Assumptions~\ref{ass:state}--\ref{ass:intensity} hold. Then, for any $u\in\cU_{ad}$, the system \eqref{eq:weightedlaw} admits a unique strong solution
	$(X^u,M^u,\mu^u)\in\bigcap_{p\ge2}\mathcal S^p(\R^d)\times\mathcal S^\infty(\R)\times \bigcap_{p\ge1}C([0,T];\mathcal P_p(\R^d))$. 
	
	Moreover, for each finite $p\ge2$, there exists a constant $C_p>0$ such that for any $t\in[0,T]$,
	\begin{equation}\label{eq:forwardmoment}
		\EE\sup_{t\le T}|X_t^u|^p+\sup_{t\le T}m_p(\mu_t^u)\le C_p\left(1+|x_0|^p+\norm{u}_{p,T}^p	\right), \text{ and } 0\le M_t^u\le\CM.
	\end{equation}
	Furthermore, for any $p\ge2$ and $R>0$, there exists $C_{p,R}>0$ such that, for all
	$u,v\in\cU_{ad}$ with $\norm{u}_{p,T}+\norm{v}_{p,T}\le R$, for any $t\in[0,T]$,
	\begin{equation}\label{eq:stability}
		\EE\sup_{s\le t}|X_s^u-X_s^v|^p+\EE\sup_{s\le t}|M_s^u-M_s^v|^p+\sup_{s\le t}\rho(\mu_s^u,\mu_s^v)^p\le C_{p,R}\EE\int_0^t|u_s-v_s|^p\,\drm s.
	\end{equation}
\end{theorem}

\begin{proof}
	Fix $u\in\Uad$. Let $\mu_t^0=\delta_{x_0}$, $X_t^0=x_0$ and $M_t^0=1$. For $j\ge0$, define $\mu^{j+1}=\Gamma(\mu^j)$, where $(X^{j+1},M^{j+1})$ is the solution corresponding to the prescribed flow $\mu^j$ in Lemma~\ref{lem:decoupled}.
	
	For every $p\ge2$, let $A_j(t)=\EE\sup_{s\le t}|X_s^j|^p$. Since $\EE M_t^j=1$, applying Jensen's inequality yields
	\[ m_1(\mu_t^j)^p=\left(\EE[M_t^j|X_t^j|]\right)^p \le\CM A_j(t). 
	\]
	It follows from \eqref{eq:decoupledmoment} that $A_{j+1}(t)\le C_0+C_p\int_0^t A_j(s)\,\drm s$, where $C_0=C_p\left(1+|x_0|^p+\norm{u}_{p,T}^p\right)$. Taking $C_0$ sufficiently large so that $A_0(t)\le C_0$. Then, by using the induction method, we can derive that   $A_j(t)\le C_0e^{C_pt},\ j\ge0$. 
	
	Moreover, since $m_p(\mu_t^j)=\EE[M_t^j|X_t^j|^p]\le\CM A_j(t)$, then there exists a constant $C_p$ such that 
	\begin{equation}\label{eq:picarduniform}
		\sup_{j\ge0}\left(\EE\sup_{t\le T}|X_t^j|^p+\sup_{t\le T}m_p(\mu_t^j)\right)\le C_p\left(1+|x_0|^p+\norm{u}_{p,T}^p\right).
	\end{equation}
	For $j\ge0$, define
	\[D_j(t)=\EE\sup_{s\le t}|X_s^{j+1}-X_s^j|^p+\EE\sup_{s\le t}|M_s^{j+1}-M_s^j|^p+\sup_{s\le t}\rho(\mu_s^{j+1},\mu_s^j)^p.
	\]
	By \eqref{eq:picarduniform}, $D_0(T)<\infty$. Applying \eqref{eq:decoupleddifference} and combining with the uniform second-moment bound in \eqref{eq:picarduniform}, we obtain $D_j(t)\le C_p\int_0^t D_{j-1}(s)\,\drm s,\ j\ge1$.  Then, iterating this inequality gives $D_j(T)\le\frac{(C_pT)^j}{j!}D_0(T)$, and hence $\sum_{j=0}^{\infty}D_j(T)^{1/p}\le D_0(T)^{1/p}\sum_{j=0}^{\infty}\frac{(C_pT)^{j/p}}{(j!)^{1/p}}<\infty$. 
	
	Thus, $(X^j)$ and $(M^j)$ are Cauchy sequences in $\mathcal S^p$, and $(\mu^j)$ is Cauchy in $C([0,T];\mathcal P_1(\R^d))$ under the uniform $\rho$ metric. Therefore, $
	X^j\to X$ and $M^j\to M$ in $\mathcal S^p$, and $\sup_{t\le T}\rho(\mu_t^j,\mu_t)\to 0$.

	Since convergence in $\mathcal S^p$ implies convergence in $\mathcal S^2$,	the limits are consistent for all finite $p\ge2$. By Assumption~\ref{ass:state},
	\[ 
		\EE\int_0^T\bigl|b(t,X_{t-}^{j+1},\mu_t^j,u_t)-b(t,X_{t-},\mu_t,u_t)\bigr|^p\,\drm t\le C_p\left(T\EE\sup_{t\le T}|X_t^{j+1}-X_t|^p
		+\int_0^T\rho(\mu_t^j,\mu_t)^p\,\drm t\right)\to0. 
	\]
	The same convergence holds for $\sigma$ and $\eta$ in the corresponding norms of Lemma~\ref{lem:maximal}. Therefore,  the state equation passes to the limit in $\mathcal S^p$. Similarly, by Assumption~\ref{ass:intensity},
	\[ |\lambda_k(t,X_{t-}^{j+1},\mu_t^j,u_t)-\lambda_k(t,X_{t-},\mu_t,u_t)|\le L\left(|X_{t-}^{j+1}-X_{t-}|+\rho(\mu_t^j,\mu_t)\right).
	\]
	
	Let  $H_{k,t}^{j+1}=M_{t-}^{j+1}\left(\frac{\lambda_k(t,X_{t-}^{j+1},\mu_t^j,u_t)}{c}-1
	\right)$, and $H_{k,t}=M_{t-}\left(\frac{\lambda_k(t,X_{t-},\mu_t,u_t)}{c}-1
	\right)$. Then, by Assumption~\ref{ass:intensity} and $0\le M^{j+1},M\le\CM$, we have
	\[\EE\int_0^T|H_{k,t}^{j+1}-H_{k,t}|^p\,\drm t\le C_p\bigg(\EE\sup_{t\le T}|M_t^{j+1}-M_t|^p+\EE\sup_{t\le T}|X_t^{j+1}-X_t|^p+\sup_{t\le T}\rho(\mu_t^j,\mu_t)^p
	\bigg)\to 0.
	\]
	Therefore, combining with Lemma~\ref{lem:maximal}, we have  $\EE\sup_{t\le T}
	\left|\sum_{k=1}^n\int_0^t(H_{k,s}^{j+1}-H_{k,s})\,\drm\widetilde Y_k(s)	\right|^p\to 0$.
	
	Since $M_t^{j+1}=1+\sum_{k=1}^n\int_0^tH_{k,s}^{j+1}\,\drm\widetilde Y_k(s)$,
	we obtain
	\begin{align}
		&\EE\sup_{t\le T}\left|M_t-1-\sum_{k=1}^n\int_0^t H_{k,s}\,\drm\widetilde Y_k(s)	\right|^p\\
		&\le C_p\EE\sup_{t\le T}|M_t-M_t^{j+1}|^p+C_p\EE\sup_{t\le T}\left|\sum_{k=1}^n\int_0^t(H_{k,s}^{j+1}-H_{k,s})\,\drm\widetilde Y_k(s)\right|^p
		\to 0,\text{ as }j\to\infty.
	\end{align}
	Hence, $M_t=1+\sum_{k=1}^n\int_0^tM_{s-}\left(\frac{\lambda_k(s,X_{s-},\mu_s,u_s)}c-1
	\right)\drm\widetilde Y_k(s)$, which implies that $M$ satisfies the likelihood equation in \eqref{eq:weightedlaw}. 
	
	Furthermore, by using Lemma~\ref{lem:distance},
	\[ \rho\left(\mu_t^j,\EE[M_t\delta_{X_t}]\right)\le\norm{M_t^j-M_t}_{L^2}\norm{1+|X_t|}_{L^2}+\CM\norm{X_t^j-X_t}_{L^2}.
	\]
	The right-hand side converges to zero uniformly in $t$.  Since $\mu^j\to\mu$ uniformly in $\rho$, we obtain $\mu_t=\EE[M_t\delta_{X_t}],\ t\in[0,T]$. Thus $(X,M,\mu)$ solves \eqref{eq:weightedlaw}.

	By \eqref{eq:picarduniform} and the convergence in $\mathcal S^p$, $\EE\sup_{t\le T}|X_t|^p\le C_p\left(1+|x_0|^p+\norm{u}_{p,T}^p\right)$. Since $m_p(\mu_t)=\EE[M_t|X_t|^p]\le\CM\EE|X_t|^p$, we obtain \eqref{eq:forwardmoment}.

	The preceding convergence implies $\mu\in C([0,T];\mathcal P_1(\R^d))$. For any finite $r>1$, choose $q>r$. By \eqref{eq:forwardmoment}, $\sup_{t\le T}m_q(\mu_t)<\infty$. Then, by using Lemma~\ref{lem:wasserstein}, $W_r(\mu_t,\mu_s)^r\le C_{r,q}\rho(\mu_t,\mu_s)^{(q-r)/(q-1)}$, and thus $\lim_{|t-s|\to0}W_r(\mu_t,\mu_s)=0$. Hence $\mu\in C([0,T];\mathcal P_r(\R^d))$ for every finite $r\ge1$.

	We next establish \eqref{eq:stability}. Let $u,v\in\Uad$, and write $\Delta X=X^u-X^v$, and $\Delta M=M^u-M^v$. Set $F_p(t)=\EE\sup_{s\le t}|\Delta X_s|^p+\EE\sup_{s\le t}|\Delta M_s|^p$. By the Lipschitz assumptions and Lemma~\ref{lem:maximal},
	$	F_p(t)\le C_p\int_0^t\bigl(F_p(s)+\rho(\mu_s^u,\mu_s^v)^p+\EE|u_s-v_s|^p\bigr)\,\drm s.
	$
	
	If $\norm{u}_{p,T}+\norm{v}_{p,T}\le R$, then combining with \eqref{eq:forwardmoment} with $p=2$ and Lemma~\ref{lem:distance}, we have $\sup_{s\le t}\rho(\mu_s^u,\mu_s^v)^p\le C_{p,R}F_p(t)$. Thus,  $F_p(t)\le C_{p,R}\int_0^t\left(F_p(s)+\EE|u_s-v_s|^p\right)\,\drm s$. Thus, by using Gronwall's inequality, $F_p(t)\le C_{p,R}\EE\int_0^t|u_s-v_s|^p\,\drm s$. Combining this with the preceding estimate for $\rho$ proves \eqref{eq:stability}.

\end{proof}
	
	\subsection{Physical measure and filtering}
For $u\in\mathcal U_{\text{ad}}$, define the probability measure $\mathbb P^u$ on $(\Omega,\mathcal F_T)$ via the density $M_T^u$. Lemma~\ref{lem:boundedlikelihood} ensures that $\mathbb P^u$ is a probability measure satisfying $\mathbb P^u\ll\mathbb Q$. Its density process is given by
\begin{equation}\label{eq:physical}
	\left.\frac{\mathrm d\mathbb P^u}{\mathrm d\mathbb Q}\right|_{\mathcal F_t}=M_t^u,
	\qquad t\in[0,T].
\end{equation}
We call $\mathbb P^u$ the physical measure associated with the control. We do not assume equivalence between $\mathbb P^u$ and $\mathbb Q$.

The martingale property yields a probability measure for each prescribed intensity process. The following lemma characterizes the observation and state noise characteristics under this measure.

	\begin{lemma}\label{lem:girsanov}
		Let $M$ be the likelihood of Lemma~\ref{lem:boundedlikelihood} and let
		$\PP^u$ be defined by \eqref{eq:physical}. Under $\PP^u$, $Y_k$ has compensator
		$\int_0^t\lambda_k(s)\dd s$, $W$ is an $\FF$-Brownian motion, and $\bN$
		retains compensator $\nu(\mathrm d\zeta)\mathrm dt$.
	\end{lemma}
	\begin{proof}
	Let $H\ge0$ be any bounded predictable process. From Lemma \ref{lem:boundedlikelihood}, we have that $M$ is a bounded martingale. Then, by using the optional sampling theorem, for any stopping time $\tau_j\le T$,  we have $\EE[M_T\mid \cF_{\tau_j}]=M_{\tau_j}$. 
	
	Moreover, note that since $M$ is a $\QQ$-martingale,   for any nonnegative predictable process $G$, 
	\begin{align}
		\EE^{\PP}\int_0^T G_t\,\drm t=\EE^{\QQ}\int_0^T\EE^{\QQ}[M_T\mid\mathcal F_t]G_t\,\drm t=\EE^{\QQ}\int_0^T M_{t-}G_t\,\drm t.
	\end{align}
	Moreover, since $\Delta M_t=M_{t-}(\lambda_k(t)/c-1)\Delta Y_k(t)$, and $Y$ admits finite jumps, then  combining with the above equation,  
	\begin{align*}
		\EE^{\PP}\int_0^T H_t\,\drm Y_k(t)
		&=\EE^{\QQ}\int_0^T H_tM_t\,\drm Y_k(t)\\
		&=\EE^{\QQ}\int_0^T
		H_tM_{t-}\frac{\lambda_k(t)}{c}\,\drm Y_k(t)\\
		&=\EE^{\QQ}\int_0^T
		H_tM_{t-}\lambda_k(t)\,\drm t\\
		&=\EE^{\PP}\int_0^T
		H_t\lambda_k(t)\,\drm t.
	\end{align*}
	Hence, the $\PP$-compensator of $Y_k$ is $\int_0^t\lambda_k(s)\drm s$.

	Similarly, for any bounded nonnegative predictable $H(t,\zeta)$ supported on a set of finite $\nu$-measure, 
	\begin{align*}
		&\EE^{\PP}\int_0^T\int_E
		H(t,\zeta)\bar N(\drm\zeta,\drm t)\\
		&=\EE^{\QQ}\int_0^T\int_E
		M_tH(t,\zeta)\bar N(\drm\zeta,\drm t)\\
		&=\EE^{\QQ}\int_0^T\int_E
		M_{t-}H(t,\zeta)\bar N(\drm\zeta,\drm t)\\
		&=\EE^{\QQ}\int_0^T\int_E
		M_{t-}H(t,\zeta)\nu(\drm\zeta)\drm t\\
		&=\EE^{\PP}\int_0^T\int_E
		H(t,\zeta)\nu(\drm\zeta)\drm t.
	\end{align*}

	Finally, since $M$ is a purely discontinuous $\QQ$-martingale, $[M,W^i]=0$ for $i=1,\ldots,d_W$. By It\^o's formula,
	\[M_tW_t^i=\int_0^t M_{s-}\,\drm W_s^i+\int_0^t W_s^i\,\drm M_s.
	\]
	By \eqref{eq:pathbound}, $\EE^{\QQ}\int_0^T M_{s-}^2\,\drm s\le C_M^2T$, and $
	\EE^{\QQ}\int_0^T |W_s^i|^2\,\drm\langle M\rangle_s\le ncC_M^2\int_0^T\EE^{\QQ}|W_s^i|^2\,\drm s<\infty.$
	
	Hence, we can derive that $MW^i$ is a square-integrable $\QQ$-martingale. 
	
	Moreover, for  $A\in\mathcal F_s$ and $0\le s\le t\le T$,
	\[ 
	\EE^{\PP}[\mathds{1}_A W_t^i]=\EE^{\QQ}[\mathds{1}_A M_tW_t^i]=\EE^{\QQ}[\mathds{1}_A M_sW_s^i]=\EE^{\PP}[\mathds{1}_A W_s^i].
	\]
	Thus, $W$ is a continuous $\PP$-martingale. Since $\PP\ll\QQ$ and for $1\le i,j\le d_W$, $[W^i,W^j]_t=\delta_{ij}t$, by using the L\'evy Characterization Theorem, we have  $W$ is an $\mathbb F$-Brownian motion under $\PP$.
	
\end{proof}
We may now identify the weighted marginal with the state distribution under the physical measure and verify finiteness of the objective. This finishes the formulation of the control problem prior to differentiation.

	\begin{theorem}\label{thm:physicalrepresentation}
		Under Assumptions~\ref{ass:state}--\ref{ass:intensity},
		$\mu_t^u=\Law_{\PP^u}(X_t^u)$ for every $u\in\Uad$ and $t\leq T$.
		If Assumption~\ref{ass:regularity} also holds, then $J(u)$ is finite and
		\begin{equation}\label{eq:physicalcost}
			J(u)=\EE^{\PP^u}\left[\Phi(X_T^u,\mu_T^u)
			+\int_0^T f(t,X_{t-}^u,\mu_t^u,u_t)\dd t\right].
		\end{equation}
	\end{theorem}
	\begin{proof}
	For any bounded Borel function $\varphi$, by the martingale property of $M^u$, we have
	\[ 
	\EE^{\PP^u}[\varphi(X_t^u)]=\EE^{\QQ}[M_T^u\varphi(X_t^u)]=\EE^{\QQ}[M_t^u\varphi(X_t^u)]=\langle\mu_t^u,\varphi\rangle. 
	\]
	Thus $\mu_t^u=\Law_{\PP^u}(X_t^u)$.
	
	Under Assumption~\ref{ass:regularity}, the quadratic growth conditions and \eqref{eq:forwardmoment} yield
	\[
	\begin{aligned}
		&\EE\left[M_T^u|\Phi(X_T^u,\mu_T^u)|+\int_0^T M_t^u	|f(t,X_{t-}^u,\mu_t^u,u_t)|\,\drm t	\right]\\
		&\le C\bigg[1+2m_2(\mu_T^u)+\int_0^T(1+2m_2(\mu_t^u))\,\drm t+\CM\norm{u}_{2,T}^2
		\bigg]<\infty.
	\end{aligned}
	\]
	Since $\EE[M_T^u\mid\mathcal F_t]=M_t^u$, by using Fubini's theorem,
	\[
	\begin{aligned}
		&\EE^{\PP^u}\left[\int_0^T f(t,X_{t-}^u,\mu_t^u,u_t)\,\drm t\right]\\
		&=\int_0^T\EE^{\QQ}\left[M_T^u f(t,X_{t-}^u,\mu_t^u,u_t)\right]\drm t\\
		&=\int_0^T\EE^{\QQ}\left[M_t^u f(t,X_{t-}^u,\mu_t^u,u_t)
		\right]\drm t.
	\end{aligned}
	\]
	Moreover, $\EE^{\PP^u}[\Phi(X_T^u,\mu_T^u)]=\EE^{\QQ}[M_T^u\Phi(X_T^u,\mu_T^u)]$.
	Combining the above identities proves \eqref{eq:physicalcost}.
\end{proof}
	
For bounded Borel $g$, define the unnormalized conditional filter
\[
V_t^u(g)=\mathbb E\big[M_t^u g(X_t^u)\mid\mathcal F_t^Y\big].
\]
By the conditional Bayes formula,
\[
\pi_t^u(g)=\frac{V_t^u(g)}{V_t^u(1)}
=\mathbb E^{\mathbb P^u}\big[g(X_t^u)\mid\mathcal F_t^Y\big]
\quad \mathbb P^u\text{-a.s.},
\]
on the set $\{V_t^u(1)>0\}$. Since
\[
\mathbb P^u\big(V_t^u(1)=0\big)
=\mathbb E\big[M_t^u\mathbf 1_{\{V_t^u(1)=0\}}\big]
=\mathbb E\big[V_t^u(1)\mathbf 1_{\{V_t^u(1)=0\}}\big]=0,\] then this set has full $\mathbb P^u$-probability.
The weighted marginal satisfies $\langle\mu_t^u,g\rangle=\mathbb E\big[V_t^u(g)\big]$.
	
	\section{First-order variations}\label{sec:variation}
We establish differentiability of the state, likelihood and weighted marginal, derive the directional expansion of the cost functional, and characterize the first variation of the physical measure. Throughout this section, Assumptions~\ref{ass:state}--\ref{ass:regularity} are in force.

Fix $\bar u\in\mathcal U_{\text{ad}}$, with associated triple $(\bar X,\bar M,\bar\mu)$. For $v\in\mathcal U_{\text{ad}}$, set $h=v-\bar u$ and $u^\varepsilon=\bar u+\varepsilon h$ for $0\leq\varepsilon\leq1$. Convexity of $U$ together with the moment conditions on controls guarantees $u^\varepsilon\in\mathcal U_{\text{ad}}$.  We denote $(X^\varepsilon,M^\varepsilon,\mu^\varepsilon)$  as the corresponding solution triple.

All limits in this section are taken as $\varepsilon\downarrow0$, with $\bar u,v$ held fixed. An overbar on a coefficient indicates evaluation at $(t,\bar X_{t-},\bar\mu_t,\bar u_t)$. A prime marks an independent copy defined on $(\Omega',\mathcal F',\mathbb Q')$, and $\mathbb E'$ denotes expectation with respect only to this copy.
	
	\subsection{Variational equations and differentiability}
Based on	Theorem~\ref{thm:forward} , for every finite $p\geq2$, we have 
	\begin{equation}\label{eq:perturbation}
		\begin{aligned}
			\EE\sup_{t\leq T}|X_t^\varepsilon-\bX_t|^p
			+\EE\sup_{t\leq T}|M_t^\varepsilon-\bM_t|^p
			+\sup_{t\leq T}\rho(\mu_t^\varepsilon,\bmu_t)^p\leq C_p\varepsilon^p\EE\int_0^T|h_t|^p\dd t.
		\end{aligned}
	\end{equation}
	Here,  $C_p$ may depend on the moments of the two fixed controls.
	Define
	\[
	\chi_t^\varepsilon=\frac{X_t^\varepsilon-\bX_t}{\varepsilon},
	\qquad K_t^\varepsilon=\frac{M_t^\varepsilon-\bM_t}{\varepsilon},
	\qquad \mu_{1,t}^\varepsilon=\frac{\mu_t^\varepsilon-\bmu_t}{\varepsilon}.
	\]
	For $\varphi\in\{b,\sigma,\eta,\lambda_1,\ldots,\lambda_n,f\}$, let
	\begin{equation}\label{eq:Dmu}
		\Dmu\varphi(t;\xi,k)=\EE'\left[
		k_{t-}'\bar\varphi_\mu(t;\bX_{t-}')
		+\bM_{t-}'\bar\varphi_{\mu,1}(t;\bX_{t-}')\xi_{t-}'\right].
	\end{equation}
 Here, $\bar\varphi_\mu$ is evaluated at the \emph{unprimed} baseline coefficient arguments: only the evaluation variable and the variation get copied. For scalar coefficients, the last product is just the inner product with the gradient of the evaluation variable. For vector and matrix-valued coefficients, this is the Jacobian acting on $\xi'$. We use this simple contraction convention everywhere in the paper.
 
 Then we introduce the next lemma, which gives an identity that separates the likelihood quotient from the state displacement when we test the weighted marginal against a random kernel.
	
	\begin{lemma}
	\label{lem:kernel}
	Let $\Psi^\varepsilon(t,\omega;y)$ be a measurable random kernel, continuously differentiable in $y$, such that the following expressions are integrable. Then, for $\varepsilon>0$,
	\begin{equation}\label{eq:kernel} 
		\inner{\mu_{1,t}^\varepsilon}{\Psi^\varepsilon(t,\omega;\cdot)}=\EE'\left[K_t^{\varepsilon\prime}\Psi^\varepsilon(t,\omega;X_t^{\varepsilon\prime})\right]+\EE'\left[\bar M_t'\int_0^1\nabla_y\Psi^\varepsilon(t,\omega;\bar X_t'+a\varepsilon\chi_t^{\varepsilon\prime})\cdot\chi_t^{\varepsilon\prime}\,\drm a	\right].
	\end{equation}
	
\end{lemma}

\begin{proof}
	By the definition of $\mu_{1,t}^\varepsilon$, for fixed $(t,\omega)$,
	\begin{align}
		&\inner{\mu_{1,t}^\varepsilon}{\Psi^\varepsilon(t,\omega;\cdot)}\\
		&=\frac{1}{\varepsilon}\EE'\left[M_t^{\varepsilon\prime}\Psi^\varepsilon(t,\omega;X_t^{\varepsilon\prime})-\bar M_t'
		\Psi^\varepsilon(t,\omega;\bar X_t')\right]\\
		&=\EE'\left[K_t^{\varepsilon\prime}	\Psi^\varepsilon(t,\omega;X_t^{\varepsilon\prime})\right]+\EE'\left[
		\bar M_t'\frac{\Psi^\varepsilon(t,\omega;X_t^{\varepsilon\prime})-\Psi^\varepsilon(t,\omega;\bar X_t')}{\varepsilon}	\right].
	\end{align}
	Since $X_t^{\varepsilon\prime} =\bar X_t'+\varepsilon\chi_t^{\varepsilon\prime}$, then 
	\[   \frac{\Psi^\varepsilon(t,\omega;X_t^{\varepsilon\prime})-\Psi^\varepsilon(t,\omega;\bar X_t')}{\varepsilon}=\int_0^1\nabla_y\Psi^\varepsilon(t,\omega;\bar X_t'+a\varepsilon\chi_t^{\varepsilon\prime})
	\cdot\chi_t^{\varepsilon\prime}\,\drm a.
	\]
	Substituting this identity into the preceding
	equality yields \eqref{eq:kernel}.
\end{proof}

	\begin{remark}\label{rem:kernel}
		The same identity holds for vector-valued and $L^j(\nu)$-valued kernels,
		using the corresponding derivative contraction and Bochner integrals.
	\end{remark}
	
Now, we will introduce the following chain rule along probability mixtures. 		For $i=0,1$, let $x_i\in\R^d$, $\mu_i\in\mathcal P_4(\R^d)$,
	$u_i\in U$. Define  $\Delta x=x_1-x_0$, and $\Delta u=u_1-u_0$. Set
	\[
	x_\theta=x_0+\theta\Delta x,\quad
	u_\theta=u_0+\theta\Delta u,\quad
	\mu_\theta=(1-\theta)\mu_0+\theta\mu_1,\quad
	z_\theta=(x_\theta,\mu_\theta,u_\theta).
	\]
	
	\begin{lemma}\label{lem:flatchain}
		Under Assumption~\ref{ass:regularity}, for every coefficient $\varphi$,
		\begin{equation}\label{eq:flatchain}
				\varphi(t,z_1)-\varphi(t,z_0)
				= \int_0^1\bigl[\varphi_x(t,z_\theta)\Delta x
				+\varphi_u(t,z_\theta)\Delta u\bigr]\dd\theta+\int_0^1\int_{\R^d}\varphi_\mu(t,z_\theta;y)
				(\mu_1-\mu_0)(\mathrm dy)\dd\theta.
		\end{equation}
	 
	\end{lemma}
	\begin{proof}
		Fix $t$ and let $F(\theta)=\varphi(t,z_\theta)$.
		For $\theta,\theta+h\in[0,1]$, the fundamental theorem of calculus in
		$(x,u)$ and \eqref{eq:flat} give
		\begin{equation}\label{eq:quotient}
		\begin{aligned}
			\frac{F(\theta+h)-F(\theta)}h
			= &\int_0^1\bigl[
			\varphi_x(t,x_\theta+ah\Delta x,\mu_{\theta+h},u_\theta+ah\Delta u)\Delta x\\
			&\hspace{17mm}+
			\varphi_u(t,x_\theta+ah\Delta x,\mu_{\theta+h},u_\theta+ah\Delta u)\Delta u
			\bigr]\dd a\\
			&+\int_0^1\int_{\R^d}
			\varphi_\mu(t,x_\theta,\mu_{\theta+ah},u_\theta;y)
			(\mu_1-\mu_0)(\mathrm dy)\dd a.
		\end{aligned}
		\end{equation}
	To verify continuity of the measure path in $W_2$, let $h>0$ and choose an optimal coupling $\pi\in\Pi(\mu_0,\mu_1)$. We define
	\[
	\pi_{\theta,h}=(1-\theta-h)(\mathrm{Id},\mathrm{Id})_\#\mu_0
	+\theta(\mathrm{Id},\mathrm{Id})_\#\mu_1+h\pi,
	\]
	which has marginals $\mu_\theta$ and $\mu_{\theta+h}$. Hence
	\[
	W_2(\mu_\theta,\mu_{\theta+h})^2
	\le\int|x-y|^2\pi_{\theta,h}(\drm x,\drm y)
	=hW_2(\mu_0,\mu_1)^2.
	\]
	Interchanging endpoints gives the analogous bound for $h<0$. Moreover,
		\[
	m_4(\mu_\theta)=(1-\theta)m_4(\mu_0)+\theta m_4(\mu_1)
	\le\max\{m_4(\mu_0),m_4(\mu_1)\}.
	\]
		By Assumption~\ref{ass:regularity}(iii), the first integral in
	the difference quotient \eqref{eq:quotient} converges to
	$\phi_x(t,z_\theta)\Delta x+\phi_u(t,z_\theta)\Delta u$.
	For the measure term, the same Wasserstein estimate gives
	\[
	\sup_{a\in[0,1]}W_2(\mu_{\theta+ah},\mu_\theta)
	\le |h|^{1/2}W_2(\mu_0,\mu_1)\longrightarrow0.
	\]
	Combining with the fixed segment, based on the growth bound in Assumption~\ref{ass:regularity}, we have  
	$|\phi_\mu(t,x_\theta,\mu_{\theta+ah},u_\theta;y)|
	\le C(1+|y|^2)$. Since
	\[
	\int_{\R^d}(1+|y|^2)|\mu_1-\mu_0|(\drm y)
	\le 2+m_2(\mu_0)+m_2(\mu_1)<\infty,
	\]
we apply the dominated convergence theorem to the measure integral in \eqref{eq:quotient}. Therefore,

	\[
	F'(\theta)=\phi_x(t,z_\theta)\Delta x+
	\phi_u(t,z_\theta)\Delta u+
	\int_{\R^d}\phi_\mu(t,z_\theta;y)
	(\mu_1-\mu_0)(\drm y).
	\]
Also based on Assumption \ref{ass:regularity}, we derive that  $F'$ is continuous in
	$\theta$, with one-sided derivatives at the endpoints.
Thus,	 integrating $F'$ over $[0,1]$ proves \eqref{eq:flatchain}.
 
	\end{proof}
	
\begin{remark}\label{rem:banachchain}
	For the jump coefficient $\eta$, we work in the Banach space $
	B_p=L^2(\nu)\cap L^p(\nu),$ with norm  $
	\|g\|_{B_p}=\|g\|_{L^2(\nu)}+\|g\|_{L^p(\nu)}. $
	The derivatives with respect to $x$ and $u$ are continuous
	in this norm, so the corresponding terms follow from the
	fundamental theorem of calculus.
	
	For the measure term, the difference of the measure derivatives
	in the preceding proof converges to zero in $B_p$ by the
	Wasserstein convergence established above. Its $B_p$-norm is
	bounded by $2C_p(1+|y|)$, which is integrable with respect to
	$|\mu_1-\mu_0|$. Therefore,  using dominated convergence  theorem to
	the Bochner integral,  \eqref{eq:flatchain} holds in $B_p$.
\end{remark}

	For $\varphi=b,\sigma,\eta,\lambda_k$, let
	\[
	D\varphi_t=\bar\varphi_x(t)\chi_{t-}
	+\Dmu\varphi(t;\chi,K)+\bar\varphi_u(t)h_t.
	\]
	The variational equations are
	\begin{equation}\label{eq:chi}
		\left\{
		\begin{aligned}
			\mathrm d\chi_t&=Db_t\dd t+D\sigma_t\dd W_t
			+\int_E D\eta_t(\zeta)\Nt(\mathrm d\zeta,\mathrm dt),\\
			\chi_0&=0,
		\end{aligned}\right.
	\end{equation}
	and
	\begin{equation}\label{eq:K}
		\left\{
		\begin{aligned}
			\mathrm dK_t&=\sum_{k=1}^n\left[
			\left(\frac{\bar\lambda_k(t)}c-1\right)K_{t-}
			+\frac{\bM_{t-}}cD\lambda_k(t)\right]\dd\Yt_k(t),\\
			K_0&=0.
		\end{aligned}\right.
	\end{equation}
		Then, we present the following Proposition to give  a unique solution $(\chi,K)$ with every finite $\mathcal S^p$ moment and $\EE K_t=0$. 
	\begin{proposition}\label{prop:variation}
		Equations~\eqref{eq:chi}--\eqref{eq:K} have a unique adapted solution in
		$\Sp^2\times\Sp^2$. For every finite $p\geq2$,
		\begin{equation}\label{eq:variationmoment}
			\EE\sup_{t\leq T}|\chi_t|^p+\EE\sup_{t\leq T}|K_t|^p
			\leq C_p\EE\int_0^T|h_t|^p\dd t.
		\end{equation}
		Moreover, $K$ is a martingale and $\EE K_t=0$.
	\end{proposition}

\begin{proof}
	Fix $p\ge2$ and let $\|Z\|_q=(\EE|Z|^q)^{1/q}$ for the random variables under $\QQ$. By Theorem~\ref{thm:forward},
	$B_*:=\sup_{t\le T}\|1+|\bar X_{t-}|\|_2<\infty$, and $0\le\bar M_t\le\CM$.

	For any $\varphi\in\{b,\sigma,\lambda_1,\ldots,\lambda_n\}$, based on Assumption~\ref{ass:regularity}, by using Cauchy--Schwarz
	inequality yield,  we have 
	\[  |\EE'[k_{t-}'\bar\varphi_\mu(t;\bar X_{t-}')]|\le C\EE'[|k_{t-}'|(1+|\bar X_{t-}'|)]\le CB_*\|k_{t-}\|_2,\] and \[
	|\EE'[\bar M_{t-}'\bar\varphi_{\mu,1}(t;\bar X_{t-}')\xi_{t-}']|\le C\CM\EE'|\xi_{t-}'|\le C\CM\|\xi_{t-}\|_2.\]
	Therefore,  combining with \eqref{eq:Dmu},
	\begin{equation}\label{eq:Dmuestimate}
		|D_\mu\varphi(t;\xi,k)|\le C\Bigl[\|1+|\bar X_{t-}|\|_2\|k_{t-}\|_2+\CM\|\xi_{t-}\|_2\Bigr]\le C\bigl(B_*\|k_{t-}\|_2+\CM\|\xi_{t-}\|_2\bigr).
	\end{equation}
	
	For the jump coefficient, by using Minkowski's inequality and combining with \eqref{eq:etaderivatives}, we have for $j\in\{2,p\}$,
	\begin{align}
		\|D_\mu\eta(t;\xi,k)\|_{L^j(\nu)}&\le\EE'\bigl[|k_{t-}'|\|\bar\eta_\mu(t;\bar X_{t-}',\cdot)\|_{L^j(\nu)}\bigr]+\EE'\bigl[\bar M_{t-}'|\xi_{t-}'|
		\|\bar\eta_{\mu,1}(t;\bar X_{t-}',\cdot)\|_{L^j(\nu)}\bigr]\\
		&\le C_j\bigl(B_*\|k_{t-}\|_2+\CM\|\xi_{t-}\|_2\bigr).
	\end{align}
	Since $\|Z\|_2\le\|Z\|_p$ for $p\ge2$, it follows that $\EE|D_\mu\varphi(t;\xi,k)|^p\le C_p\EE\bigl(|\xi_{t-}|^p+|k_{t-}|^p\bigr)$, and $
	\EE\left[\|D_\mu\eta(t;\xi,k)\|_{L^2(\nu)}^p+\|D_\mu\eta(t;\xi,k)\|_{L^p(\nu)}^p\right]\le C_p\EE\bigl(|\xi_{t-}|^p+|k_{t-}|^p\bigr)$.
	And based on the 	linearity of $D_\mu\varphi$,	the same bounds hold for differences of two pairs.

	For $\varphi\in\{b,\sigma,\eta,\lambda_1,\ldots,\lambda_n\}$, let $\mathcal L_\varphi(t;\xi,k)=\bar\varphi_x(t)\xi_{t-}+D_\mu\varphi(t;\xi,k)$, and $
	G_\varphi(t;\xi,k;h)=\mathcal L_\varphi(t;\xi,k)+\bar\varphi_u(t)h_t$. For each observation channel $r$, also let 
	\[Q_r(t;\xi,k;h)=	\left(\frac{\bar\lambda_r(t)}c-1\right)k_{t-}+\frac{\bar M_{t-}}cG_{\lambda_r}(t;\xi,k;h).
	\]
	Then, combining with the bounds on the derivatives, $0\le\bar\lambda_r\le c$, and $0\le\bar M\le\CM$, we have 	for $\varphi\in\{b,\sigma,\lambda_1,\ldots,\lambda_n\}$ and $j\in\{2,p\}$,
	\[
	\begin{aligned}
		|G_\varphi(t;\xi,k;h)|&\le C(|\xi_{t-}|+|h_t|)+C(B_*\|k_{t-}\|_2+\CM\|\xi_{t-}\|_2),\\
		\|G_\eta(t;\xi,k;h)\|_{L^j(\nu)}&\le C_j(|\xi_{t-}|+|h_t|)+C_j(B_*\|k_{t-}\|_2+\CM\|\xi_{t-}\|_2),\\
		|Q_r(t;\xi,k;h)|&\le |k_{t-}|+\frac{\CM}{c}|G_{\lambda_r}(t;\xi,k;h)|.
	\end{aligned}
	\]
	Using $\|Z\|_2\le\|Z\|_p$, we obtain
	\begin{align*}
		&\EE\biggl[|G_b(t;\xi,k;h)|^p+|G_\sigma(t;\xi,k;h)|^p
		+\|G_\eta(t;\xi,k;h)\|_{L^2(\nu)}^p\biggr]\\
		&\quad+\EE\biggl[\|G_\eta(t;\xi,k;h)\|_{L^p(\nu)}^p
		+\sum_{r=1}^n|Q_r(t;\xi,k;h)|^p\biggr]\\
		&\quad\le C_p\EE\bigl(|\xi_{t-}|^p+|k_{t-}|^p+|h_t|^p\bigr).
	\end{align*}
	
	Let $(\mathcal X,\mathcal K)$ be defined by
	\[ \mathcal X_t=\int_0^tG_b(s;\xi,k;h)\,\drm s
	+\int_0^tG_\sigma(s;\xi,k;h)\,\drm W_s+\int_0^t\int_EG_\eta(s;\xi,k;h,\zeta)\Nt(\drm\zeta,\drm s),\]
	and \[
	\mathcal K_t=\sum_{r=1}^n\int_0^tQ_r(s;\xi,k;h)\,\drm\Yt_r(s). 
	\]
	Since the baseline coefficients and the left limits of the input processes are predictable, and the
	independent-copy expectations preserve this property, the integrands are predictable. Thus, by  using Lemma~\ref{lem:maximal} and H\"older's inequality in time, we have that 
	\[ \EE\sup_{s\le t}|\mathcal X_s|^p\le C_p\EE\int_0^t\bigl[|G_b|^p+|G_\sigma|^p+\|G_\eta\|_{L^2(\nu)}^p+\|G_\eta\|_{L^p(\nu)}^p\bigr]\,\drm s,\]
	and \[\EE\sup_{s\le t}|\mathcal K_s|^p\le C_p\Bigg\{\EE\left(	\int_0^t c\sum_{r=1}^n|Q_r|^2\,\drm s	\right)^{p/2}
	+c\EE\int_0^t\sum_{r=1}^n|Q_r|^p\,\drm s\Bigg\}\le C_p\EE\int_0^t\sum_{r=1}^n|Q_r|^p\,\drm s.
	\]
	Consequently,
	\begin{equation}\label{eq:variationPicardbound}
		\EE\sup_{s\le t}\bigl(|\mathcal X_s|^p+|\mathcal K_s|^p\bigr)\le C_p\int_0^t
		\EE\bigl(|\xi_{s-}|^p+|k_{s-}|^p+|h_s|^p\bigr)\,\drm s.
	\end{equation}
	By linearity of $D_\mu\varphi$ and cancellation of the fixed
	source $h$, the same estimate holds for differences of two outputs,
	with $(\xi,k)$ replaced by $(\xi-\widehat\xi,k-\widehat k)$ and without $h$.
	
	Starting from $(\chi^0,K^0)=(0,0)$, define 
	\[
	(\chi^{j+1},K^{j+1})=(\mathcal X(\chi^j,K^j),\mathcal K(\chi^j,K^j)),
	\qquad j\in \NN.
	\]
	Since $h\in L^p(\drm t\otimes d\QQ)$, the preceding estimate
	gives
	\[
	\EE\sup_{t\le T}\bigl(|\chi_t^1|^p+|K_t^1|^p\bigr)
	\le C_p\EE\int_0^T|h_t|^p\,\drm t<\infty.
	\]
	Inductively, $(\chi^j,K^j)\in\mathcal S^p\times\mathcal S^p$
	for every $j$. Define
	\[
	D_j(t)=\EE\sup_{s\le t}\bigl(
	|\chi_s^{j+1}-\chi_s^j|^p
	+|K_s^{j+1}-K_s^j|^p\bigr).
	\]
	The difference estimate corresponding to
	\eqref{eq:variationPicardbound} gives that for any $j\in\NN$, $
	D_{j+1}(t)\le C_p\int_0^tD_j(s)\,\drm s.$
	
	Therefore, $D_j(T)\le D_0(T)(C_pT)^j/j!$, which implies
	\[
	\sum_{j=0}^\infty D_j(T)^{1/p}
	\le D_0(T)^{1/p}\sum_{j=0}^\infty
	\frac{(C_pT)^{j/p}}{(j!)^{1/p}}<\infty.
	\]
	By the triangle inequality, $(\chi^j,K^j)$ is Cauchy
	in $\mathcal S^p\times\mathcal S^p$.
	Denote its limit by $(\chi,K)$. Applying the difference estimate
	once more, we have 
	\[
	\begin{aligned}
		&\EE\sup_{t\le T}|\mathcal X_t(\chi^j,K^j)-\mathcal X_t(\chi,K)|^p+\EE\sup_{t\le T}|\mathcal K_t(\chi^j,K^j)-\mathcal K_t(\chi,K)|^p\\
		&\quad\le C_pT\bigl(\|\chi^j-\chi\|_{\mathcal S^p}^p
		+\|K^j-K\|_{\mathcal S^p}^p\bigr)\to0.
	\end{aligned}
	\]
	Passing to the limit in the iteration equations yields
	$(\chi,K)=(\mathcal X(\chi,K),\mathcal K(\chi,K))$,
	which is precisely \eqref{eq:chi}--\eqref{eq:K}.
	
	Applying \eqref{eq:variationPicardbound} to this fixed point and
	setting $F_p(t)=\EE\sup_{s\le t}\bigl(|\chi_s|^p+|K_s|^p\bigr),$  we obtain
	\[
	F_p(t)\le C_p\int_0^t F_p(s)\,\drm s
	+C_p\EE\int_0^t|h_s|^p\,\drm s.
	\]
	Then, by using Gronwall's inequality,
	\[
	F_p(T)\le C_pe^{C_pT}\EE\int_0^T|h_s|^p\,\drm s,
	\]
	which is \eqref{eq:variationmoment} after changing $C_p$.
	If $(\widehat\chi,\widehat K)$ is another square-integrable
	solution with the same $h$, the difference estimate at $p=2$ gives
	\[
	\EE\sup_{s\le t}\bigl(
	|\chi_s-\widehat\chi_s|^2
	+|K_s-\widehat K_s|^2\bigr)
	\le C\int_0^t\EE\sup_{r\le s}\bigl(
	|\chi_r-\widehat\chi_r|^2
	+|K_r-\widehat K_r|^2\bigr)\,\drm s.
	\]
	Then, by using Gronwall's inequality again, we obtain $\chi=\widehat\chi$ and $K=\widehat K$. In particular, the solutions constructed for different
	finite values of $p\ge2$ coincide.
	
	Finally, the estimate for $Q_r$ with $p=2$ implies
	\[ \EE\int_0^T c\sum_{r=1}^n
	\left|\left(\frac{\bar\lambda_r(s)}c-1\right)K_{s-}
	+\frac{\bar M_{s-}}cD\lambda_r(s)\right|^2\,\drm s\le C\EE\int_0^T
	\bigl(|\chi_{s-}|^2+|K_{s-}|^2+|h_s|^2\bigr)\,\drm s
	<\infty. 
	\]
	Therefore the stochastic integral in \eqref{eq:K} is a square-integrable $\QQ$-martingale. Since $K_0=0$,
	$\EE K_t=0$ for every $t\in[0,T]$.
\end{proof}
	
	Fix an arbitrary $\kappa\geq0$. For $j=1,2$, define
	\[
	\mathcal T_j^\kappa=
	\left\{\psi\in C^j(\R^d):
	\|\psi\|_{j,\kappa}:=
	\sup_{y\in\R^d}\frac{\sum_{a=0}^j|D^a\psi(y)|}{1+|y|^\kappa}<\infty
	\right\}.
	\]
	For $\kappa=0$, use $|y|^0=1$ also at $y=0$.
	Define
	\begin{equation}\label{eq:lawvariation}
		\langle\dot\mu_t,\psi\rangle
		=\EE\left[K_t\psi(\bX_t)+\bM_t\nabla\psi(\bX_t)\cdot\chi_t\right],
		\qquad\psi\in\mathcal T_1^\kappa.
	\end{equation}
The moment estimates make this a continuous linear functional on
$\mathcal T_1^\kappa$, and $\langle\dot\mu_t,1\rangle=0$.
Constants in estimates involving these spaces may depend on the fixed
$\kappa$. We write $\dot\mu_t$ for this functional. We do not assume the state marginal admits a signed-measure or total-variation derivative.

	Let $X_\theta^\varepsilon=\bX+\theta\varepsilon\chi^\varepsilon$,
	$\mu_\theta^\varepsilon=(1-\theta)\bmu+\theta\mu^\varepsilon$, and
	$u_\theta^\varepsilon=\bu+\theta\varepsilon h$.
	A superscript $(\varepsilon,\theta)$ on a coefficient derivative means
	evaluation at these arguments. Define
	\[
	\Delta_\varepsilon\varphi(t)
	=\frac{\varphi(t,X_{t-}^\varepsilon,\mu_t^\varepsilon,u_t^\varepsilon)
		-\bar\varphi(t)}\varepsilon.
	\]
	Based on Lemma~\ref{lem:flatchain}, applied with the two endpoint probability measures, we have 
	\begin{equation}\label{eq:chainquotient}
		\Delta_\varepsilon\varphi(t)=\int_0^1
		\left[\varphi_x^{\varepsilon,\theta}(t)\chi_{t-}^\varepsilon
		+\varphi_u^{\varepsilon,\theta}(t)h_t
		+\langle\mu_{1,t}^\varepsilon,
		\varphi_\mu^{\varepsilon,\theta}(t;\cdot)\rangle\right]
		\dd\theta.
	\end{equation}
	Define the baseline linear operator $\mathcal L_\varphi(t;\xi,k)=\bar\varphi_x(t)\xi_{t-}+D_\mu\varphi(t;\xi,k)$. Applying Lemma~\ref{lem:kernel} to \eqref{eq:chainquotient} yields
\begin{equation}\label{eq:exactremainder}
	\Delta_\varepsilon\varphi-D\varphi
	=\mathcal L_\varphi(t;\chi^\varepsilon-\chi,K^\varepsilon-K)
	+r_\varphi^\varepsilon,
\end{equation}
where the remainder is 
\begin{align}
	r_\varphi^\varepsilon=&\int_0^1(\varphi_x^{\varepsilon,\theta}-\bar\varphi_x)\chi^\varepsilon\,d\theta
	+\int_0^1(\varphi_u^{\varepsilon,\theta}-\bar\varphi_u)h\,d\theta+\int_0^1\EE'\left[K^{\varepsilon\prime}
	\bigl(\varphi_\mu^{\varepsilon,\theta}(X^{\varepsilon\prime})
	-\bar\varphi_\mu(\bar X')\bigr)\right]d\theta\notag\\
	&+\int_0^1\int_0^1\EE'\left[\bar M'
	\bigl(\varphi_{\mu,1}^{\varepsilon,\theta}
	(\bar X'+a\varepsilon\chi^{\varepsilon\prime})
	-\bar\varphi_{\mu,1}(\bar X')\bigr)
	\chi^{\varepsilon\prime}\right]da\,d\theta.
	\label{eq:remainderformula}
\end{align}
Identity \eqref{eq:exactremainder} decomposes the coefficient error into a bounded linear operator acting on the process error plus a remainder term. The subsequent estimate controls the size of this remainder via our perturbation bounds, and it does not rely on convergence of the difference quotients.
 
	\begin{proposition}\label{prop:remainder}
		For every finite $p\geq2$ and $\varphi=b,\sigma,\lambda_k$,
		\begin{equation}\label{eq:remainderconvergence}
			\EE\int_0^T|r_\varphi^\varepsilon(t)|^p\dd t\longrightarrow0.
		\end{equation}
		For $\eta$, the same convergence holds in $L^2(\nu)\cap L^p(\nu)$.
	\end{proposition}
	\begin{proof}
	By \eqref{eq:perturbation}, for every finite $q\ge2$,
	$ \EE\sup_{t\le T}|X_t^\varepsilon-\bar X_t|^q\to0$, and $ 
	\sup_{0<\varepsilon\le1}
	\left(\|\chi^\varepsilon\|_{\mathcal S^q}
	+\|K^\varepsilon\|_{\mathcal S^q}\right)<\infty$. 
	Moreover, the moment estimate in \eqref{eq:forwardmoment} gives
	$\sup_{0<\varepsilon\le1}\sup_{t\le T}m_4(\mu_t^\varepsilon)<\infty$.
	Combining \eqref{eq:perturbation} with Lemma~\ref{lem:wasserstein}, we obtain $
	\sup_{t\le T}W_2(\mu_t^\varepsilon,\bar\mu_t)\to0.$ 
	Since $\mu_t^{\varepsilon,\theta}=(1-\theta)\bar\mu_t+\theta\mu_t^\varepsilon$, $
	W_2(\mu_t^{\varepsilon,\theta},\bar\mu_t)^2
	\le\theta W_2(\mu_t^\varepsilon,\bar\mu_t)^2.$ 
	Therefore, $
	\sup_{\theta\in[0,1]}\sup_{t\le T}
	W_2(\mu_t^{\varepsilon,\theta},\bar\mu_t)\to0,$ and $ 
	\sup_{\varepsilon,\theta,t}m_4(\mu_t^{\varepsilon,\theta})<\infty.$ 
	
	Fix $p\ge2$ and consider the finite measure space $\mathfrak S=[0,T]\times\Omega\times\Omega'\times[0,1]^2$
	equipped with $\drm \mathfrak m=\drm t\,\drm \QQ\,\drm \QQ'\drm \theta\drm a.$ For notational simplicity, we define  $
		A_1^\varepsilon=\phi_x^{\varepsilon,\theta}-\bar\phi_x$, $A_2^\varepsilon=\phi_u^{\varepsilon,\theta}-\bar\phi_u$, $
		A_3^\varepsilon=\phi_\mu^{\varepsilon,\theta}(X^{\varepsilon\prime})-\bar\phi_\mu(\bar X')$, and $A_4^\varepsilon
		=\phi_{\mu,1}^{\varepsilon,\theta}(\bar X'+a\varepsilon\chi^{\varepsilon\prime})
		-\bar\phi_{\mu,1}(\bar X').$ 
		
	Then, from the preceding convergence and Assumption~\ref{ass:regularity}(iii), we have $A_i^\varepsilon\to0$ in $\mathfrak m$-measure for $i=1,2,3,4$. Moreover, by \eqref{eq:boundedderivatives}, $|A_1^\varepsilon|
	+|A_2^\varepsilon|+|A_4^\varepsilon|\le C$, and $|A_3^\varepsilon|\le C(1+|X^{\varepsilon\prime}|+|\bar X'|).$
	
	Together with \eqref{eq:forwardmoment}, these bounds give, for every finite $q\ge2$, $
	\sup_{0<\varepsilon\le1}\sum_{i=1}^4\|A_i^\varepsilon\|_{L^q(\mathfrak m)}<\infty.$ 
	
	Hence, by uniform integrability and Vitali's convergence theorem, for any $i\in\{1,2,3,4\}$,
	\[
	\|A_i^\varepsilon\|_{L^{2p}(\mathfrak m)} \to 0.
	\]
	
	On the other hand, by \eqref{eq:perturbation} and the moment bounds on the admissible controls,
	\[ \sup_{0<\varepsilon\le1}\big(\|\chi^\varepsilon\|_{L^{2p}(\mathfrak m)}+\|h\|_{L^{2p}(\mathfrak m)}+\|K^{\varepsilon\prime}\|_{L^{2p}(\mathfrak m)}+\|\bar M'\chi^{\varepsilon\prime}\|_{L^{2p}(\mathfrak m)}\big)<\infty.
	\]
	Applying H\"older's inequality, we obtain respectively
	\[
	\begin{aligned}
		\|A_1^\varepsilon\chi^\varepsilon\|_{L^p}
		\le&
		\|A_1^\varepsilon\|_{L^{2p}}
		\|\chi^\varepsilon\|_{L^{2p}}
		\to 0,&\qquad
		\|A_2^\varepsilon h\|_{L^p}
		\le&
		\|A_2^\varepsilon\|_{L^{2p}}
		\|h\|_{L^{2p}}
		\to 0,\\
		\|A_3^\varepsilon K^{\varepsilon\prime}\|_{L^p}
		\le&
		\|A_3^\varepsilon\|_{L^{2p}}
		\|K^{\varepsilon\prime}\|_{L^{2p}}
		\to 0,&\qquad
		\|A_4^\varepsilon
		\bar M'\chi^{\varepsilon\prime}\|_{L^p}
		\le&
		\|A_4^\varepsilon\|_{L^{2p}}
		\|\bar M'\chi^{\varepsilon\prime}\|_{L^{2p}}
		\to 0.
	\end{aligned}
	\] 
	
	Therefore, by \eqref{eq:remainderformula}, along with Jensen's inequality and Minkowski's inequality, we have 
	\[ \left(\EE\int_0^T|r_\phi^\varepsilon(t)|^p\drm t\right)^{1/p}\le\|A_1^\varepsilon\chi^\varepsilon\|_{L^p}+\|A_2^\varepsilon h\|_{L^p}+
	\|A_3^\varepsilon K^{\varepsilon\prime}\|_{L^p}+\|A_4^\varepsilon\bar M'\chi^{\varepsilon\prime}\|_{L^p}\to 0.
	\]
	This proves \eqref{eq:remainderconvergence}.
	
	For $\phi=\eta$, by \eqref{eq:etaderivatives} and
	Assumption~\ref{ass:regularity}(iii), the same estimates hold with the $L^2(\nu)\cap L^p(\nu)$ norm.
\end{proof}

This remainder estimate now closes the error equations for $(\chi^\varepsilon-\chi,K^\varepsilon-K)$. It gives both process differentiability and uniform convergence against the weighted test functions we use for the marginal variation.

	\begin{theorem}
		\label{thm:differentiability}
		For every finite $p\geq2$,
		\begin{equation}\label{eq:processdifferentiability}
			\|\chi^\varepsilon-\chi\|_{\Sp^p}
			+\|K^\varepsilon-K\|_{\Sp^p}\longrightarrow0.
		\end{equation}
		For every fixed $\kappa\geq0$,
		\begin{equation}\label{eq:lawdifferentiability}
			\sup_{t\leq T}\sup_{\|\psi\|_{2,\kappa}\leq1}
			|\langle\mu_{1,t}^\varepsilon-\dot\mu_t,\psi\rangle|
			\longrightarrow0.
		\end{equation}
	\end{theorem}

\begin{proof}
	Let  $\widehat\chi^\varepsilon=\chi^\varepsilon-\chi$, and $ \widehat K^\varepsilon=K^\varepsilon-K$. And define  $
	\mathcal L_\phi^\varepsilon(t)=\mathcal L_\phi(t;\widehat\chi^\varepsilon,\widehat K^\varepsilon).$
	
	By \eqref{eq:chi} and \eqref{eq:exactremainder}, we obtain that the dynamics of $\hat\chi^\ep$ is as follows 
	\[ \drm \widehat\chi_t^\varepsilon=\bigl(\mathcal L_b^\varepsilon(t)+r_b^\varepsilon(t)\bigr)\,\drm t+\bigl(\mathcal L_\sigma^\varepsilon(t)+r_\sigma^\varepsilon(t)\bigr)\,\drm W_t+\int_E\bigl(\mathcal L_\eta^\varepsilon(t,\zeta)+r_\eta^\varepsilon(t,\zeta)\bigr)\Nt(\drm\zeta,\drm t).
	\]
	
	For the likelihood process, we have
	\[\drm K_t^\varepsilon=\sum_{k=1}^n\bigg[\left(\frac{\bar\lambda_k(t)}c-1\right)
	K_{t-}^\varepsilon\\+\frac{\bar M_{t-}}c\Delta_\varepsilon\lambda_k(t)+\frac{K_{t-}^\varepsilon}{c}\bigl(\lambda_k^\varepsilon(t)
	-\bar\lambda_k(t)\bigr)\bigg]\drm\Yt_k(t).
	\]
	Subtracting \eqref{eq:K} and using \eqref{eq:exactremainder}, we obtain
	\[\drm \widehat K_t^\varepsilon=\sum_{k=1}^n\bigg[\left(\frac{\bar\lambda_k(t)}c-1\right)\widehat K_{t-}^\varepsilon+\frac{\bar M_{t-}}c
	\mathcal L_{\lambda_k}^\varepsilon(t)+g_k^\varepsilon(t)\bigg]\drm\Yt_k(t),
	\]
	where $g_k^\varepsilon(t)=\frac{\bar M_{t-}}cr_{\lambda_k}^\varepsilon(t)+\frac{K_{t-}^\varepsilon}{c}\bigl(\lambda_k^\varepsilon(t)
	-\bar\lambda_k(t)\bigr).$
	
	By \eqref{eq:remainderconvergence} and $0\le\bar M\le\CM$, $\EE\int_0^T\left|\frac{\bar M_{t-}}cr_{\lambda_k}^\varepsilon(t)\right|^p\drm t\to 0$. Moreover, combining with Assumption~\ref{ass:intensity} and \eqref{eq:perturbation}, we derive that 
	\[
	\begin{aligned}
		&\EE\int_0^T|\lambda_k^\varepsilon(t)-\bar\lambda_k(t)|^{2p}\,\drm t\\
		&\le C_p\bigg[T\EE\sup_{t\le T}|X_t^\varepsilon-\bar X_t|^{2p}+T\sup_{t\le T}\rho(\mu_t^\varepsilon,\bar\mu_t)^{2p}
		+\varepsilon^{2p}\EE\int_0^T|h_t|^{2p}\,\drm t\bigg]\\
		&\le C_p\varepsilon^{2p}.
	\end{aligned}
	\]
	Since $\sup_{0<\varepsilon\le1}	\|K^\varepsilon\|_{\mathcal S^{2p}}<\infty$, by applying H\"older's inequality,
	\[
	\begin{aligned}
		&\EE\int_0^T|K_{t-}^\varepsilon|^p
		|\lambda_k^\varepsilon(t)-\bar\lambda_k(t)|^p\,\drm t\\
		&\quad\le
		\left(\EE\int_0^T|K_{t-}^\varepsilon|^{2p}\,\drm t\right)^{1/2}
		\left(\EE\int_0^T|\lambda_k^\varepsilon(t)-\bar\lambda_k(t)|^{2p}\,\drm t\right)^{1/2}\\
		&\quad\le C_p\varepsilon^p\to0.
	\end{aligned}
	\]
	Consequently, $\EE\int_0^T|g_k^\varepsilon(t)|^p\drm t\to 0$.
	
	Define
	\[ R_p^\varepsilon=\EE\int_0^T\bigg[|r_b^\varepsilon(t)|^p+|r_\sigma^\varepsilon(t)|^p+\|r_\eta^\varepsilon(t)\|_{L^2(\nu)}^p
	+\|r_\eta^\varepsilon(t)\|_{L^p(\nu)}^p+\sum_{k=1}^n|g_k^\varepsilon(t)|^p\bigg]\drm t.
	\]
	It follows from \eqref{eq:remainderconvergence} and the preceding estimates that $R_p^\varepsilon\to 0$. 
	
	Let $E_p^\varepsilon(t)=\EE\sup_{s\le t}\bigl(	|\widehat\chi_s^\varepsilon|^p	+|\widehat K_s^\varepsilon|^p
	\bigr).$ Applying Lemma~\ref{lem:maximal} to the error equations and using \eqref{eq:Dmuestimate}, we obtain $
	E_p^\varepsilon(t)\le C_p\int_0^tE_p^\varepsilon(s)\,\drm s+C_pR_p^\varepsilon.$ Then by using Gronwall's inequality,
	$E_p^\varepsilon(T)\le C_pe^{C_pT}R_p^\varepsilon\to 0$. Therefore, $\|\chi^\varepsilon-\chi\|_{\mathcal S^p}+\|K^\varepsilon-K\|_{\mathcal S^p}\to 0$, which proves \eqref{eq:processdifferentiability}.
	
	We next prove \eqref{eq:lawdifferentiability}. Let $\psi\in\mathcal T_2^\kappa$. By \eqref{eq:kernel} and \eqref{eq:lawvariation},
	\begin{align}
		\inner{\mu_{1,t}^\varepsilon-\dot\mu_t}{\psi}=&\EE\left[(K_t^\varepsilon-K_t)\psi(\bar X_t)\right]+\EE\left[
		\bar M_t\nabla\psi(\bar X_t)\cdot(\chi_t^\varepsilon-\chi_t)	\right]+\EE\left[K_t^\varepsilon\bigl(\psi(X_t^\varepsilon)
		-\psi(\bar X_t)\bigr)\right]\\
		&+\EE\left[\bar M_t\int_0^1	\bigl(\nabla\psi(\bar X_t+a\varepsilon\chi_t^\varepsilon)
		-\nabla\psi(\bar X_t)\bigr)\cdot\chi_t^\varepsilon	\,\drm a\right].
	\end{align}
	
	By the definition of $\|\psi\|_{2,\kappa}$, $|\psi(x)|+|\nabla\psi(x)|+|D^2\psi(x)|\le \|\psi\|_{2,\kappa}(1+|x|^\kappa).$
	Hence, by the Cauchy--Schwarz inequality, $\left|\EE[(K_t^\varepsilon-K_t)\psi(\bar X_t)]\right|\le\|\psi\|_{2,\kappa}
	\|K^\varepsilon-K\|_{\mathcal S^2}	\|1+|\bar X_t|^\kappa\|_{L^2}$,  and $\left|\EE\left[\bar M_t\nabla\psi(\bar X_t)\cdot
	(\chi_t^\varepsilon-\chi_t)\right]\right|\le\CM\|\psi\|_{2,\kappa}\|\chi^\varepsilon-\chi\|_{\mathcal S^2}\|1+|\bar X_t|^\kappa\|_{L^2}.$
	
	Furthermore,  by directly calculating, $ \psi(X_t^\varepsilon)-\psi(\bar X_t)=\varepsilon\int_0^1\nabla\psi(\bar X_t+a\varepsilon\chi_t^\varepsilon)	\cdot\chi_t^\varepsilon\,\drm a$, and $\nabla\psi(\bar X_t+a\varepsilon\chi_t^\varepsilon)-\nabla\psi(\bar X_t)=\varepsilon\int_0^aD^2\psi(\bar X_t+r\varepsilon\chi_t^\varepsilon)
	\chi_t^\varepsilon\,\drm r.$ 
	
	Since $
	\bar X_t+r\varepsilon\chi_t^\varepsilon
	=(1-r)\bar X_t+rX_t^\varepsilon,$ 
	we have, for $r\in[0,1]$, $
	1+|\bar X_t+r\varepsilon\chi_t^\varepsilon|^\kappa
	\le C_\kappa
	\bigl(
	1+|\bar X_t|^\kappa
	+|X_t^\varepsilon|^\kappa
	\bigr).$
	It follows that $
	|\psi(X_t^\varepsilon)-\psi(\bar X_t)|\le C_\kappa\varepsilon
	\|\psi\|_{2,\kappa}
	|\chi_t^\varepsilon|
	\bigl(
	1+|\bar X_t|^\kappa
	+|X_t^\varepsilon|^\kappa
	\bigr),$
	and $ \bigl|
	\nabla\psi(\bar X_t
	+a\varepsilon\chi_t^\varepsilon)
	-\nabla\psi(\bar X_t)
	\bigr|\le C_\kappa\varepsilon
	\|\psi\|_{2,\kappa}
	|\chi_t^\varepsilon|
	\bigl(
	1+|\bar X_t|^\kappa
	+|X_t^\varepsilon|^\kappa
	\bigr).$
	
	Let $B_t^\varepsilon=1+|\bar X_t|^\kappa+|X_t^\varepsilon|^\kappa$.
	By \eqref{eq:forwardmoment}, \eqref{eq:perturbation}, and \eqref{eq:variationmoment}, $
	\sup_{\varepsilon,t}\EE|B_t^\varepsilon|^2<\infty,$ and $ 
	\sup_\varepsilon\left(
	\EE\sup_{t\le T}|K_t^\varepsilon|^4
	+\EE\sup_{t\le T}|\chi_t^\varepsilon|^4\right)<\infty.$ 
	
	Consequently, by using H\"older's inequality, $$
	\EE[B_t^\varepsilon|K_t^\varepsilon||\chi_t^\varepsilon|]
	\le(\EE|B_t^\varepsilon|^2)^{1/2}
	(\EE|K_t^\varepsilon|^4)^{1/4}
	(\EE|\chi_t^\varepsilon|^4)^{1/4}\le C,$$
	and $$
	\EE[B_t^\varepsilon|\chi_t^\varepsilon|^2]\le(\EE|B_t^\varepsilon|^2)^{1/2}
	(\EE|\chi_t^\varepsilon|^4)^{1/2}\le C.$$ 
	
	Therefore, $\left|
	\EE\left[
	K_t^\varepsilon
	(\psi(X_t^\varepsilon)-\psi(\bar X_t))
	\right]\right|  \le C\varepsilon\|\psi\|_{2,\kappa},$
	and $
	\bigg|
	\EE\bigg[
	\bar M_t\int_0^1
	\bigl(
	\nabla\psi(\bar X_t
	+a\varepsilon\chi_t^\varepsilon)
	-\nabla\psi(\bar X_t)
	\bigr)\cdot\chi_t^\varepsilon
	\,\drm a
	\bigg]\bigg|\le C\varepsilon\|\psi\|_{2,\kappa}.$ 
	
	Combining the preceding estimates
	and using
	$\sup_{t\le T}
	\|1+|\bar X_t|^\kappa\|_{L^2}<\infty$,
	we obtain
	\[ \sup_{t\le T}
	\sup_{\|\psi\|_{2,\kappa}\le1}
	\left|
	\inner{\mu_{1,t}^\varepsilon-\dot\mu_t}{\psi}
	\right| \le C\bigl(
	\|K^\varepsilon-K\|_{\mathcal S^2}
	+\|\chi^\varepsilon-\chi\|_{\mathcal S^2}
	+\varepsilon
	\bigr)\to 0.
	\]
	This proves \eqref{eq:lawdifferentiability}.
\end{proof}
	
\subsection{Directional cost expansion}
Let $\bar f_t=f(t,\bX_{t-},\bmu_t,\bu_t)$ and
	$\bar\Phi=\Phi(\bX_T,\bmu_T)$. Define
	\begin{align}\label{eq:Df}
	Df_t=\bar f_x(t)\cdot\chi_{t-}+\bar f_u(t)\cdot h_t
	+D_\mu f(t;\chi,K)
\end{align}
and 
\begin{align}\label{eq:DPhi}
	D\Phi_T=\Phi_x(\bar X_T,\bar\mu_T)\cdot\chi_T+\EE'\left[K_T'\Phi_\mu(\bar X_T,\bar\mu_T;\bar X_T')
	+\bar M_T'\Phi_{\mu,1}(\bar X_T,\bar\mu_T;\bar X_T')\cdot\chi_T'\right].
\end{align}
	We now differentiate the running and terminal costs. Together with the likelihood variation, these derivatives
	yield the first-order expansion of the objective.
	
	\begin{theorem}\label{thm:costderivative}
		Under Assumptions~\ref{ass:state}--\ref{ass:regularity},
		$J(\bu+\varepsilon h)=J(\bu)+\varepsilon J_1(h)+o(\varepsilon)$, where
		\begin{equation}\label{eq:J1}
			J_1(h)=\EE\left[K_T\bar\Phi+\bM_TD\Phi_T
			+\int_0^T(K_t\bar f_t+\bM_tDf_t)\dd t\right].
		\end{equation}
		In particular, if $\bu$ minimizes $J$ over $\Uad$, then
		$J_1(v-\bu)\geq0$ for every $v\in\Uad$.
	\end{theorem}

\begin{proof}
	Define  $f_t^\varepsilon=f(t,X_{t-}^\varepsilon,\mu_t^\varepsilon,u_t^\varepsilon)$, $\Phi_T^\varepsilon=\Phi(X_T^\varepsilon,\mu_T^\varepsilon)$, $\Delta_\varepsilon f_t=\frac{f_t^\varepsilon-\bar f_t}{\varepsilon}$, and $\Delta_\varepsilon\Phi_T=\frac{\Phi_T^\varepsilon-\bar\Phi}{\varepsilon}.$ 
	
	We first show that
	\begin{equation}\label{eq:costquotientconvergence} 
			\Delta_\varepsilon f\to Df
			\text{in }L^2(dt\otimes d\QQ),\qquad
			\Delta_\varepsilon\Phi_T\to D\Phi_T
		\text{in }L^2(\QQ). 
	\end{equation}
	Applying Lemmas~\ref{lem:flatchain} and \ref{lem:kernel} to the running cost, we obtain
	\[ \Delta_\varepsilon f_t
	=\int_0^1\bigg[f_x^{\varepsilon,\theta}(t)\cdot\chi_{t-}^\varepsilon
	+f_u^{\varepsilon,\theta}(t)\cdot h_t+\EE'\bigl[K_t^{\varepsilon\prime}
	f_\mu^{\varepsilon,\theta}(t;X_t^{\varepsilon\prime})\bigr]+\int_0^1\EE'\bigl[\bar M_t'
	f_{\mu,1}^{\varepsilon,\theta}
	(t;\bar X_t'+a\varepsilon\chi_t^{\varepsilon\prime})
	\cdot\chi_t^{\varepsilon\prime}\bigr]\,\drm a\bigg]\,\drm\theta.
	\]
	Therefore,
	\begin{equation}\label{eq:costquotientremainder}
		\begin{aligned}
			\Delta_\varepsilon f_t-Df_t
			={}&\bar f_x(t)\cdot(\chi_{t-}^\varepsilon-\chi_{t-})
			+D_\mu f(t;\chi^\varepsilon-\chi,K^\varepsilon-K)
			+r_f^\varepsilon(t),
		\end{aligned}
	\end{equation}
	where $r_f^\varepsilon(t)
	=\int_0^1\bigl[
	(f_x^{\varepsilon,\theta}-\bar f_x)\cdot\chi_{t-}^\varepsilon
	+(f_u^{\varepsilon,\theta}-\bar f_u)\cdot h_t
	\bigr]\,\drm\theta+\int_0^1\EE'\bigl[
	K_t^{\varepsilon\prime}
	\bigl(f_\mu^{\varepsilon,\theta}(t;X_t^{\varepsilon\prime})
	-\bar f_\mu(t;\bar X_t')\bigr)\bigr]\,\drm\theta+\int_0^1\int_0^1\EE'\bigl[\bar M_t'
	\bigl(f_{\mu,1}^{\varepsilon,\theta}
	(t;\bar X_t'+a\varepsilon\chi_t^{\varepsilon\prime})
	-\bar f_{\mu,1}(t;\bar X_t')\bigr)
	\cdot\chi_t^{\varepsilon\prime}\bigr]\,\drm a\,\drm\theta.$
	
	Since for 	the terminal term, 
	\begin{align}
		\Delta_\varepsilon\Phi_T-D\Phi_T
		=&\Phi_x(\bar X_T,\bar\mu_T)
		\cdot(\chi_T^\varepsilon-\chi_T)+\EE'\bigl[(K_T^{\varepsilon\prime}-K_T')
		\Phi_\mu(\bar X_T,\bar\mu_T;\bar X_T')\bigr]\\
		&+\EE'\bigl[\bar M_T'
		\Phi_{\mu,1}(\bar X_T,\bar\mu_T;\bar X_T')
		\cdot(\chi_T^{\varepsilon\prime}-\chi_T')\bigr]
		+r_\Phi^\varepsilon,
	\end{align}
	where $
	r_\Phi^\varepsilon
	=\int_0^1
	(\Phi_x^{\varepsilon,\theta}-\bar\Phi_x)
	\cdot\chi_T^\varepsilon\,\drm\theta+\int_0^1\EE'\bigl[K_T^{\varepsilon\prime}
	\bigl(\Phi_\mu^{\varepsilon,\theta}(X_T^{\varepsilon\prime})
	-\bar\Phi_\mu(\bar X_T')\bigr)\bigr]\,\drm\theta+\int_0^1\int_0^1\EE'\bigl[\bar M_T'
	\bigl(\Phi_{\mu,1}^{\varepsilon,\theta}
	(\bar X_T'+a\varepsilon\chi_T^{\varepsilon\prime})
	-\bar\Phi_{\mu,1}(\bar X_T')\bigr)
	\cdot\chi_T^{\varepsilon\prime}\bigr]\,\drm a\,\drm\theta.$
	
	We prove that
	\begin{equation}\label{eq:costremainderlimit}
		\|r_f^\varepsilon\|_{L^2(dt\otimes d\QQ)}
		+\|r_\Phi^\varepsilon\|_{L^2(\QQ)}
		\longrightarrow0.
	\end{equation}
	By \eqref{eq:perturbation}, \eqref{eq:forwardmoment}, and Lemma~\ref{lem:wasserstein}, $
	\sup_{t\le T}W_2(\mu_t^\varepsilon,\bar\mu_t)\to 0$, and $ 
	\sup_{\varepsilon,\theta,t}
	m_4(\mu_t^{\varepsilon,\theta})<\infty.$ 
	Moreover, $
	\varepsilon\chi^\varepsilon=X^\varepsilon-\bar X
	\longrightarrow0$ in $\mathcal S^q$, and $
	\varepsilon h=u^\varepsilon-\bar u\to 0 $ in $L^q(\drm t\otimes \drm\QQ)$
	for every finite $q\ge2$, and the same holds for
	independent copies of the state variables.
	
	Consider the finite measure space
	\[
	\mathfrak S=[0,T]\times\Omega\times\Omega'
	\times[0,1]^2,
	\qquad
	d\mathfrak m=dt\,d\QQ\,d\QQ'\,d\theta\,da.
	\]
	The four derivative differences in $r_f^\varepsilon$,
	namely
	\[
	\begin{aligned}
		A_1^\varepsilon&=f_x^{\varepsilon,\theta}-\bar f_x,
		&A_2^\varepsilon&=f_u^{\varepsilon,\theta}-\bar f_u,\\
		A_3^\varepsilon&=f_\mu^{\varepsilon,\theta}(X^{\varepsilon\prime})
		-\bar f_\mu(\bar X'),
		&A_4^\varepsilon&=f_{\mu,1}^{\varepsilon,\theta}
		(\bar X'+a\varepsilon\chi^{\varepsilon\prime})
		-\bar f_{\mu,1}(\bar X'),
	\end{aligned}
	\]
	converge to zero in $\mathfrak m$-measure by
	Assumption~\ref{ass:regularity}(iii).
	
Set
\[ 
	H_t^\varepsilon= 1
	+|\bar X_{t-}|^2+|X_{t-}^\varepsilon|^2
	+|\bar X_t'|^2+|X_t^{\varepsilon\prime}|^2\\
	+|\bar u_t|^2+|v_t|^2. 
\]
Then, based on  \eqref{eq:forwardmoment} and the admissibility
of $\bar u,v$, we have 
\[
\sup_{0<\varepsilon\leq1}
\int_{\mathfrak S}(H_t^\varepsilon)^{16}\,
\mathrm d\mathfrak m<\infty.
\]
On $\{H_t^\varepsilon\leq R\}$, all interpolated state and control arguments are bounded. Their convergence in measure, uniform $W_2$ convergence of the measure arguments, and uniform bound on the fourth moments permit application of Assumption~\ref{ass:regularity}(iii) on this set. Combined with Markov's inequality, we obtain
\[
\limsup_{\varepsilon\downarrow0}
\mathfrak m(|A_i^\varepsilon|>\delta)
\leq \frac{C}{R^{16}},
\qquad \delta>0.
\]
Taking $R\to\infty$, we have  $A_i^\varepsilon\to0$ in $\mathfrak m$-measure.

All interpolated arguments are convex combinations of their endpoints, so applying \eqref{eq:costderivatives} yields
$|A_i^\varepsilon|\leq C(H_t^\varepsilon)^2,$ and $ \sup_{0<\varepsilon\leq1}\sum_{i=1}^4\|A_i^\varepsilon\|_{L^8(\mathfrak m)}<\infty.$ 
Hence $\{|A_i^\varepsilon|^4\}_{0<\varepsilon\leq1}$ is uniformly integrable. By using Vitali's theorem,
\[
\|A_i^\varepsilon\|_{L^4(\mathfrak m)}\longrightarrow0,
\qquad i=1,2,3,4.
\]

Moreover, set $B_1^\varepsilon=\chi^\varepsilon$, $B_2^\varepsilon=h$,
$B_3^\varepsilon=K^{\varepsilon\prime}$, and
$B_4^\varepsilon=\bar M'\chi^{\varepsilon\prime}$, which satisfy
\[
\sup_{0<\varepsilon\le1}\sum_{i=1}^4
\|B_i^\varepsilon\|_{L^4(\mathfrak S)}<\infty.
\]
By H\"older's inequality, $
\|A_i^\varepsilon B_i^\varepsilon\|_{L^2(\mathfrak S)}
\le\|A_i^\varepsilon\|_{L^4(\mathfrak S)}
\|B_i^\varepsilon\|_{L^4(\mathfrak S)}\to0.$ 

	By Jensen's inequality with respect to $\QQ'$, $\theta$, and $a$, we obtain
	\[
	\|r_f^\varepsilon\|_{L^2(dt\otimes d\QQ)}
	\le\sum_{i=1}^4\|A_i^\varepsilon B_i^\varepsilon\|_{L^2(\mathfrak S)}\to0.
	\]
	Applying the same estimates on $\Omega\times\Omega'\times[0,1]^2$, with the control terms omitted, yields $\|r_\Phi^\varepsilon\|_{L^2(\QQ)}\to0$.  Therefore, we obtain \eqref{eq:costremainderlimit}.
	
We next prove \eqref{eq:costquotientconvergence}.
The quadratic-growth bounds and moment estimates yield
$\bar f_x\in L^4(dt\otimes d\mathbb Q)$, and
$\Phi_x(\bar X_T,\bar\mu_T)\in L^4(\mathbb Q)$.
By \eqref{eq:processdifferentiability},
$\|\bar f_x\cdot(\chi^\varepsilon-\chi)\|_{L^2(dt\otimes \drm \mathbb Q)} \to 0$,
and $\|\Phi_x(\bar X_T,\bar\mu_T)\cdot(\chi_T^\varepsilon-\chi_T)\|_{L^2(\mathbb Q)} \to 0$.

Now set $\bar A_t=A_2(\bar X_{t-},\bar\mu_t,\bar u_t)$.
By the Cauchy--Schwarz inequality together with \eqref{eq:costderivatives},
\[
|D_\mu f(t;\xi,k)|\le C\bar A_t\bigl(\|k_{t-}\|_{L^2(\mathbb Q)}+\|\xi_{t-}\|_{L^2(\mathbb Q)}\bigr),
\]
and
\[
\|D_\mu f(\cdot;\xi,k)\|_{L^2(dt\otimes d\mathbb Q)}\le C\bigl(\|\xi\|_{\mathcal S^2}+\|k\|_{\mathcal S^2}\bigr).
\]
Here $C$ absorbs the finite $L^2(dt\otimes d\mathbb Q)$ norm of $\bar A$ and the fourth moments of $\bar X$.
The same estimate applies to the two copied terminal-cost terms, with $A_2^0(\bar X_T,\bar\mu_T)$ in place of $\bar A_t$.
Combining \eqref{eq:processdifferentiability},
\eqref{eq:costquotientremainder}, and
\eqref{eq:costremainderlimit}, we obtain \eqref{eq:costquotientconvergence}.

	Finally, based on 
	$M^\varepsilon=\bar M+\varepsilon K^\varepsilon$,
	$f^\varepsilon=\bar f+\varepsilon\Delta_\varepsilon f$,
	and
	$\Phi_T^\varepsilon=\bar\Phi
	+\varepsilon\Delta_\varepsilon\Phi_T$, 
	\[ 
	\frac{J(u^\varepsilon)-J(\bar u)}{\varepsilon}
	=\EE\bigl[
	K_T^\varepsilon\bar\Phi
	+\bar M_T\Delta_\varepsilon\Phi_T
	+\varepsilon K_T^\varepsilon\Delta_\varepsilon\Phi_T
	\bigr]+\EE\int_0^T\bigl[
	K_t^\varepsilon\bar f_t
	+\bar M_t\Delta_\varepsilon f_t
	+\varepsilon K_t^\varepsilon\Delta_\varepsilon f_t
	\bigr]\,\drm t. 
	\]
	By \eqref{eq:costquotientconvergence},
	$\Delta_\varepsilon\Phi_T$ and
	$\Delta_\varepsilon f$ are bounded in their
	respective $L^2$ spaces. Also,
	$\sup_\varepsilon\|K^\varepsilon\|_{\mathcal S^2}<\infty$.
	Hence,
	\[ 
	\varepsilon\EE
	|K_T^\varepsilon\Delta_\varepsilon\Phi_T|
	\le C\varepsilon
	\|K^\varepsilon\|_{\mathcal S^2}
	\|\Delta_\varepsilon\Phi_T\|_{L^2(\QQ)} \to 0,\] and \[
	\varepsilon\EE\int_0^T
	|K_t^\varepsilon\Delta_\varepsilon f_t|\,\drm t
	\le C\varepsilon
	\|K^\varepsilon\|_{\mathcal S^2}
	\|\Delta_\varepsilon f\|_{L^2(\drm t\otimes \drm \QQ)}\to 0.\]
	
	Since $\bar f\in L^2(dt\otimes d\QQ)$,
	$\bar\Phi\in L^2(\QQ)$, and
	$0\le\bar M\le\CM$, combining with 
	\eqref{eq:processdifferentiability} and
	\eqref{eq:costquotientconvergence} imply
	\[
	\lim_{\varepsilon\downarrow0}
	\frac{J(u^\varepsilon)-J(\bar u)}{\varepsilon}
	=J_1(h),
	\]
	which proves the expansion.
	If $\bar u$ is optimal, then
	$J(\bar u+\varepsilon(v-\bar u))\ge J(\bar u)$
	for all $\varepsilon\in[0,1]$.
	Dividing by $\varepsilon>0$ and letting
	$\varepsilon\downarrow0$ yields
	$J_1(v-\bar u)\ge0$.
\end{proof}
	
	\subsection{Variation of the physical measure}
The dominating reference intensity is a representation parameter. We first verify invariance under changes to this parameter, including invariance of the admissible control class.

	\begin{proposition}\label{prop:referenceinvariance}
	Let $\QQ=\QQ^{(c)}$ for the reference measure used in the main construction, where
	$c\ge\max\{1,\lambda_*\}$. For any constant $d\ge\max\{1,\lambda_*\}$, let 
	\begin{equation}\label{eq:referencechange}
		\frac{\drm \QQ^{(d)}}{\drm \QQ^{(c)}}\bigg|_{\mathcal F_t}=:D_t=\exp\{n(c-d)t\}(d/c)^{\sum_{k=1}^nY_k(t)},
	\end{equation}
	Then $\QQ^{(d)}\sim\QQ^{(c)}$. Moreover, for
	every $u\in\Uad$, the unique solutions in the two dominating reference formulations satisfy
	\begin{equation}\label{eq:transferredM}
		X^{u,d}=X^{u,c},\qquad M_t^{u,d}=M_t^{u,c}/D_t,\qquad \mu_t^{u,d}=\mu_t^{u,c}.
	\end{equation}
	
\end{proposition}

\begin{proof}
	For every $q\in\R$, the exponential moment
	formula for Poisson processes gives
	\[
	\EE^{\QQ^{(c)}}D_t^q
	=
	\exp\left\{
	nq(c-d)t
	+nct\left[\left(\frac dc\right)^q-1\right]
	\right\}<\infty.
	\]
	In particular, $\EE^{\QQ^{(c)}}D_t=1$.
	Moreover, for $0\le s\le t\le T$,
	\[
	\begin{aligned}
		\EE^{\QQ^{(c)}}[D_t\mid\mathcal F_s]
		&=D_s e^{n(c-d)(t-s)}
		\EE^{\QQ^{(c)}}\left[
		\left(\frac dc\right)^{
			\sum_{k=1}^n(Y_k(t)-Y_k(s))}
		\right]\\
		&=D_s e^{n(c-d)(t-s)}
		e^{nc(t-s)(d/c-1)}\\
		&=D_s.
	\end{aligned}
	\]
	Thus $D$ is a strictly positive
	$\QQ^{(c)}$-martingale, and
	$\QQ^{(d)}\sim\QQ^{(c)}$.
	
	Since $D_T$ is measurable with respect
	to $\mathcal F_T^Y$, the joint law of
	$(W,\bar N)$ and its independence from
	$Y$ are preserved under $\QQ^{(d)}$.
	The processes $Y_1,\ldots,Y_n$ are
	independent Poisson processes with
	intensity $d$ under $\QQ^{(d)}$.
	
	For every finite $r\ge1$,
	H\"older's inequality gives
	\[ \EE^{\QQ^{(d)}}\int_0^T|u_t|^r\,\drm t =\EE^{\QQ^{(c)}}\left[
		D_T\int_0^T|u_t|^r\,\drm t
		\right]\le
		\left(\EE^{\QQ^{(c)}}D_T^2\right)^{1/2}
		\left(
		T\EE^{\QQ^{(c)}}
		\int_0^T|u_t|^{2r}\,\drm t
		\right)^{1/2}. 
	\]
	The reverse implication follows similarly
	from
	\[
	\EE^{\QQ^{(d)}}D_T^{-r}
	=\EE^{\QQ^{(c)}}D_T^{1-r}<\infty.
	\]
	Since the two measures are equivalent,
	their completed observation filtrations
	coincide. Hence the admissible control
	class is unchanged.
	
	Fix $u\in\cU_{ad}$ and let
	$(X^{u,c},M^{u,c},\mu^{u,c})$
	be the solution under $\QQ^{(c)}$.
	Define $
	\widetilde X=X^{u,c},$ $
	\widetilde M=\frac{M^{u,c}}{D},$ and $
	\widetilde\mu=\mu^{u,c}.$ 
	By \eqref{eq:product} and
	\eqref{eq:referencechange},
	\[ 
	\widetilde M_t
	=
	\exp\left\{
	\sum_{k=1}^n\int_0^t
	(d-\lambda_k^u(s))\,\drm s
	\right\}\times
	\prod_{k=1}^n
	\prod_{\substack{0<s\le t
			\Delta Y_k(s)=1}}
	\frac{\lambda_k^u(s)}d.
	\]
	Since $0\le\lambda_k^u\le\lambda^*\le d$, $
	0\le\widetilde M_t\le e^{ndt}.$ 
	Thus $\widetilde M$ satisfies the
	likelihood equation with reference
	intensity $d$.
	
	Furthermore, for every finite $p\ge2$,
	\[ \EE^{\QQ^{(d)}}
	\sup_{t\le T}|\widetilde X_t|^p\le
	\left(\EE^{\QQ^{(c)}}D_T^2\right)^{1/2}
	\left(
	\EE^{\QQ^{(c)}}
	\sup_{t\le T}|X_t^{u,c}|^{2p}
	\right)^{1/2}
	<\infty.
	\]
	Since the state noises have the same
	characteristics under the two reference
	measures, $\widetilde X$ satisfies the
	state equation with marginal
	$\widetilde\mu$ under $\QQ^{(d)}$. Then, for every bounded Borel function
	$\varphi$, we have
	\[\EE^{\QQ^{(d)}}[
		\widetilde M_t\varphi(\widetilde X_t)]=\EE^{\QQ^{(c)}}[
		D_t\widetilde M_t
		\varphi(\widetilde X_t)]=\EE^{\QQ^{(c)}}[
		M_t^{u,c}\varphi(X_t^{u,c})]=\inner{\mu_t^{u,c}}{\varphi}
		=\inner{\widetilde\mu_t}{\varphi}.
	\]
	Therefore
	$(\widetilde X,\widetilde M,
	\widetilde\mu)$
	solves the system under $\QQ^{(d)}$.
	By the uniqueness in
	Theorem~\ref{thm:forward}, we have $X^{u,d}=X^{u,c}$, $ M^{u,d}=\frac{M^{u,c}}D$, and $	\mu^{u,d}=\mu^{u,c}.$ 
	
	Finally, since  $
	M_T^{u,d}\,d\QQ^{(d)}=M_T^{u,c}\,d\QQ^{(c)}$, and $
	M_t^{u,d}\,d\QQ^{(d)}
	\big|_{\mathcal F_t} =M_t^{u,c}\,d\QQ^{(c)}
	\big|_{\mathcal F_t}.$ 
	Hence, the physical probability measures coincide. Moreover,
	\[ \begin{aligned}
		J^{(d)}(u)
		=&\EE^{\QQ^{(d)}}\bigg[
		M_T^{u,d}\Phi(X_T^{u,d},\mu_T^{u,d})+\int_0^T M_t^{u,d}
		f(t,X_{t-}^{u,d},\mu_t^{u,d},u_t)
		\,\drm t\bigg]\\
		=& \EE^{\QQ^{(c)}}\bigg[
		M_T^{u,c}\Phi(X_T^{u,c},\mu_T^{u,c})+\int_0^T M_t^{u,c}
		f(t,X_{t-}^{u,c},\mu_t^{u,c},u_t)
		\,\drm t\bigg]\\
		= &J^{(c)}(u).
	\end{aligned}
	\]
\end{proof}
	
For notational simplicity, we set $\bar\PP=\mathbb P^{\bar u}$ and $\mathbb P^\varepsilon=\mathbb P^{u^\varepsilon}$. For any finite signed measure $\nu$, its total variation norm is
\[
\|\nu\|_{\mathrm{TV}}
=|\nu|(\Omega)
=\sup_{\substack{F:\Omega\to\mathbb R\\|F|\le1}}
\left|\int_\Omega F\,\mathrm{d}\nu\right|.
\]

\begin{proposition}
	\label{prop:singular}
	
	\[
	\frac{\PP^\varepsilon-\bar\PP}{\varepsilon}
	\xrightarrow{\mathrm{TV}}\dot\PP,
	\qquad
	\frac{d\dot\PP}{d\QQ}=K_T.
	\]
	Set $A=\{\bar M_T>0\}$. Define $S_T=K_T/\bar M_T$ on $A$ and $S_T=0$ on $A^c$, and set $\nu^\perp(B)=\EE[K_T\ind_{A^c}\ind_B]$. 
	Then
	\begin{equation}\label{eq:Lebesguedecomp} 
		\dot\PP=S_T\bar\PP+\nu^\perp,
	\end{equation}
	with $ 	\nu^\perp\ge0$ and $\nu^\perp\perp\bar\PP.$ Moreover, $  
	S_T\in L^1(\bar\PP),$ and $\EE^{\bar\PP}S_T=-\nu^\perp(\Omega).$ 
	
	Therefore, for any $F\in L^q(\QQ)$, $q>1$,
	\begin{equation}\label{eq:physicalderivative}
		\left.\frac{d}{d\varepsilon}
		\EE^{\PP^\varepsilon}F\right|_{0+}
		=\EE^{\bar\PP}[S_TF]
		+\int_\Omega F\,d\nu^\perp.
	\end{equation}
\end{proposition}

\begin{proof}
	Since $
	\frac{\drm \PP^\varepsilon}{\drm \QQ}=M_T^\varepsilon,$ and $
	\frac{\drm \bar\PP}{\drm \QQ}=\bar M_T$,  we have, for every $B\in\mathcal F_T$, $
	\frac{\PP^\varepsilon(B)-\bar\PP(B)}
	{\varepsilon}
	=\EE[K_T^\varepsilon\ind_B].$ 
	Consequently, by Theorem~\ref{thm:differentiability}, $ 
	\left\|
	\frac{\PP^\varepsilon-\bar\PP}{\varepsilon}
	-\dot\PP
	\right\|_{\mathrm{TV}}=\int_\Omega
	\left|
	\frac{M_T^\varepsilon-\bar M_T}{\varepsilon}
	-K_T
	\right|\drm \QQ=\mathbb E^\QQ|K_T^\varepsilon-K_T| \to 0.$
	
	Let $A=\{\bar M_T>0\}$. Since
	$\bar M_T=0$ on $A^c$ and
	$M_T^\varepsilon\ge0$, we have $
	K_T^\varepsilon
	=\frac{M_T^\varepsilon}{\varepsilon}
	\ge0$ on $A^c$.
	
	Taking $x^-=\max\{-x,0\}$, we obtain $
	\EE[(K_T)^-\ind_{A^c}]\le\EE[
	|K_T^\varepsilon-K_T|\ind_{A^c}]\to 0$.
	Hence ,$K_T\ge0$ on $A^c$,
	$\QQ$-almost surely, and $
	\nu^\perp(B)
	=\EE[K_T\ind_{A^c}\ind_B]\ge0.$ 
	
	Moreover, $
	\bar\PP(A^c)
	=\EE[\bar M_T\ind_{A^c}]=0,$ and $
	\nu^\perp(A)=0.$ 
	Thus, $\nu^\perp\perp\bar\PP$.
	
	Moreover, by the definition of $S_T$, $
	\EE^{\bar\PP}|S_T|=\EE[\bar M_T|S_T|] =\EE[|K_T|\ind_A]<\infty.$  Then, for every $B\in\mathcal F_T$,
	\[
	\begin{aligned}
		\dot\PP(B)
		&=\EE[K_T\ind_B]\\
		&=\EE[K_T\ind_A\ind_B]
		+\EE[K_T\ind_{A^c}\ind_B]\\
		&=\EE^{\bar\PP}[S_T\ind_B]
		+\nu^\perp(B).
	\end{aligned}
	\]
	This proves the decomposition
	\eqref{eq:Lebesguedecomp}.
	
	Since $K$ is a $\QQ$-martingale
	with $K_0=0$, $
	0=\EE K_T =\EE[K_T\ind_A]
	+\EE[K_T\ind_{A^c}] =\EE^{\bar\PP}S_T
	+\nu^\perp(\Omega). $ 
	Hence, $
	\EE^{\bar\PP}S_T
	=-\nu^\perp(\Omega).$ 
	
	Therefore, for any bounded $\mathcal F_T$-measurable random variable $F$,
	\[ \frac{\EE^{\PP^\varepsilon}F-\EE^{\bar\PP}F}{\varepsilon}=\EE[K_T^\varepsilon F]\to \EE[K_TF]=\EE^{\bar\PP}[S_TF]
	+\int_\Omega F\,d\nu^\perp. 
	\]
	This proves \eqref{eq:physicalderivative}.
	
	For $F\in L^q(\QQ)$ with $q>1$,
	let $r=q/(q-1)$. By
	\eqref{eq:processdifferentiability},
	$K_T^\varepsilon\to K_T$ in $L^r(\QQ)$.
	Hence by using H\"older's inequality, $\left|\EE[(K_T^\varepsilon-K_T)F]\right|\le
	\|K_T^\varepsilon-K_T\|_{L^r(\QQ)}	\|F\|_{L^q(\QQ)}\to 0.$
	Thus \eqref{eq:physicalderivative} also holds for such $F$.
	
	Finally, by
	Proposition~\ref{prop:referenceinvariance},
	the physical measures
	$\PP^\varepsilon$ and $\bar\PP$
	are independent of the reference
	intensity $c$. The uniqueness of the
	total-variation limit and the Lebesgue
	decomposition implies that
	$S_T\bar\PP$ and $\nu^\perp$
	are independent of $c$.
\end{proof}
We now apply this decomposition to the cost. Define
\[G^\varepsilon=\Phi(X_T^\varepsilon,\mu_T^\varepsilon)
	+\int_0^Tf(t,X_{t-}^\varepsilon,\mu_t^\varepsilon,u_t^\varepsilon)
	\dd t,\]
	and \[
	\dot G=D\Phi_T+\int_0^T Df_t\dd t.
\]
Thus $G^0=\bar\Phi+\int_0^T\bar f_t\dd t$.
The growth bounds and \eqref{eq:costquotientconvergence} imply
$G^0,\dot G\in L^2(\QQ)$ and $
\Delta_\varepsilon G
:=\frac{G^\varepsilon-G^0}{\varepsilon}
\to\dot G$ in $L^2(\QQ).$

By the previously established cost-quotient of $\Phi$ and $f$ convergence given in \eqref{eq:costquotientconvergence}, by using the triangle inequality and
Cauchy--Schwarz inequality, we have 
\[ 
	\left\|
	\frac{G^\varepsilon-G^0}{\varepsilon}-\dot G
	\right\|_{L^2(\QQ)}\le\left\|\frac{\Phi_T^\varepsilon-\bar\Phi}{\varepsilon}-D\Phi_T
	\right\|_{L^2(\QQ)}+\sqrt T
	\left\|
	\frac{f^\varepsilon-\bar f}{\varepsilon}-Df
	\right\|_{L^2(\mathrm dt\otimes\mathrm d\QQ)}
\to 0.
\]
Hence, $ \Delta_\varepsilon G:=\frac{G^\varepsilon-G^0}{\varepsilon}
\to\dot G $ in $L^2(\QQ)$.

Moreover, by Theorem~\ref{thm:physicalrepresentation},
$J(u^\varepsilon)=\EE[M_T^\varepsilon G^\varepsilon]$.
Then 
\[
\frac{J(u^\varepsilon)-J(\bu)}{\varepsilon}
=
\EE\left[
K_T^\varepsilon G^0
+\bM_T\Delta_\varepsilon G
+\varepsilon K_T^\varepsilon\Delta_\varepsilon G
\right].
\]
The first two terms converge in expectation. For the last term, $
\varepsilon
\left|\EE[K_T^\varepsilon\Delta_\varepsilon G]\right|
\leq
\varepsilon\|K_T^\varepsilon\|_{L^2(\QQ)}
\|\Delta_\varepsilon G\|_{L^2(\QQ)}
\to 0$.
Combining with \eqref{eq:costquotientconvergence} , we have 
$$ J_1(h)=\EE[\bM_T\dot G+K_TG^0]=\EE^{\overline{\PP}}[\dot G]
	+\EE^{\overline{\PP}}[S_TG^0]
	+\int_\Omega G^0\,\mathrm d\nu^\perp.$$
 
Since $ 
\EE^{\overline{\PP}}|\dot G|
\leq C_M\EE|\dot G|<\infty,$ and $ 
\EE^{\overline{\PP}}|S_TG^0|
+\int_\Omega|G^0|\,\mathrm d\nu^\perp
=\EE|K_TG^0|<\infty$,  thus all three terms are absolutely integrable.

 The final term arising in the cost variation is supported on paths whose histories carry zero probability under the baseline measure. More broadly, identical singular contributions arise within the derivative of $\mathbb{E}^{\mathbb{P}^\varepsilon}[F]$ for any fixed bounded functional $F$. Provided $\nu^\perp\neq 0$, not all such derivatives admit representation via a single integrable score function under $\overline{\mathbb{P}}$, as demonstrated by the corollary below.
 
	\begin{corollary}\label{cor:noscore}
		Suppose that $\nu^\perp\neq0$. Then there
	exists no $Z\in L^1(\bar\PP)$ such that $
	\left.\frac{d}{d\varepsilon}
	\EE^{\PP^\varepsilon}[F]\right|_{0+}
	=\EE^{\bar\PP}[ZF]$ 
	for every bounded $\mathcal F_T$-measurable
	random variable $F$.
	\end{corollary}
	\begin{proof}
	Let $A=\{\bar M_T>0\}$. By
	Proposition~\ref{prop:singular}, $
	\bar\PP(A^c)=0,$ and $ 
	\left.\frac{d}{d\varepsilon}
	\PP^\varepsilon(A^c)\right|_{0+}
	=\nu^\perp(\Omega)>0.$ 
	Taking $F=\ind_{A^c}$ gives $
	\EE^{\bar\PP}[ZF]=0$ 
	for every $Z\in L^1(\bar\PP)$,
	which is a contradiction.
\end{proof}
	
	The following product identity gives the pathwise interpretation of the
	boundary contribution in terms of the vanishing likelihood multipliers.

		\begin{proposition}\label{prop:productvariation}
		Let the jump times up to time $t$ be $0<\tau_1<\cdots<\tau_{J(t)}\le t$, where $k_j$ denotes the channel at time $\tau_j$.
		Define  $
		A_t=\sum_{k=1}^n\int_0^t(c-\bar\lambda_k(s))ds,$ $ B_t=\sum_{k=1}^n\int_0^tD\lambda_k(s)ds$, $a_j=\frac{\bar\lambda_{k_j}(\tau_j)}c$, and $\dot a_j=\frac{D\lambda_{k_j}(\tau_j)}c$. Then, 
		\begin{equation}\label{eq:Kproduct}
			K_t=e^{A_t}\left[
			-B_t\prod_{i=1}^{J(t)}a_i
			+\sum_{j=1}^{J(t)}\dot a_j
			\prod_{\substack{1\leq i\leq J(t)\\i\ne j}}a_i\right].
		\end{equation}
	\end{proposition}

\begin{proof}
	Since $D\lambda_k$ is predictable and
	belongs to $L^p(dt\otimes d\QQ)$ for every
	finite $p\ge2$, then  
	\[
	\EE\sum_{k=1}^n\int_0^T
	|D\lambda_k(s)|^p\,\drm Y_k(s)
	=c\sum_{k=1}^n\EE\int_0^T
	|D\lambda_k(s)|^p\,\drm s<\infty.
	\]
	In particular, $D\lambda_k$ is finite
	at the observation jump times almost surely.
	
	For $t\in[0,T]$, let $P_t=\prod_{i=1}^{J(t)}a_i$, $S_t=\sum_{j=1}^{J(t)}\dot a_j\prod_{\substack{1\le i\le J(t)\\i\ne j}}a_i$,   
	and define $\mathcal K_t=e^{A_t}(S_t-B_tP_t)$. By \eqref{eq:product}, we have $\bar M_t=e^{A_t}P_t$.
	
	Since 	between successive jumps of $Y$, both $P_t$ and $S_t$ are constant, whereas $\drm A_t=\sum_{k=1}^n(c-\bar\lambda_k(t))\,\drm t$, and $\drm B_t=\sum_{k=1}^nD\lambda_k(t)\,\drm t.$  Therefore, 
	\[ 
	\drm\mathcal K_t
	=e^{A_t}(S_t-B_tP_t)\,\drm A_t
	-e^{A_t}P_t\,\drm B_t=\left[
	\sum_{k=1}^n(c-\bar\lambda_k(t))
	\mathcal K_t
	-\bar M_t\sum_{k=1}^nD\lambda_k(t)
	\right]\drm t.
	\]
	
	Moreover,	at an observation jump time $\tau_j$, by the definitions of $P$ and $S$, we have  $
	P_{\tau_j}=a_jP_{\tau_j-}$, and $S_{\tau_j}=a_jS_{\tau_j-}+\dot a_jP_{\tau_j-}.$  Since $A$ and $B$ are continuous, thus
	\[
	\begin{aligned}
		\mathcal K_{\tau_j}
		&=e^{A_{\tau_j}}
		\left(
		a_jS_{\tau_j-}
		+\dot a_jP_{\tau_j-}
		-B_{\tau_j}a_jP_{\tau_j-}
		\right)\\
		&=a_j\mathcal K_{\tau_j-}
		+\dot a_j\bar M_{\tau_j-}\\
		&=\frac{\bar\lambda_{k_j}(\tau_j)}c
		\mathcal K_{\tau_j-}
		+\frac{\bar M_{\tau_j-}}c
		D\lambda_{k_j}(\tau_j).
	\end{aligned}
	\]
	Thus $\mathcal K$ satisfies the same
	continuous dynamics and jump relation
	as \eqref{eq:K}. Since
	$\mathcal K_0=K_0=0$, then from the uniqueness of the
	linear equation of $\cK$, we derive that 
	$\mathcal K=K$,  which implies that  \eqref{eq:Kproduct}.
	
\end{proof}
		\begin{corollary}
		Under the same Assumptions as Proposition \ref{prop:variation},   if none of the multipliers of  \eqref{eq:Kproduct} is zero, then
		\begin{equation}\label{eq:scoreonsupport}
			\frac{K_t}{\bar M_t}
			=-\sum_k\int_0^tD\lambda_k(s)ds
			+\sum_{j=1}^{J(t)}\frac{D\lambda_{k_j}(\tau_j)}{
				\bar\lambda_{k_j}(\tau_j)}.
		\end{equation}
		If at least two multipliers  of  \eqref{eq:Kproduct}  are zero, then $K_t=0$. If exactly one, say
		$a_{j_0}$, is zero, then
		\begin{equation}\label{eq:onezerofactor}
			K_t=e^{A_t}\dot a_{j_0}\prod_{i\ne j_0}a_i\ge0.
		\end{equation}
		The last sign assertion concerns the feasible one-sided variation
		$u^\varepsilon=\bar u+\varepsilon(v-\bar u)$.
	\end{corollary}
	
	\begin{proof}
		
		If $a_j>0$ for every $j\le J(t)$,
		then $\bar M_t=e^{A_t}P_t>0$ and
		\[
		\begin{aligned}
			\frac{K_t}{\bar M_t}
			&=-B_t+\frac{S_t}{P_t}\\
			&=-\sum_{k=1}^n\int_0^t
			D\lambda_k(s)\,\drm s
			+\sum_{j=1}^{J(t)}
			\frac{\dot a_j}{a_j}\\
			&=-\sum_{k=1}^n\int_0^t
			D\lambda_k(s)\,\drm s
			+\sum_{j=1}^{J(t)}
			\frac{D\lambda_{k_j}(\tau_j)}
			{\bar\lambda_{k_j}(\tau_j)}.
		\end{aligned}
		\]
		This proves \eqref{eq:scoreonsupport}.
		
		If at least two factors $a_j$ vanish,
		then $P_t=S_t=0$, and hence $K_t=0$.
		If exactly one factor $a_{j_0}$ vanishes,
		then $
		P_t=0$, and $
		S_t=\dot a_{j_0}
		\prod_{i\ne j_0}a_i,$ 
		which gives $
		K_t=e^{A_t}\dot a_{j_0}
		\prod_{i\ne j_0}a_i.$ 
		
		It remains to establish the
		nonnegativity of $\dot a_{j_0}$.
		For every $\varepsilon>0$, on the
		predictable set $\{\bar\lambda_k=0\}$, $
		\Delta_\varepsilon\lambda_k(t)
		=\frac{\lambda_k^\varepsilon(t)}
		{\varepsilon}\ge0.$ 
		By \eqref{eq:exactremainder},
		\eqref{eq:remainderconvergence} and
		\eqref{eq:processdifferentiability}, $
		\EE\int_0^T
		|\Delta_\varepsilon\lambda_k(t)
		-D\lambda_k(t)|^p\,\drm t
		\to 0$.
		
		Writing $x^-=\max\{-x,0\}$, we obtain
		\[ 
		\EE\int_0^T
		\ind_{\{\bar\lambda_k(t)=0\}}
		|(D\lambda_k(t))^-|^p\,\drm t\le
		\EE\int_0^T
		|\Delta_\varepsilon\lambda_k(t)
		-D\lambda_k(t)|^p\,\drm t
		\to 0.
		\]
		Therefore, $
		D\lambda_k(t)\ge0$ on set $\{\bar\lambda_k(t)=0\},$ $\drm t\otimes\drm \QQ$-a.s..
		
		Since $\bar\lambda_k$ and
		$D\lambda_k$ are predictable, then 
		\[  \EE\sum_{j=1}^{J(T)}
		\ind_{\{
			\bar\lambda_{k_j}(\tau_j)=0,\,
			D\lambda_{k_j}(\tau_j)<0
			\}}=c\sum_{k=1}^n
		\EE\int_0^T
		\ind_{\{
			\bar\lambda_k(t)=0,\,
			D\lambda_k(t)<0
			\}}\,\drm t
		=0. 
		\]
		Hence $\dot a_{j_0}\ge0$ almost surely
		whenever $a_{j_0}=0$. Since all the
		remaining factors are nonnegative,
		\eqref{eq:onezerofactor} follows.
	\end{proof}
	The directional cost formula is now established, including the contribution carried by the likelihood derivative. We next turn to its Hamiltonian system, which eliminates the state and likelihood variations from the optimality condition.
	
	\section{A stochastic maximum principle}\label{sec:smp}
	Throughout this section,  fix $\bar u\in\mathcal U_{\mathrm{ad}}$ and let $(\bar X,\bar M,\bar\mu)$ denote the associated solution.

We represent the directional cost derivative through adjoint equations and derive a necessary optimality condition by projection onto the observation-predictable $\sigma$-field.

 For $P\in\R^d$, $Q\in\R^{d\times d_W}$, $R\in L^2(E,\nu;\R^d)$, and $\ell\in\R^n$, define
	\begin{align*}
		H(t,x,a,\mu,u;P,Q,R,\ell)
		= &af(t,x,\mu,u)+P\cdot b(t,x,\mu,u)
		+\tr(Q^\top\sigma(t,x,\mu,u))\\
		&+\int_E R(\zeta)\cdot\eta(t,x,\mu,u,\zeta)\nu(\mathrm d\zeta)
		+a\sum_{k=1}^n\ell_k(\lambda_k(t,x,\mu,u)-c).
	\end{align*}
	Then, the adjoint equations are as follows 
\begin{equation}\label{eq:scalaradjoint}
	\left\{
	\begin{aligned}
		\drm p_t=&-\left(H_a(t)+\EE'[H_\mu(t,\omega';\bar X_{t-})]\right)\drm t
		+q_t\cdot \drm W_t+\int_E\varrho_t(\zeta)\ti N(\drm \zeta,\drm t)
		+\sum_k\ell_{k,t}\drm \Yt_k(t),\\
		p_T=&\Phi(\bar X_T,\bar\mu_T)
		+\EE'\left[
		\bar M_T'\Phi_\mu
		(\bar X_T',\bar\mu_T;\bar X_T)
		\right],
	\end{aligned}
	\right.
\end{equation}
and 
\begin{equation}\label{eq:vectoradjoint}
	\left\{
	\begin{aligned}
		\drm P_t=&-\left(H_x(t)+\bar M_{t-}\EE'[H_{\mu,1}(t,\omega';\bar X_{t-})]\right)\drm t
		+Q_t\,\drm W_t +\int_E R_t(\zeta)\Nt(\drm \zeta,\dd t)
		+\sum_k\Pi_{k,t}d\Yt_k(t),\\
		P_T=&\bar M_T\bigg(
		\Phi_x(\bar X_T,\bar\mu_T)
		+\EE'\left[
		\bar M_T'\Phi_{\mu,1}
		(\bar X_T',\bar\mu_T;\bar X_T)
		\right]\bigg).
	\end{aligned}
	\right.
\end{equation}
Here, the derivatives of $H$ are evaluated at $(t,\bar X_{t-},\bar M_{t-},\bar\mu_t,\bar u_t;P_{t-},Q_t,R_t,\ell_t).$

The terminal conditions of the adjoint equations are chosen
to match the terminal part of the first-order cost variation.
More precisely, we seek to express
$\EE[K_T\bar\Phi+\bM_TD\Phi_T]$ as
$\EE[p_TK_T+P_T\cdot\chi_T]$,
so that it can subsequently be treated by Itô's product formula.

Using the definition of $D\Phi_T$ and interchanging the
independent copies, we obtain
\begin{align*}
	\EE[K_T\bar\Phi+\bM_TD\Phi_T]
	=\EE\bigg[&
	K_T\Bigl(
	\bar\Phi+
	\EE'[\bM_T'\Phi_\mu(\bX_T',\bmu_T;\bX_T)]
	\Bigr)\\
	&+\bM_T\chi_T\cdot\Bigl(
	\Phi_x(\bX_T,\bmu_T)
	+\EE'[\bM_T'\Phi_{\mu,1}(\bX_T',\bmu_T;\bX_T)]
	\Bigr)
	\bigg].
\end{align*}
Thus the prescribed terminal values of $p$ and $P$ give
\begin{equation}\label{eq:terminalduality}
	\EE[p_TK_T+P_T\cdot\chi_T]
	=
	\EE[K_T\bar\Phi+\bM_TD\Phi_T].
\end{equation}
This identity provides the terminal matching needed for the
duality argument.  

Before carrying out this argument, we establish existence
and uniqueness of the adjoint processes in the required
square-integrable spaces on the natural filtration fixed
in Section~\ref{sec:model}.
	
	\begin{lemma}
	\label{lem:martingale}
	Let $\FF$ be the augmented natural filtration
	generated by $W$, $\bar N$ and $Y_1,\ldots,Y_n$.
	Then every square-integrable $\FF$-martingale
	$L$ admits a unique representation
	\begin{equation}\label{eq:martingaledecomp}
		L_t=L_0+\int_0^t Z_s\,\drm W_s
		+\int_0^t\int_E U_s(\zeta)\Nt(\drm\zeta,\drm s)
		+\sum_{k=1}^n\int_0^t l_{k,s}\,\drm\Yt_k(s),
	\end{equation}
	where the predictable integrands satisfy
	\[
	\EE\int_0^T\left(
	|Z_s|^2+\|U_s\|_{L^2(\nu)}^2
	+c\sum_{k=1}^n|l_{k,s}|^2
	\right)\drm s<\infty.
	\]
\end{lemma}

\begin{proof}
	Let
	\[
	E^*=E\sqcup\{1,\ldots,n\},
	\qquad
	\nu^*=\nu+c\sum_{k=1}^n\delta_k.
	\]
	Define the random measure $N^*$ on $E^*$
	by
	\[
	N^*(A\times(0,t])
	=\bar N((A\cap E)\times(0,t])
	+\sum_{k=1}^n\ind_{\{k\in A\}}Y_k(t).
	\]
	By independence, $N^*$ is a Poisson random
	measure with compensator
	$\nu^*(d\zeta)\drm t$, independent of $W$.
	
	The predictable representation property
	for the augmented natural filtration of
	$(W,N^*)$ (see \citet[Section~5.3]{applebaum2009levy}) yields
	\[
	L_t=L_0+\int_0^t Z_s\,\drm W_s
	+\int_0^t\int_{E^*}
	U_s^*(\zeta)\widetilde N^*(\drm\zeta,\drm s).
	\]
	Taking $
	U_s(\zeta)=U_s^*(\zeta)$ with $\zeta\in E$, and $ l_{k,s}=U_s^*(k),$ then 
	\eqref{eq:martingaledecomp} holds.
	
	Since the stochastic integrals are
	pairwise orthogonal, by using  It\^o's isometry, we obtain\[
	\EE|L_T-L_0|^2
	=\EE\int_0^T\left(
	|Z_s|^2+\|U_s\|_{L^2(\nu)}^2
	+c\sum_{k=1}^n|l_{k,s}|^2
	\right)\drm s.
	\]
	
To prove uniqueness, suppose that $(Z,U,l)$ and $(\widehat Z,\widehat U,\widehat l)$
are two predictable, square-integrable families of integrands
representing the same martingale $L$.
Set $
\Delta Z=Z-\widehat Z$, $
\Delta U=U-\widehat U,$ and $
\Delta l_k=l_k-\widehat l_k.$ 
Then, based on \eqref{eq:martingaledecomp}, we have 
\[
0=
\int_0^T\Delta Z_s\,\drm W_s
+\int_0^T\int_E
\Delta U_s(\zeta)\Nt(\drm\zeta,\drm s)
+\sum_{k=1}^n
\int_0^T\Delta l_{k,s}\,\drm\Yt_k(s).
\]
All the integrands in this identity remain square integrable.
By the pairwise orthogonality of the stochastic integrals
and It\^o's isometry,
\[
\begin{aligned}
	0
	&=\EE\left|
	\int_0^T\Delta Z_s\,\drm W_s
	+\int_0^T\int_E
	\Delta U_s(\zeta)\Nt(\drm\zeta,\drm s)
	+\sum_{k=1}^n
	\int_0^T\Delta l_{k,s}\,\drm\Yt_k(s)
	\right|^2\\
	&=\EE\int_0^T|\Delta Z_s|^2\,\drm s
	+\EE\int_0^T\int_E
	|\Delta U_s(\zeta)|^2\,\nu(\drm\zeta)\,\drm s
	+c\sum_{k=1}^n
	\EE\int_0^T|\Delta l_{k,s}|^2\,\drm s.
\end{aligned}
\]
Since every term on the right-hand side is nonnegative,
\[
\EE\int_0^T|\Delta Z_s|^2\,\drm s=0,
\qquad
\EE\int_0^T\int_E
|\Delta U_s(\zeta)|^2\,\nu(\drm\zeta)\,\drm s=0.
\]
Moreover, since  $c>0$, then for $k=1,\cdots, n$, $
\EE\int_0^T|\Delta l_{k,s}|^2\,\drm s=0.$ 
It follows that $
Z=\widehat Z$, and $ l_k=\widehat l_k,$  $\drm t\otimes \drm \QQ$-a.e. 
and $
U=\widehat U$ $\drm t\otimes\drm \QQ\otimes \nu$-a.e. 
Therefore, the representation is unique in the corresponding
stochastic integration spaces.
\end{proof}

	\begin{lemma}\label{lem:BSDE}
			Suppose $\xi\in L^2(\mathcal F_{s_1})$ and predictable
		$g\in L^2(I\times\Omega,dt\otimes d\QQ)$. The equation
		\begin{align*}
			Y_t=\xi+\int_t^{s_1}g_s\ds
			-\int_t^{s_1}Z_s\,dW_s-\int_t^{s_1}\int_E U_s(\zeta)\Nt(d\zeta,ds)
			-\sum_k\int_t^{s_1}L_{k,s}d\Yt_k(s)
		\end{align*}
		has a unique solution in $\mathcal B^2(I)$, with the displayed integrands. If
		$\delta=s_1-s_0\le1$, then
		\begin{equation}\label{eq:BSDEestimate}
			\norm{(Y,Z,U,L)}_{\mathcal B^2(I)}^2
			\le C\left(\EE|\xi|^2+\delta\EE\int_I|g_t|^2\dd t\right).
		\end{equation}
		For two such equations with the same terminal condition,
		\begin{equation}\label{eq:BSDEdifference}
			\norm{\Delta(Y,Z,U,L)}_{\mathcal B^2(I)}
			\le C\delta^{1/2}\norm{g^1-g^2}_{L^2(I\times\Omega)}.
		\end{equation} 
	\end{lemma}
	\begin{proof}
		Let $A=\int_{s_0}^{s_1}g_s\,\drm s$, and $ 
	\zeta=\xi+A.$  By the Cauchy--Schwarz inequality, $
	\EE|A|^2
	\le\delta\EE\int_I|g_s|^2\,\drm s,$ 
	and hence
	\begin{equation}\label{eq:BSDEterminalbound}
		\EE|\zeta|^2
		\le2\EE|\xi|^2
		+2\delta\EE\int_I|g_s|^2\,\drm s.
	\end{equation}
	
	For $t\in I$, define $
	\mathcal R_t
	=\EE[\zeta\mid\mathcal F_t].$ 
	By Lemma~\ref{lem:martingale}, there exist
	unique predictable processes $(Z,U,L)$
	such that
	\[ 
	\mathcal R_t
	=\mathcal R_{s_0}
	+\int_{s_0}^t Z_s\,\drm W_s+\int_{s_0}^t\int_E
	U_s(\zeta)\Nt(\drm\zeta,\drm s)+\sum_{k=1}^n
	\int_{s_0}^t L_{k,s}\,\drm\Yt_k(s).
	\]
	Set $
	Y_t=\mathcal R_t-\int_{s_0}^t
	g_s\,\drm s.$ 
	Since $\mathcal R_{s_1}=\zeta$, we have
	$Y_{s_1}=\xi$. Moreover,
	\[ 
	Y_t
	=\xi+\int_t^{s_1}g_s\,\drm s
	-\int_t^{s_1}Z_s\,\drm W_s-\int_t^{s_1}\int_E
	U_s(\zeta)\Nt(\drm\zeta,\drm s)-\sum_{k=1}^n
	\int_t^{s_1}L_{k,s}\,\drm\Yt_k(s).
	\]
	Thus $(Y,Z,U,L)$ solves the
	backward equation.
	
	By the orthogonality of the stochastic
	integrals and It\^o's isometry,
	\[\|Z\|_{\mathcal H_W^2(I)}^2
		+\|U\|_{\mathcal J_N^2(I)}^2
		+\|L\|_{\mathcal J_Y^2(I)}^2=\EE|\mathcal R_{s_1}
		-\mathcal R_{s_0}|^2\le\EE|\zeta|^2.
	\]
	Furthermore, by using Doob's inequality,
	\[
		\EE\sup_{t\in I}|Y_t|^2\le2\EE\sup_{t\in I}
		|\mathcal R_t|^2
		+2\EE\sup_{t\in I}
		\left|\int_{s_0}^t g_s\,\drm s\right|^2\le8\EE|\zeta|^2
		+2\delta\EE\int_I|g_s|^2\,\drm s.
	\]
	Combining these inequalities with
	\eqref{eq:BSDEterminalbound}, we obtain
	\[
	\begin{aligned}
		\|(Y,Z,U,L)\|_{\mathcal B^2(I)}^2
		&\le9\EE|\zeta|^2
		+2\delta\EE\int_I|g_s|^2\,\drm s\\
		&\le C\left(
		\EE|\xi|^2
		+\delta\EE\int_I|g_s|^2\,\drm s
		\right),
	\end{aligned}
	\]
	which proves \eqref{eq:BSDEestimate}.
		
	To prove uniqueness, let
	$(Y^1,Z^1,U^1,L^1)$ and $(Y^2,Z^2,U^2,L^2)$
	be two solutions in $\mathcal B^2(I)$ with the same
	terminal condition $\xi$ and the same prescribed driver $g$.
	
	Since their stochastic integrals are square-integrable
	martingales, taking conditional expectations in each
	backward equation gives, for every fixed $t\in I$, 
	\[
	Y_t^i
	=
	\EE\left[
	\xi+\int_t^{s_1}g_s\,\drm s
	\,\middle|\,\mathcal F_t
	\right],
	\qquad i=1,2.
	\]
	Therefore, $Y_t^1=Y_t^2$ almost surely for each fixed $t$.
	Choose a countable dense subset of $I$ containing its
	endpoints. Outside a single null set, the two processes
	agree at every point of this subset.
	Since both processes have c\`adl\`ag paths, they agree
	at every $t\in I$ outside that same null set.
	Thus $Y^1$ and $Y^2$ are indistinguishable.
	
	Set $
	\Delta Z=Z^1-Z^2$, $
	\Delta U=U^1-U^2,$ and $
	\Delta L_k=L_k^1-L_k^2.$ 
	Subtracting the two backward equations at $t=s_0$,
	and using $Y_{s_0}^1=Y_{s_0}^2$, yields
	\[
	0=
	\int_{s_0}^{s_1}\Delta Z_s\,\drm W_s
	+\int_{s_0}^{s_1}\int_E
	\Delta U_s(\zeta)\Nt(\drm\zeta,\drm s)
	+\sum_{k=1}^n
	\int_{s_0}^{s_1}\Delta L_{k,s}\,\drm\Yt_k(s).
	\]
	By the orthogonality of these stochastic integrals
	and It\^o's isometry,
	\[ 
		0=\EE\int_I|\Delta Z_s|^2\,\drm s+\EE\int_I\int_E
		|\Delta U_s(\zeta)|^2\,\nu(\drm\zeta)\,\drm s+c\sum_{k=1}^n
		\EE\int_I|\Delta L_{k,s}|^2\,\drm s.
	\]
	Each term is nonnegative, so each term is zero.
	Since $c>0$, it follows that 
	$Z^1=Z^2$ and $ L_k^1=L_k^2,$ $\drm\otimes\drm \QQ$-a.e., 
	for every $k=1,\ldots,n$, and $
	U^1=U^2$, $\drm t\otimes\drm\QQ\otimes\nu)$-a.e.
	Therefore the solution is unique in $\mathcal B^2(I)$.
	\end{proof}
	The following proposition establishes well-posedness of the adjoint system. 
\begin{proposition}\label{prop:adjoint}
	Under Assumptions~\ref{ass:state}--\ref{ass:regularity}, equations
	\eqref{eq:scalaradjoint}--\eqref{eq:vectoradjoint}  have a unique solution satisfying
	\begin{equation}\label{eq:adjointspaces}
		(p,q,\varrho,\ell)\in\mathcal B^2([0,T]), \qquad (P,Q,R,\Pi)\in\mathcal B^2([0,T]).
	\end{equation}
	In particular,
	\begin{equation}\label{eq:Mell}
		\EE\int_0^T\bar M_{t-}^2|\ell_t|^2\dd t	\le\frac{\CM^2}{c}\norm{\ell}_{\mathcal J_Y^2}^2<\infty, 
	\end{equation}
	and $H_u\in L^2(\dd t\otimes \drm \QQ)$.
\end{proposition}

\begin{proof}
	For $\Theta=(\widetilde P,\widetilde Q,	\widetilde R,\widetilde\ell)$, define
	\begin{equation}\label{eq:Gmu}
		\begin{aligned}
			G_\mu[\Theta](t;y)
			={}&\EE'\bigg[\widetilde P_{t-}'\cdot\bar b_\mu'(t;y)
			+\tr\bigl((\widetilde Q_t')^\top\bar\sigma_\mu'(t;y)\bigr)\bigg]\\
			&+\EE'\bigg[\int_E\widetilde R_t'(\zeta)\cdot
			\bar\eta_\mu'(t,\zeta;y)\nu(\drm\zeta)
			+\bar M_{t-}'\sum_{k=1}^n\widetilde\ell_{k,t}'
			\bar\lambda_{k,\mu}'(t;y)\bigg].
		\end{aligned}
	\end{equation}
	Let $G_{\mu,1}[\Theta]=\partial_yG_\mu[\Theta]$. Then,
	\[
	\begin{aligned}
		G_{\mu,1}[\Theta](t;y)
		={}&\EE'\bigg[(\bar b_{\mu,1}'(t;y))^*\widetilde P_{t-}'
		+(\bar\sigma_{\mu,1}'(t;y))^*\widetilde Q_t'\bigg]\\
		&+\EE'\bigg[\int_E(\bar\eta_{\mu,1}'(t,\zeta;y))^*
		\widetilde R_t'(\zeta)\nu(\drm\zeta)
		+\bar M_{t-}'\sum_{k=1}^n\widetilde\ell_{k,t}'
		\bar\lambda_{k,\mu,1}'(t;y)\bigg].
	\end{aligned}
	\]
	Here, the symbol \(^*\) denotes adjoints.
	
	Let  $A_\Theta(t)^2=\EE\bigg[|\widetilde P_{t-}|^2+|\widetilde Q_t|^2+\|\widetilde R_t\|_{L^2(\nu)}^2+c|\widetilde\ell_t|^2\bigg]$.  Based on  Assumption~\ref{ass:regularity}, by using the Cauchy--Schwarz inequality,  $$
	\left|\EE'\left[\widetilde P_{t-}'\cdot\bar b_\mu'(t;y)\right]\right|\le C(1+|y|)\left(\EE|\widetilde P_{t-}|^2\right)^{1/2},$$
	and
	\[  \left|\EE'\left[\int_E\widetilde R_t'(\zeta)\cdot\bar\eta_\mu'(t,\zeta;y)\nu(d\zeta)\right]\right|\le C(1+|y|)\left(\EE\|\widetilde R_t\|_{L^2(\nu)}^2\right)^{1/2}.
	\]
	Similarly,
	\[ \left|\EE'\left[\bar M_{t-}'	\widetilde\ell_{k,t}'\bar\lambda_{k,\mu}'(t;y)	\right]\right|\le
	C\CM(1+|y|)	\left(\EE|\widetilde\ell_{k,t}|^2	\right)^{1/2}.
	\]
	The corresponding estimates for $\sigma_\mu$ follow from its bounded
	derivatives. Moreover, since $b_{\mu,1}$, $\sigma_{\mu,1}$,
	$\eta_{\mu,1}$ and $\lambda_{k,\mu,1}$ are uniformly bounded, therefore
	\begin{equation}\label{eq:Gpointwise} 
		|G_\mu[\Theta](t;y)|\le C(1+|y|)A_\Theta(t),\qquad |G_{\mu,1}[\Theta](t;y)|\le CA_\Theta(t). 
	\end{equation}

	Furthermore, by using  Theorem~\ref{thm:forward}, $\sup_{t\le T}\EE|\bar X_t|^2<\infty.$ 
	Thus, combining with \eqref{eq:Gpointwise}, 
	\begin{equation}\label{eq:Goperator} 
		\|G_\mu[\Theta](t;\bar X_{t-})\|_{L^2} +\|\bar M_{t-}G_{\mu,1}[\Theta](t;\bar X_{t-})\|_{L^2}\le C A_\Theta(t).
	\end{equation}
	By linearity, for any $\Theta^1,\Theta^2$,
	\[\|G_\mu[\Theta^1](t;\bar X_{t-})-G_\mu[\Theta^2](t;\bar X_{t-})\|_{L^2}+
	\|\bar M_{t-}(G_{\mu,1}[\Theta^1]-G_{\mu,1}[\Theta^2])(t;\bar X_{t-})\|_{L^2}\le C A_{\Theta^1-\Theta^2}(t).
	\]
	
	We next verify the square integrability of the terminal conditions and the nonhomogeneous terms of the adjoint equations.
	
	Define $a_2(t)=1+\EE|\bar X_{t-}|^2+m_2(\bar\mu_t)+\EE|\bar u_t|^2.$  By \eqref{eq:forwardmoment} and the admissibility of $\bar u$,
	\[  \int_0^T a_2(t)^2\,\drm t\le C\bigg(1+\EE\sup_{t\le T}|\bar X_t|^4+\EE\int_0^T|\bar u_t|^4\,\drm t\bigg)<\infty.
	\]
	
	Define  $g_\mu^0(t;y)=\EE'[\bar M_{t-}'\bar f_\mu'(t;y)]$, and $g_{\mu,1}^0(t;y)=\EE'[\bar M_{t-}'\bar f_{\mu,1}'(t;y)].$  By \eqref{eq:costderivatives} and $0\le\bar M\le\CM$, $
	|g_\mu^0(t;y)|+|g_{\mu,1}^0(t;y)|\le Ca_2(t)(1+|y|^2).$  Therefore, 
	\[
	\begin{aligned}
		&\EE\int_0^T|g_\mu^0(t;\bar X_{t-})|^2\,\drm t +\EE\int_0^T|\bar M_{t-}g_{\mu,1}^0(t;\bar X_{t-})|^2\,\drm t\\
		&\le C\int_0^T a_2(t)^2 \EE(1+|\bar X_{t-}|^2)^2\,\drm t\\
		&\le C\left( 1+\sup_{t\le T}\EE|\bar X_t|^4 \right)\int_0^T a_2(t)^2\,\drm t
		<\infty.
	\end{aligned}
	\]
	Furthermore, the quadratic growth conditions imply $\EE\int_0^T\left(|\bar f_t|^2+|\bar M_{t-}\bar f_x(t)|^2	\right)\drm t<\infty.$ 
	
	For the terminal conditions, using \eqref{eq:costderivatives} and the definition of $\bar\mu_T$, we have
	\[
	\begin{aligned}
		&	\left|\EE'[\bar M_T'\Phi_\mu(\bar X_T',\bar\mu_T;	\bar X_T)]\right|\\
		&\quad\le C(1+|\bar X_T|^2)\EE'\bigl[\bar M_T'(1+|\bar X_T'|^2+m_2(\bar\mu_T))
		\bigr]\\
		&\quad\le C(1+|\bar X_T|^2).
	\end{aligned}
	\]
	The same estimate holds with $\Phi_\mu$ replaced by $\Phi_{\mu,1}$.
	Hence, by the terminal conditions in \eqref{eq:scalaradjoint}--\eqref{eq:vectoradjoint},
	$|p_T|+|P_T|\le C(1+|\bar X_T|^2),$  and therefore $\EE\left(|p_T|^2+|P_T|^2\right)
	\le C\left(1+\EE|\bar X_T|^4\right)<\infty.$
	
	Let $I=[s_0,s_1]\subset[0,T]$ and set $\mathcal V(I)=\mathcal S^2(I)\times\mathcal H_W^2(I)\times\mathcal J_N^2(I)\times\mathcal J_Y^2(I).$
	For $\Theta=(\widetilde P,\widetilde Q,\widetilde R,\widetilde\ell)\in\mathcal V(I)$,
	define the scalar driver
	\[F_s[\Theta](t)=\bar f_t+\sum_{k=1}^n(\bar\lambda_k(t)-c)\widetilde\ell_{k,t}+g_\mu^0(t;\bar X_{t-})
	+G_\mu[\Theta](t;\bar X_{t-}),
	\]
	and the vector driver
	\begin{align}
		F_v[\Theta](t)=&\bar M_{t-}\bar f_x(t)+\bar b_x(t)^*\widetilde P_{t-}+\bar\sigma_x(t)^*\widetilde Q_t+\int_E\bar\eta_x(t,\zeta)^*\widetilde R_t(\zeta)\nu(d\zeta)\\
		&+\bar M_{t-}\sum_{k=1}^n\bar\lambda_{k,x}(t)\widetilde\ell_{k,t}+\bar M_{t-}\bigl(g_{\mu,1}^0+G_{\mu,1}[\Theta]\bigr)(t;\bar X_{t-}).
	\end{align}
	By the preceding estimates, the boundedness of the local coefficient derivatives, and $0\le\bar M\le\CM$, $F_s[\Theta],F_v[\Theta]\in L^2(I\times\Omega).$
	
	For $\Theta^1,\Theta^2\in\mathcal V(I)$, define  $\Delta\Theta=\Theta^1-\Theta^2$.
	By \eqref{eq:Goperator}, $$\EE\left[|F_s[\Theta^1](t)-F_s[\Theta^2](t)|^2\right]+
	\EE\left[|F_v[\Theta^1](t)-F_v[\Theta^2](t)|^2\right]\le C A_{\Delta\Theta}(t)^2.$$
	Since
	\[ \int_I A_{\Delta\Theta}(t)^2\,\drm t=\EE\int_I\bigg(|\Delta\widetilde P_{t-}|^2
	+|\Delta\widetilde Q_t|^2+\|\Delta\widetilde R_t\|_{L^2(\nu)}^2+c|\Delta\widetilde\ell_t|^2
	\bigg)\drm t\le C\|\Delta\Theta\|_{\mathcal V(I)}^2,
	\]
	we obtain
	\begin{equation}\label{eq:adjointdriverlip}
			\|F_s[\Theta^1]-F_s[\Theta^2]\|_{L^2(I\times\Omega)}
			+\|F_v[\Theta^1]-F_v[\Theta^2]\|_{L^2(I\times\Omega)}\le C\|\Theta^1-\Theta^2\|_{\mathcal V(I)}.
	\end{equation}
	
	Fix square-integrable terminal values $\xi_s,\xi_v$ at $s_1$. For each $\Theta\in\mathcal V(I)$, apply Lemma~\ref{lem:BSDE} to the scalar and vector equations with
	drivers $F_s[\Theta]$ and $F_v[\Theta]$, respectively. Denote the resulting solutions by $
	(p^\Theta,q^\Theta,\varrho^\Theta,\ell^\Theta)$, and $ (P^\Theta,Q^\Theta,R^\Theta,\Pi^\Theta)$. 
	
	Define $\Gamma_I:\mathcal V(I)\to\mathcal V(I)$, and $\Gamma_I(\Theta)=(P^\Theta,Q^\Theta,R^\Theta,\ell^\Theta).$  By \eqref{eq:BSDEdifference} and \eqref{eq:adjointdriverlip}, $
	\|\Gamma_I(\Theta^1)-\Gamma_I(\Theta^2)\|_{\mathcal V(I)}\le C(s_1-s_0)^{1/2}
	\|\Theta^1-\Theta^2\|_{\mathcal V(I)}.$
	
	Choose $\delta_0\in(0,1]$ such that $C\delta_0^{1/2}<1$. For $s_1-s_0\le\delta_0$,
	$\Gamma_I$ is a contraction on $\mathcal V(I)$. Hence it admits a unique fixed point
	$\Theta^*=(P,Q,R,\ell).$  The corresponding scalar and vector BSDE solutions satisfy 	\eqref{eq:scalaradjoint} and \eqref{eq:vectoradjoint} on $I$.
	
	Starting with the terminal values \eqref{eq:scalaradjoint}--\eqref{eq:vectoradjoint}
	at $T$, choose a partition
	\[ 0=t_0<t_1<\cdots<t_N=T, \qquad t_{j+1}-t_j\le\delta_0.
	\]
	The preceding construction gives a unique solution on $[t_{N-1},T]$. Its values at $t_{N-1}$ belong to $L^2(\mathcal F_{t_{N-1}})$ and serve as terminal conditions on
	$[t_{N-2},t_{N-1}]$. Repeating this procedure yields a solution on $[0,T]$ satisfying
	\eqref{eq:adjointspaces}. Uniqueness follows by applying the same contraction estimate
	successively on the subintervals.
	
	Finally, since $0\le\bar M\le\CM$, \[
	\EE\int_0^T\bar M_{t-}^2 	|\ell_t|^2\,\drm t\le\CM^2
	\EE\int_0^T|\ell_t|^2\,\drm t=\frac{\CM^2}{c}	\|\ell\|_{\mathcal J_Y^2([0,T])}^2
	<\infty.\] 
	Moreover, since 
	\begin{align}
		H_a(t)&=\bar f_t+\sum_k\ell_{k,t}(\bar\lambda_k(t)-c),\label{eq:Ha}\\
		H_u(t)&=\bar M_{t-}\bar f_u(t)+\bar b_u(t)^*P_{t-}
		+\bar\sigma_u(t)^*Q_t
		+\int_E\bar\eta_u(t,\zeta)^*R_t(\zeta)\nu(d\zeta)
		+\bar M_{t-}\sum_k\ell_{k,t}\bar\lambda_{k,u}(t),
		\label{eq:Hu}
	\end{align} 	then by the boundedness of the derivatives of $b,\sigma,\eta,\lambda_k$, we have 
	\[ |H_u(t)|\le C\bigg(\bar M_{t-}|\bar f_u(t)|+|P_{t-}|+|Q_t|+\|R_t\|_{L^2(\nu)}
	+\bar M_{t-}|\ell_t|\bigg).
	\]
	The quadratic growth condition implies $\EE\int_0^T	|\bar M_{t-}\bar f_u(t)|^2	\,\drm t<\infty.$ Combining these estimates with \eqref{eq:adjointspaces} and
	\eqref{eq:Mell} yields $\EE\int_0^T|H_u(t)|^2\,\drm t<\infty.$ 
\end{proof}

	The final lemma transfers the copied variation to the adjoint factor. It is used in Proposition~\ref{prop:duality} and in the terminal-cost identity.

\begin{lemma}\label{lem:copyduality}
	For an integrable contraction with a random scalar, vector, matrix, or
	$L^2(\nu)$-valued factor $A_t$,
	\begin{equation}\label{eq:copyduality}
			\EE[A_t\cdot D_\mu\phi(t;\chi,K)]
			=\EE\left[K_{t-}\EE'[A_t'\cdot\bar\phi_\mu'(t;\bar X_{t-})]\right]+\EE\left[\bar M_{t-}\chi_{t-}\cdot
			\EE'[(\bar\phi_{\mu,1}'(t;\bar X_{t-}))^*A_t']\right].
	\end{equation}
\end{lemma}

\begin{proof}
	By the definition of $D_\mu\phi$ in \eqref{eq:Dmu}, we have
	\[ \EE[A_t\cdot D_\mu\phi(t;\chi,K)]=\EE\EE'\left[K_{t-}'A_t\cdot\bar\phi_\mu(t;\bar X_{t-}')
	\right]+\EE\EE'\left[\bar M_{t-}'A_t\cdot\bigl(	\bar\phi_{\mu,1}(t;\bar X_{t-}')\chi_{t-}'
	\bigr)\right].
	\]
	By Fubini's theorem and the symmetry of the independent copies,
	\[  \EE\EE'\left[K_{t-}'A_t\cdot\bar\phi_\mu(t;\bar X_{t-}')\right]=\EE\EE'\left[
	K_{t-}A_t'\cdot\bar\phi_\mu'(t;\bar X_{t-})\right]=\EE\left[K_{t-}\EE'\left[
	A_t'\cdot\bar\phi_\mu'(t;\bar X_{t-})\right]\right].
	\]
	Similarly,
	\begin{align}
		&\EE\EE'\left[\bar M_{t-}'A_t\cdot\bigl(\bar\phi_{\mu,1}(t;\bar X_{t-}')\chi_{t-}'
		\bigr)\right]\\
		&=\EE\EE'\left[\bar M_{t-}A_t'\cdot\bigl(	\bar\phi_{\mu,1}'(t;\bar X_{t-})
		\chi_{t-}\bigr)\right]\\
		&=\EE\EE'\left[\bar M_{t-}\chi_{t-}\cdot(\bar\phi_{\mu,1}'(t;\bar X_{t-}))^*
		A_t'\right]\\
		&=\EE\left[\bar M_{t-}\chi_{t-}\cdot\EE'\left[(\bar\phi_{\mu,1}'(t;\bar X_{t-}))^*
		A_t'\right]\right].
	\end{align}
	Combining these identities gives \eqref{eq:copyduality}.
	
	We verify the integrability required for Fubini's theorem. For
	$\phi\in\{b,\sigma,\lambda_1,\ldots,\lambda_n\}$, based on 	Assumption~\ref{ass:regularity},
	\[  \EE\EE'\left[	|A_t||K_{t-}'||\bar\phi_\mu(t;\bar X_{t-}')|\right]\le C\|A_t\|_{L^2}
	\|K_{t-}\|_{L^2}\|1+|\bar X_{t-}|\|_{L^2},
	\]
	and
	\[\EE\EE'\left[|A_t|\bar M_{t-}'|\bar\phi_{\mu,1}(t;\bar X_{t-}')||\chi_{t-}'|\right]\le C\CM
	\|A_t\|_{L^2}\|\chi_{t-}\|_{L^2}.
	\]
	The same estimates hold for $\eta$ with the $L^2(\nu)$ inner product and
	corresponding norms.
	
	By \eqref{eq:forwardmoment}, \eqref{eq:variationmoment} and \eqref{eq:adjointspaces}, the right-hand sides are integrable over $[0,T]$ for the adjoint factors appearing in
	the duality formula.
	
	For the cost term $\phi=f$ with $A_t=\bar M_{t-}$, set $\bar A_2(t)=1+|\bar X_{t-}|^2
	+m_2(\bar\mu_t)+|\bar u_t|^2$. By \eqref{eq:costderivatives}, $ |\bar f_\mu(t;y)| +|\bar f_{\mu,1}(t;y)|\le C\bar A_2(t)(1+|y|^2).$  Therefore, 
	\[  \EE\EE'\left[\bar M_{t-}|K_{t-}'||\bar f_\mu(t;\bar X_{t-}')|\right]\le C\CM\EE[\bar A_2(t)]
	\|K_{t-}\|_{L^2}\|1+|\bar X_{t-}|^2\|_{L^2},
	\]
	and
	\[\EE\EE'\left[\bar M_{t-}\bar M_{t-}'|\bar f_{\mu,1}(t;\bar X_{t-}')||\chi_{t-}'|\right]\le C\CM^2\EE[\bar A_2(t)]\|\chi_{t-}\|_{L^2}	\|1+|\bar X_{t-}|^2\|_{L^2}.
	\]
	These quantities are integrable in time by the forward moment estimates and the
	admissibility of $\bar u$. The terminal cost is treated in the same way, with
	the control variable omitted.
\end{proof}

	\begin{proposition}\label{prop:duality}
	For every baseline $\bar u\in\Uad$ and competitor $v\in\Uad$, with $h=v-\bar u$,  then 
	\begin{equation}\label{eq:duality}
		J_1(h)=\EE\int_0^T H_u(t)\cdot h_t\dd t.
	\end{equation}
\end{proposition}

\begin{proof}
	For $k=1,\ldots,n$, let  $J_{k,t}=\left(\frac{\bar\lambda_k(t)}c-1\right)K_{t-}
	+\frac{\bar M_{t-}}cD\lambda_k(t).$  Then, by \eqref{eq:K},
	\[
	\drm K_t=\sum_{k=1}^nJ_{k,t}\,\drm\Yt_k(t).
	\]
	Since the observation processes have no common jumps with the state Poisson random measure, the quadratic covariation of $p$ and $K$ is
	\[
	\drm [p,K]_t
	=\sum_{k=1}^n\ell_{k,t}J_{k,t}\,\drm Y_k(t).
	\]
	Its predictable compensator is
	\begin{equation}\label{eq:pKbracket}
		\begin{aligned}
			\drm \langle p,K\rangle_t
			=c\sum_{k=1}^n\ell_{k,t}J_{k,t}\,\drm t=\sum_{k=1}^n\ell_{k,t}\left[(\bar\lambda_k(t)-c)K_{t-}+\bar M_{t-}D\lambda_k(t)
			\right]\drm t.
		\end{aligned}
	\end{equation}
	
	Applying It\^o's product formula to
	$p_tK_t$ and taking expectations gives
	\[\EE[p_TK_T]=\EE\int_0^T\bigg[-K_{t-}H_a(t)-K_{t-}\EE'[H_\mu(t,\omega';\bar X_{t-})]+c\sum_{k=1}^n\ell_{k,t}J_{k,t}
	\bigg]\drm t.
	\]
	Since $H_a(t)=\bar f_t+\sum_{k=1}^n\ell_{k,t}(\bar\lambda_k(t)-c)$, we obtain
	\begin{equation}\label{eq:pKidentity} 
		\EE[p_TK_T]=\EE\int_0^T\bigg[-K_{t-}\bar f_t-K_{t-}\EE'[H_\mu(t,\omega';\bar X_{t-})]+\bar M_{t-}\sum_{k=1}^n\ell_{k,t}D\lambda_k(t)\bigg]\drm t.
	\end{equation}
	
	Similarly, by \eqref{eq:chi} and \eqref{eq:vectoradjoint}, the predictable quadratic covariation of $P$ and $\chi$ is given by
	\[ \drm \langle P,\chi\rangle_t=\bigg[\tr(Q_t^\top D\sigma_t)+\int_E R_t(\zeta)\cdot
	D\eta_t(\zeta)\nu(d\zeta)\bigg]\drm t.
	\]
	Indeed, $\chi$ has no martingale integral with respect to $\Yt_k$. Applying It\^o's product formula to $P_t\cdot\chi_t$ yields
	\begin{align}\label{eq:Pchiidentity} 
		\EE[P_T\cdot\chi_T]=\EE\int_0^T\bigg[&-\chi_{t-}\cdot H_x(t)-\bar M_{t-}\chi_{t-}\cdot
		\EE'[H_{\mu,1}(t,\omega';\bar X_{t-})]+P_{t-}\cdot Db_t\\
		&+\tr(Q_t^\top D\sigma_t)+\int_E R_t(\zeta)\cdot D\eta_t(\zeta)\nu(d\zeta)\bigg]\drm t. 
	\end{align}
	
	Let $\mathcal C_\mu(t)=P_{t-}\cdot D_\mu b(t;\chi,K)+\tr\bigl(Q_t^\top D_\mu\sigma(t;\chi,K)\bigr)+\int_E R_t(\zeta)\cdot D_\mu\eta(t;\chi,K,\zeta)\nu(d\zeta)+\bar M_{t-}\sum_{k=1}^n\ell_{k,t}D_\mu\lambda_k(t;\chi,K).$ 
	By the definitions of the coefficient variations and the Hamiltonian function,
	\[
	\begin{aligned}
		&P_{t-}\cdot Db_t
		+\tr(Q_t^\top D\sigma_t)
		+\int_E R_t(\zeta)\cdot
		D\eta_t(\zeta)\nu(d\zeta) +
		\bar M_{t-}\sum_{k=1}^n
		\ell_{k,t}D\lambda_k(t)\\
		&=
		\bigl(H_x(t)-\bar M_{t-}\bar f_x(t)\bigr)
		\cdot\chi_{t-} +
		\bigl(H_u(t)-\bar M_{t-}\bar f_u(t)\bigr)
		\cdot h_t
		+\mathcal C_\mu(t).
	\end{aligned}
	\]
	
	Moreover, applying
	Lemma~\ref{lem:copyduality} to the
	four terms of $\mathcal C_\mu(t)$ gives
	\[ 
	\begin{aligned}
		\EE[\mathcal C_\mu(t)]
		={}&\EE\bigg[K_{t-}\EE'\big[
		H_\mu(t,\omega';\bar X_{t-})
		-\bar M_{t-}'\bar f_\mu'(t;\bar X_{t-})\big]\bigg]\\
		&+\EE\bigg[\bar M_{t-}\chi_{t-}\cdot\EE'\big[
		H_{\mu,1}(t,\omega';\bar X_{t-})
		-\bar M_{t-}'\bar f_{\mu,1}'(t;\bar X_{t-})\big]\bigg].
	\end{aligned}
	\]
	Also by using 	Lemma~\ref{lem:copyduality} ,
	\[  \EE\bigg[
	K_{t-}\EE'\big[
	\bar M_{t-}'\bar f_\mu'
	(t;\bar X_{t-})
	\big]\bigg] +
	\EE\bigg[
	\bar M_{t-}\chi_{t-}\cdot
	\EE'\big[
	\bar M_{t-}'\bar f_{\mu,1}'
	(t;\bar X_{t-})
	\big]\bigg] =\EE\left[
	\bar M_{t-}D_\mu f(t;\chi,K)
	\right]. 
	\]
	Therefore, 
	\[ 
	\EE[\mathcal C_\mu(t)]
	=\EE\bigg[ K_{t-}\EE'
	[H_\mu(t,\omega';\bar X_{t-})]+\bar M_{t-}\chi_{t-}\cdot
	\EE'[H_{\mu,1}
	(t,\omega';\bar X_{t-})]-\bar M_{t-}
	D_\mu f(t;\chi,K)
	\bigg]. 
	\]
	
	Adding \eqref{eq:pKidentity} and
	\eqref{eq:Pchiidentity}, and using
	the preceding identities, we obtain
	\[  \EE[p_TK_T+P_T\cdot\chi_T]=\EE\int_0^T\bigg[
	H_u(t)\cdot h_t
	-K_{t-}\bar f_t-\bar M_{t-}
	\bigl(
	\bar f_x(t)\cdot\chi_{t-}
	+\bar f_u(t)\cdot h_t
	+D_\mu f(t;\chi,K)
	\bigr)
	\bigg]\drm t. 
	\]
	Since $
	Df_t
	=\bar f_x(t)\cdot\chi_{t-}
	+\bar f_u(t)\cdot h_t
	+D_\mu f(t;\chi,K),$ 
	and $K_t=K_{t-}$,
	$\bar M_t=\bar M_{t-}$ for
	$dt\otimes d\QQ$-almost every $(t,\omega)$,
	we conclude that
	\begin{equation}\label{eq:dualityintermediate} 
		\EE[p_TK_T+P_T\cdot\chi_T]=\EE\int_0^T
		\left(
		H_u(t)\cdot h_t
		-K_t\bar f_t
		-\bar M_tDf_t
		\right)\drm t. 
	\end{equation}
	Finally, by \eqref{eq:terminalduality} and \eqref{eq:J1},
	\[ 
	J_1(h)
	= \EE[p_TK_T+P_T\cdot\chi_T]+\EE\int_0^T
	(K_t\bar f_t+\bar M_tDf_t)\,\drm t
	=\EE\int_0^T H_u(t)\cdot h_t\,\drm t, 
	\]
	which proves \eqref{eq:duality}.
	
	We now justify the expectations in the
	preceding product formulas.
	By \eqref{eq:variationmoment} and
	\eqref{eq:adjointspaces}, the variation
	processes and the adjoint processes
	belong to $\mathcal S^2$, and their
	martingale integrands are square
	integrable. In particular, $
	\EE\int_0^T\int_E
	|R_t(\zeta)\cdot D\eta_t(\zeta)|
	\nu(d\zeta)\,\drm t \le
	\|R\|_{\mathcal J_N^2}
	\|D\eta\|_{\mathcal J_N^2}<\infty, $
	and $
	\EE\int_0^T
	c\sum_{k=1}^n|\ell_{k,t}J_{k,t}|
	\,\drm t\le
	\|\ell\|_{\mathcal J_Y^2}
	\|J\|_{\mathcal J_Y^2}<\infty.$
	
	For any $V\in\mathcal S^2$ and
	square-integrable martingale $N$,
	using	the Cauchy--Schwarz inequality yields
	\[  \EE\left(
	\int_0^T|V_{t-}|^2\,\drm[N]_t
	\right)^{1/2}\le
	\left(
	\EE\sup_{t\le T}|V_t|^2
	\right)^{1/2}
	\left(\EE[N]_T\right)^{1/2}
	<\infty. 
	\]
	Then the Burkholder--Davis--Gundy inequality implies that the stochastic
	integrals appearing in the product formulas are uniformly integrable martingales and have expectation zero.
	
	Furthermore, $\EE\sup_{t\le T}|p_tK_t|\le\|p\|_{\mathcal S^2}\|K\|_{\mathcal S^2}$, and $
	\EE\sup_{t\le T}|P_t\cdot\chi_t|\le\|P\|_{\mathcal S^2}\|\chi\|_{\mathcal S^2}.$ 
	Together with the preceding estimates, these bounds justify passage to the
	limit in the localized product formulas.
\end{proof}
	
 Since admissible controls are observation-predictable, their pairing with $H_u$ depends only on the observation-predictable projection introduced below.  Let $\mathcal P^Y$ denote the predictable $\sigma$-field associated with $\mathbb F^Y$.
Since $H_u\in L^2(dt\otimes d\mathbb Q)$, we define
\begin{equation}\label{eq:projection}
	\widehat H_u =\mathbb E^{dt\otimes\mathbb Q}[H_u\mid\mathcal P^Y].
\end{equation}
Then, for every $\mathbb F^Y$-predictable $\xi\in L^2(dt\otimes d\mathbb Q)$,
\begin{equation}\label{eq:projectionduality}
	\mathbb E\int_0^T H_u(t)\cdot\xi_t\,\mathrm dt
	=\mathbb E\int_0^T \widehat H_u(t)\cdot\xi_t\,\mathrm dt.
\end{equation}
We now state the necessary conditions of the stochastic maximum principle.
\begin{theorem}\label{thm:SMP}
	Suppose that Assumptions~\ref{ass:state}--
	\ref{ass:regularity} hold. Let
	$\bar u\in\cU_{ad}$ be an optimal control.
	Then, for $dt\otimes d\QQ$-almost every
	$(t,\omega)$,
	\begin{equation}\label{eq:normalineq}
		\widehat H_u(t)\cdot(w-\bar u_t)\ge0,
		\qquad \forall w\in U.
	\end{equation}
	Equivalently,
	\begin{equation}\label{eq:normalcone}
		-\widehat H_u(t)\in N_U(\bar u_t),
	\end{equation}
	where $
	N_U(x)=\{z\in\R^m:
	z\cdot(w-x)\le0,\ \forall w\in U\}.$ 
\end{theorem}

\begin{proof}
	By Theorem~\ref{thm:costderivative},
	Proposition~\ref{prop:duality}, and
	\eqref{eq:projectionduality}, for any $v\in\cU_{ad}$,
	\[
	\EE\int_0^T
	\widehat H_u(t)\cdot(v_t-\bar u_t)
	\,\drm t\ge0.
	\]
	Fix $w\in U$ and $B\in\mathcal P^Y$.
	Taking $
	v_t=w\ind_B(t)+\bar u_t\ind_{B^c}(t),$ 
	we obtain $
	\EE\int_0^T
	\ind_B(t)\widehat H_u(t)
	\cdot(w-\bar u_t)\,\drm t\ge0.$ 
	
	Since $
	g_w(t)=\widehat H_u(t)
	\cdot(w-\bar u_t)$ 
	is predictable and integrable,
	we may take $B=\{g_w<0\}$.
	Therefore,  $
	\EE\int_0^T
	\ind_{\{g_w<0\}}g_w(t)\,\drm t
	\ge0.$ 
	Since the integrand is nonpositive,
	we have $
	g_w(t)\ge0$, $\drm t\otimes \drm \QQ$-a.e. 
	
	Let $\{w_j\}_{j\ge1}$ be a countable
	dense subset of $U$.
	There exists a common
	$\drm t\otimes \drm \QQ$-null set outside which $
	\widehat H_u(t)\cdot(w_j-\bar u_t)\ge0,$ for any $j\geq 1$.
	By continuity in $w$, $
	\widehat H_u(t)\cdot(w-\bar u_t)\ge0,$ $\forall w\in U.$ 
	
	This proves \eqref{eq:normalineq}, and \eqref{eq:normalcone} follows
	from the definition of $N_U$.
\end{proof}

	\begin{corollary}\label{cor:stationarity}
	Under the assumptions of
	Theorem~\ref{thm:SMP}, $
	\widehat H_u(t)=0$ on $\{\bar u_t\in\Int(U)\}$, $
	\drm t\otimes \drm \QQ$-a.e.
\end{corollary}

\begin{proof}
	For $\bar u_t\in\Int(U)$, taking
	$w=\bar u_t\pm\delta e_i$ in
	\eqref{eq:normalineq}, where
	$\{e_i\}_{i=1}^m$ is the canonical
	basis of $\R^m$ and $\delta>0$
	is sufficiently small, gives $
	\widehat H_u(t)\cdot e_i=0,$ for $ i=1,\ldots,m.$  Hence $\widehat H_u(t)=0$.
\end{proof}
	
\section{A quadratic example with distribution-dependent observations}
\label{sec:example}

We consider a stylized monitoring and signal-following problem with one observation channel whose later intensity depends on the unconditional physical variance of an accumulated action. The running cost combines a quadratic penalty on the action, a reward for alignment with the hidden signal, and a fee for early observation effort. Early activation can change this variance and thereby contribute a first-order term to the cost. We identify the optimal zero-activation benchmark and compute a first-order activation threshold within the policy family in \eqref{eq:feedbackpolicy}, under the coercivity and nonbinding conditions stated below.

\subsection{The model and the zero-activation benchmark}

Fix $0<t_0<t_1<T$, and set $L=t_1-t_0$ and $D=T-t_1$.
Use the reference space and natural observation filtration of
Section~\ref{sec:model}, with a Brownian motion $W$ and an independent
rate-$c$ Poisson process $Y$ under $\QQ$. Assume
$0<\theta<1$, $\lambda_0>0$, $|\kappa|<\lambda_0$,
$\beta\in\R$, $k\geq0$ and $V\geq1$.
The control $u=(\alpha,v)$ is $\FF^Y$-predictable with values in
$[0,1]\times[-V,V]$, and the state $X=(S,Z)$ satisfies
\begin{equation}\label{eq:feedbackstate}
	\left\{
	\begin{aligned}
			\mathrm dS_t=&\ind_{\{t<t_0\}}\dd W_t,\\
		\mathrm dZ_t=&\ind_{\{t<t_0\}}(\beta Z_t+v_t)\dd t,\\
		S_0=&Z_0=0.
	\end{aligned}
	\right.
\end{equation}
Thus $S_t=W_{t\wedge t_0}$ and both state components are fixed after
$t_0$. In the general notation, $d=2$, $d_W=n=1$, $\eta=0$,  $
b(t,(s,z),\mu,(\alpha,v))
=\ind_{\{t<t_0\}}(0,\beta z+v)^\top$,  and $
\sigma(t,(s,z),\mu,(\alpha,v))
=\ind_{\{t<t_0\}}(1,0)^\top.$ 

Let $w(t)=e^{\beta(t_0-t)}$, choose
$R>V\int_0^{t_0}w(t)\dd t$, and define
\[
\phi_R(z)=
\begin{cases}
	z^2,&|z|\leq R,\\
	2R|z|-R^2,&|z|>R.
\end{cases}
\]
For $\mu\in\mathcal P_1(\R^2)$, set
\begin{equation}\label{eq:varianceextension}
	m(\mu)=\int y_z\mu(\mathrm dy),\qquad
	\mathcal V_R(\mu)=\int\phi_R(y_z)\mu(\mathrm dy)
	-\phi_R(m(\mu)).
\end{equation}
Since $\phi_R$ is convex and $C^1$, with
$|\phi_R'|\leq2R$ and $\operatorname{Lip}(\phi_R')\leq2$,
$\mathcal V_R$ is finite and nonnegative on $\mathcal P_1(\R^2)$.
Moreover,
\[
|\mathcal V_R(\mu)-\mathcal V_R(\widehat\mu)|
\leq2R\,W_1(\mu,\widehat\mu)
+2R|m(\mu)-m(\widehat\mu)|
\leq4R\,W_1(\mu,\widehat\mu),
\]
and its normalized linear functional derivative is
$(\mathcal V_R)_\mu(\mu;y)
=\phi_R(y_z)-\phi_R'(m(\mu))y_z$.
For $r_t=t\wedge t_0$, solving \eqref{eq:feedbackstate} gives $
Z_t=\int_0^{r_t}e^{\beta(r_t-s)}v_s\dd s$, thus $
|Z_t|\leq V\int_0^{r_t}e^{\beta(r_t-s)}\dd s
\leq V\int_0^{t_0}w(s)\dd s<R.$

The same bound holds under every physical probability, and also bounds
the physical mean. Consequently,
\begin{equation}\label{eq:actualvariance}
	\mathcal V_R(\mu_t^u)
	=\EE^{\PP_\kappa^u}[(Z_t^u)^2]
	-(\EE^{\PP_\kappa^u}Z_t^u)^2
	=\operatorname{Var}_{\PP_\kappa^u}(Z_t^u).
\end{equation}

Set $q(s)=1+\theta\tanh s$ and define
\begin{equation}\label{eq:feedbackintensity}
	\lambda_\kappa(t,(s,z),\mu,(\alpha,v))
	=
	\begin{cases}
		\alpha q(s),&0\leq t<t_0,\\
		[\lambda_0+\kappa\tanh\mathcal V_R(\mu)]q(s),
		&t_0\leq t<t_1,\\
		0,&t_1\leq t\leq T.
	\end{cases}
\end{equation}
The early interval is inactive when $\alpha=0$, whereas the middle
multiplier lies in
$[\lambda_0-|\kappa|,\lambda_0+|\kappa|]\subset(0,\infty)$.
Choose $c\geq(1+\theta)\max\{1,2\lambda_0\}$.
Then $0\leq\lambda_\kappa\leq c$, and
\eqref{eq:weightedlaw} defines $\PP_\kappa^u$.
Since $\mu_t^u=\mu_{t_0}^u$ for $t\geq t_0$, the middle multiplier
$\ell_\kappa^u=\lambda_0+\kappa\tanh\mathcal V_R(\mu_{t_0}^u)$
is deterministic.

The objective is
\begin{equation}\label{eq:feedbackcost}
	J_\kappa(u)=\EE^{\PP_\kappa^u}\left[
	\int_0^{t_0}\left(\frac12v_t^2-v_tS_t+k\alpha_t\right)\dd t
	+\int_{t_1}^T\left(\frac12v_t^2-v_t\tanh S_t\right)\dd t
	\right].
\end{equation}
Thus, in the general notation, $\Phi=0$ and $
f(t,(s,z),\mu,(\alpha,v))
=\ind_{\{t<t_0\}}\left(\frac12v^2-vs+k\alpha\right)
+\ind_{\{t_1<t\leq T\}}\left(\frac12v^2-v\tanh s\right).$

These coefficients satisfy
Assumptions~\ref{ass:state}--\ref{ass:regularity}. Explicitly,
$b,\sigma$ are affine and globally Lipschitz, the control set is compact,
and $f$ and its derivatives have the required growth.
The intensity is bounded and globally Lipschitz in the state, control
and measure variables. Its measure derivative vanishes outside
$[t_0,t_1)$; on that interval,
\begin{equation}\label{eq:feedbackmeasurederivative}
	(\lambda_\kappa)_\mu(t,(s,z),\mu,u;y)
	=\kappa q(s)\sech^2(\mathcal V_R(\mu))
	[\phi_R(y_z)-\phi_R'(m(\mu))y_z].
\end{equation}
This kernel vanishes at $y=0$, grows at most linearly in $y$, and has
bounded continuous evaluation derivative
\[
(\lambda_\kappa)_{\mu,1}(t,(s,z),\mu,u;y)
=\kappa q(s)\sech^2(\mathcal V_R(\mu))
\bigl(0,\phi_R'(y_z)-\phi_R'(m(\mu))\bigr)^\top.
\]
The Lipschitz bounds for $m$, $\mathcal V_R$ and $\phi_R'$ give the
required measure continuity. The deterministic changes at $t_0,t_1$
are consistent with Borel regularity in time.

For a fixed middle multiplier $\ell>0$, let $H=\tanh W_{t_0}$. And,
on an auxiliary space, take
$N^\ell\mid H\sim\operatorname{Poisson}(\ell L(1+\theta H))$.
We denote $\EE_\ell$ as the  expectation under this joint law and define
\begin{equation}\label{eq:poissonweights}
	p_n(\ell)
	=\EE\left[e^{-\ell L(1+\theta H)}
	\frac{[\ell L(1+\theta H)]^n}{n!}\right],
	\qquad n\in\mathbb N_0.
\end{equation}
Set $g_\ell(n)=\EE_\ell[H\mid N^\ell=n]$ and
$v_n(\ell)=\operatorname{Var}_\ell(H\mid N^\ell=n)$.
All $p_n(\ell)$ are positive. For
$\mathcal D_V=\{g:\mathbb N_0\to[-V,V]\}$, let
\[
\mathcal C(\ell,g)
=D\EE_\ell\left[\frac12g(N^\ell)^2-Hg(N^\ell)\right],
\qquad
\mathcal R(\ell)=\inf_{g\in\mathcal D_V}\mathcal C(\ell,g).
\]
The following calculation uses posterior identities for Poisson
observations; see also \citet{dytso2020estimation}.

\begin{lemma}\label{lem:latercost}
	The unique minimizer is $g_\ell$, and
	\begin{equation}\label{eq:latervalue}
		\mathcal R(\ell)=-\frac D2\sum_{n\geq0}p_n(\ell)g_\ell(n)^2.
	\end{equation}
	Moreover, $\mathcal R\in C^2((0,\infty))$ and
	\begin{equation}\label{eq:latervaluederivative}
		\mathcal R'(\ell)
		=-\frac{DL\theta^2}{2}
		\sum_{n\geq0}p_n(\ell)
		\frac{v_n(\ell)^2}{1+\theta g_\ell(n)}<0.
	\end{equation}
	For $g_0=g_{\lambda_0}$,
	\begin{equation}\label{eq:laterdecisiongap}
		0\leq\mathcal C(\ell,g_0)-\mathcal R(\ell)
		\leq\frac{DL^2\theta^2}{2}(\ell-\lambda_0)^2.
	\end{equation}
\end{lemma}

\begin{proof}
	Write $g_n=g_\ell(n)$, $v_n=v_n(\ell)$ and
	$\EE_{\ell,n}[\cdot]=\EE_\ell[\cdot\mid N^\ell=n]$.
	Conditional completion of the square gives
	\[
	\mathcal C(\ell,g)
	=-\frac D2\sum_{n\geq0}p_n(\ell)g_n^2
	+\frac D2\sum_{n\geq0}p_n(\ell)[g(n)-g_n]^2.
	\]
	Since $|g_n|\leq1\leq V$ and $p_n(\ell)>0$, this proves
	\eqref{eq:latervalue} and uniqueness.
	
	Let $q_H=1+\theta H$ and $\bar q_n=1+\theta g_n$.
	For bounded $\varphi$,
	\[
	\EE_{\ell,n}[\varphi(H)]
	=\frac{\EE[\varphi(H)q_H^ne^{-\ell Lq_H}]}
	{\EE[q_H^ne^{-\ell Lq_H}]}.
	\]
	Then, by differentiating this ratio, we have 
	$\partial_\ell\EE_{\ell,n}[\varphi(H)]
	=-L\theta\operatorname{Cov}_{\ell,n}(\varphi(H),H)$.
	Also, $g_{n+1}=\EE_{\ell,n}[Hq_H]/\bar q_n$. Hence
	\begin{equation}\label{eq:posterioridentities}
		\partial_\ell g_n=-L\theta v_n,\qquad
		g_{n+1}-g_n=\frac{\theta v_n}{\bar q_n}.
	\end{equation}
	
	For every compact $K\subset(0,\infty)$, also by differentiating
	\eqref{eq:poissonweights}, we obtain that there exist  constants $C_K,B_K>0$ such that
	\[
	\sup_{\ell\in K}|p_n^{(j)}(\ell)|
	\leq C_K(1+n)^2\frac{B_K^n}{n!},
	\qquad j=0,1,2.
	\]
	Moreover, $|g_n|\leq1$, $|\partial_\ell g_n|\leq L\theta$, and
	$\partial_\ell^2g_n
	=(L\theta)^2\EE_{\ell,n}[(H-g_n)^3]$ is bounded by
	$8(L\theta)^2$. Thus
	\[
	\sum_{n\geq0}\sup_{\ell\in K}
	\left|\partial_\ell^j\bigl(p_n(\ell)g_n^2\bigr)\right|
	<\infty,\qquad j=0,1,2.
	\]
	This proves $\mathcal R\in C^2((0,\infty))$ and justifies the
	following differentiation and reindexing. Since
	$p_n'=L(p_{n-1}\bar q_{n-1}-p_n\bar q_n)$, with the first term
	omitted for $n=0$, \eqref{eq:posterioridentities} gives
	\begin{align*}
		\frac{\mathrm d}{\mathrm d\ell}\sum_{n\geq0}p_ng_n^2
		=L\sum_{n\geq0}p_n\bar q_n(g_{n+1}^2-g_n^2)
		-2L\theta\sum_{n\geq0}p_ng_nv_n=L\sum_{n\geq0}p_n\bar q_n(g_{n+1}-g_n)^2
		=L\theta^2\sum_{n\geq0}p_n\frac{v_n^2}{\bar q_n}.
	\end{align*}
	Each posterior has a strictly positive density relative to the
	nondegenerate law of $H$, so $v_n>0$.
	Together with $\bar q_n\geq1-\theta>0$, this proves
	\eqref{eq:latervaluederivative}.
	
	Finally, \eqref{eq:posterioridentities} and $v_n\leq1$ imply
	$|g_\ell(n)-g_0(n)|\leq L\theta|\ell-\lambda_0|$. Thus, 
	\[
	0\leq\mathcal C(\ell,g_0)-\mathcal R(\ell)
	=\frac D2\sum_{n\geq0}p_n(\ell)[g_0(n)-g_\ell(n)]^2
	\leq\frac{DL^2\theta^2}{2}(\ell-\lambda_0)^2.
	\]
\end{proof}

Suppose now that $\alpha=0$ on $[0,t_0)$.
Then $\PP_\kappa^u(Y_{t_0}=0)=1$. By the natural-filtration
assumption and $\PP_\kappa^u\ll\QQ$, $\mathcal F_t^Y$ is
$\PP_\kappa^u$-trivial for $t\leq t_0$.
Thus $v_t=v^{\mathrm d}(t)$
$\mathrm dt\otimes\mathrm d\PP_\kappa^u$-a.e., where
$v^{\mathrm d}(t)=\EE^{\PP_\kappa^u}v_t$.
It follows that $Z_{t_0}=\int_0^{t_0}w(t)v^{\mathrm d}(t)\dd t$
is deterministic under $\PP_\kappa^u$, so
$\mathcal V_R(\mu_{t_0}^u)=0$ and the middle multiplier equals
$\lambda_0$.

Therefore, for $H=\tanh W_{t_0}$ and $N=Y_{t_1}-Y_{t_0}$,  
$N\mid H\sim\operatorname{Poisson}(\lambda_0L(1+\theta H))$.
The middle-arrival likelihood depends on $H$ only through
$(1+\theta H)^Ne^{-\lambda_0L(1+\theta H)}$, and there are no
arrivals after $t_1$. Hence
$\EE^{\PP_\kappa^u}[H\mid\mathcal F_t^Y]=g_0(N)$ for
$t\in[t_1,T]$. Conditional completion of the square gives
\[
\EE^{\PP_\kappa^u}\int_{t_1}^T
\left(\frac12v_t^2-v_tH\right)\dd t
=\mathcal R(\lambda_0)
+\frac12\EE^{\PP_\kappa^u}\int_{t_1}^T[v_t-g_0(N)]^2\dd t.
\]
Since $W$ remains Brownian under $\PP_\kappa^u$,
$\EE^{\PP_\kappa^u}\int_0^{t_0}v_tS_t\dd t=0$. Therefore,
\[
J_\kappa(u)-\mathcal R(\lambda_0)
=\frac12\int_0^{t_0}|v^{\mathrm d}(t)|^2\dd t
+\frac12\EE^{\PP_\kappa^u}\int_{t_1}^T[v_t-g_0(N)]^2\dd t
\geq0.
\]
Obviously, the equality  holds for $\alpha=0$ and
$v_t=g_0(N)\ind_{\{t_1<t\leq T\}}$. Thus the full zero-activation
value is
\begin{equation}\label{eq:feedbackbaselinevalue}
	J_{\mathrm{base}}
	:=\inf_{\substack{u\in\mathcal U_{\mathrm{ad}}\\
			\alpha=0\text{ on }[0,t_0)}}J_\kappa(u)
	=\mathcal R(\lambda_0).
\end{equation}

\subsection{Activation and a nonzero distribution contribution}

Let $\Delta_{t_0}=\{(t,s):0\leq s\leq t\leq t_0\}$,
$\mathcal A=\{a:[0,t_0]\to[0,1]:a\text{ is Borel measurable}\}$
and $\Psi_V=C(\Delta_{t_0};[-V,V])$.
For $(a,\psi,g)\in\mathcal A\times\Psi_V\times\mathcal D_V$,
define, for $0\leq\varepsilon\leq1$,
\begin{equation}\label{eq:feedbackpolicy} 
		\alpha_t^{\varepsilon;a,\psi,g}
		=\varepsilon a(t)\ind_{\{t\leq t_0\}},\qquad
		v_t^{\varepsilon;a,\psi,g}
		=\psi(t,\tau)\ind_{\{\tau<t\leq t_0\}}
		+g(N)\ind_{\{t_1<t\leq T\}}, 
\end{equation}
where $\tau=\inf\{t:Y_t\geq1\}$ and $N=Y_{t_1}-Y_{t_0}$.
Each term is zero outside its indicated domain.
The action is adapted and left-continuous, while $\alpha$ is
deterministic and Borel measurable; hence these policies are admissible.

In \eqref{eq:feedbackpolicy}, the final action uses only the later count \(N\). This restriction isolates the contribution of variance feedback to the later decision cost by excluding direct use of the early observation history in that action.

Set $\bar u^\psi=u^{0;a,\psi,g_0}$, independently of $a$.
Its likelihood is
\[
\frac{\mathrm d\PP_\kappa^{\bar u^\psi}}{\mathrm d\QQ}
=e^{cT-\lambda_0L(1+\theta H)}
\left(\frac{\lambda_0(1+\theta H)}c\right)^N
\ind_{\{Y_{t_0}=0,\;Y_T-Y_{t_1}=0\}},
\]
which is independent of $\psi,\kappa$ and gives
$J_\kappa(\bar u^\psi)=J_{\mathrm{base}}$.
The event $\{\tau<t_0\}$ is baseline-null but has
$\QQ$-probability $1-e^{-ct_0}>0$, so different responses specify
different complete reference-space baselines.
For general $g$,
\[
J_\kappa(u^{0;a,\psi,g})
=J_{\mathrm{base}}
+\frac D2\sum_{n\geq0}p_n(\lambda_0)[g(n)-g_0(n)]^2.
\]
Therefore, by keeping $g=g_0$, we obtain the feasible variation
$u^{\varepsilon;a,\psi,g_0}=\bar u^\psi+\varepsilon h$ with
$h=(a\ind_{[0,t_0]},0)$.

Let $r(s)=\EE[W_s\tanh W_s]$ and
$z_\psi(s)=\int_s^{t_0}w(t)\psi(t,s)\dd t$, and set
\begin{equation}\label{eq:feedbackquantities}
	Q(a,\psi)=\int_0^{t_0}a(s)z_\psi(s)^2\dd s,\qquad
	G=-\mathcal R'(\lambda_0)>0.
\end{equation}
Then $Z_{t_0}=z_\psi(\tau)\ind_{\{\tau<t_0\}}$.
Gaussian integration by parts gives
$r\in C([0,t_0])$, $r(0)=0$ and
$r(s)=s\EE[\sech^2(W_s)]>0$ for $s>0$.
Symmetry and the Brownian martingale property also give
$\EE q(W_s)=1$ and $\EE[W_tq(W_s)]=\theta r(s)$ for
$s\leq t\leq t_0$.

\begin{proposition}\label{prop:feedbackactivation}
	Define
	\begin{equation}\label{eq:feedbackindex}
		\Gamma_{\kappa,\psi}(s)
		=k+\int_s^{t_0}
		\left(\frac12\psi(t,s)^2-\theta r(s)\psi(t,s)\right)\dd t
		-\kappa Gz_\psi(s)^2.
	\end{equation}
	Fix $\kappa$. As $\varepsilon\downarrow0$, the expansions
	\begin{align}
		J_\kappa(u^{\varepsilon;a,\psi,g_0})
		&=J_{\mathrm{base}}
		+\varepsilon\int_0^{t_0}a(s)\Gamma_{\kappa,\psi}(s)\dd s
		+O(\varepsilon^2),\label{eq:feedbackuniform}\\
		\inf_{g\in\mathcal D_V}J_\kappa(u^{\varepsilon;a,\psi,g})
		&=J_{\mathrm{base}}
		+\varepsilon\int_0^{t_0}a(s)\Gamma_{\kappa,\psi}(s)\dd s
		+O(\varepsilon^2).\label{eq:feedbackoptimizedlate}
	\end{align}
	hold uniformly for $(a,\psi)\in\mathcal A\times\Psi_V$.
	
	Take $h=(a\ind_{[0,t_0]},0)$. Then
	\begin{equation}\label{eq:feedbackdifference}
		DJ_\kappa(\bar u^\psi;h)-DJ_0(\bar u^\psi;h)
		=-\kappa GQ(a,\psi),
	\end{equation}
	where $DJ_0$ denotes the directional derivative for the case $\kappa=0$.
	
	The induced intensity variation satisfies
	\begin{equation}\label{eq:feedbackratevariation}
		D\lambda_\kappa(t)=\kappa q(S_{t_0})Q(a,\psi),
		\qquad t_0<t<t_1,
	\end{equation}
	which is non-zero whenever $\kappa\neq0$ and $Q(a,\psi)>0$.
\end{proposition}

\begin{proof}
	Write $u^\varepsilon=u^{\varepsilon;a,\psi,g}$,
	$\PP_\varepsilon=\PP_\kappa^{u^\varepsilon}$,
	$\EE_\varepsilon=\EE^{\PP_\varepsilon}$ and
	$\mu^\varepsilon=\mu^{u^\varepsilon}$.
	Throughout, $0\leq\varepsilon\leq1$ and $\kappa$ is fixed.
	All constants below depend only on the fixed model parameters,
	not on $(a,\psi,g)$.
	The early action is fixed on the reference space, the state
	coefficients have no measure argument, and actions after $t_0$
	do not affect $X$. Hence $X^{u^\varepsilon}=\bar X^\psi$ and
	$\chi=0$.
	
	Under $\QQ$, the likelihood factors over the three intervals,
	with unit conditional means given $W$. The later factors also have
	unit conditional means given $W$ and the early history.
	Thus $W$ retains its Brownian law, the early observations are
	independent of $\kappa$ and $g$ in distribution, and
	the first arrival has density $
	\PP_\varepsilon(\tau\in\mathrm ds\mid W)
	=\varepsilon a(s)q(W_s)e^{-\varepsilon B_s^a}\dd s,$ with $ 
	B_s^a=\int_0^s a(r)q(W_r)\dd r,\quad 0<s<t_0.$
	
	Set $q_+=1+\theta$, $b_*=q_+t_0$ and
	$Z_*=V\int_0^{t_0}w(t)\dd t$.
	Then $0\leq B_s^a\leq b_*$,
	$|e^{-\varepsilon B_s^a}-1|\leq\varepsilon b_*$
	and $|z_\psi(s)|\leq Z_*$.
By using $\EE q(W_s)=1$, for $j=1,2$, $$
		\EE_\varepsilon Z_{t_0}^j
		=\varepsilon\int_0^{t_0}a(s)z_\psi(s)^j
		\EE[q(W_s)e^{-\varepsilon B_s^a}]\dd s,$$ and $$
		\left|\EE_\varepsilon Z_{t_0}^j
		-\varepsilon\int_0^{t_0}a(s)z_\psi(s)^j\dd s\right|
		\leq\varepsilon^2b_*t_0Z_*^j.$$
		
	Also, $|\EE_\varepsilon Z_{t_0}|\leq\varepsilon t_0Z_*$
	and $\EE_\varepsilon Z_{t_0}^2\leq\varepsilon t_0Z_*^2$.
	It follows from \eqref{eq:actualvariance} that for $t>t_0$,
	\begin{equation}\label{eq:varianceactivation}
		\mathcal V_R(\mu_t^\varepsilon)
		=\EE_\varepsilon Z_{t_0}^2-(\EE_\varepsilon Z_{t_0})^2
		=\varepsilon Q(a,\psi)+O(\varepsilon^2),
		\qquad t\geq t_0.
	\end{equation}
	For $\mathcal V^\varepsilon
	=\mathcal V_R(\mu_{t_0}^\varepsilon)$ and
	$C_V=(b_*t_0+t_0^2)Z_*^2$, these estimates give
	$|\mathcal V^\varepsilon-\varepsilon Q(a,\psi)|
	\leq C_V\varepsilon^2$ and
	$0\leq\mathcal V^\varepsilon\leq\varepsilon t_0Z_*^2$.
	The variance is deterministic and independent of $g$.
	Since $|\tanh x|\leq|x|$ and $|\tanh x-x|\leq|x|^3/3$,
	\begin{equation}\label{eq:latermultiplierexpansion}
		\ell^\varepsilon
		=\lambda_0+\kappa\tanh\mathcal V^\varepsilon
		=\lambda_0+\varepsilon\kappa Q(a,\psi)+O(\varepsilon^2),
	\end{equation}
	where $
	|\ell^\varepsilon-\lambda_0|
	\leq|\kappa|t_0Z_*^2\varepsilon$, and $
	|\ell^\varepsilon-\lambda_0-\varepsilon\kappa Q(a,\psi)|
	\leq C_\ell\varepsilon^2,$ 
	with $C_\ell=|\kappa|\bigl(C_V+(t_0Z_*^2)^3/3\bigr)$.
	
	For the early cost, the same density gives
	\begin{align*}
		\mathcal E^\varepsilon
		&:=\EE_\varepsilon\int_0^{t_0}
		\left(\frac12v_t^2-v_tW_t+k\alpha_t\right)\dd t\\
		&=\varepsilon k\int_0^{t_0}a(s)\dd s
		+\varepsilon\int_0^{t_0}a(s)\int_s^{t_0}
		\EE\left[
		q(W_s)e^{-\varepsilon B_s^a}
		\left(\frac12\psi(t,s)^2-\psi(t,s)W_t\right)
		\right]\dd t\,\dd s\\
		&=\varepsilon\int_0^{t_0}a(s)
		\left[
		k+\int_s^{t_0}
		\left(\frac12\psi(t,s)^2-\theta r(s)\psi(t,s)\right)\dd t
		\right]\dd s+O(\varepsilon^2).
	\end{align*}
 Since  the survival factor introduces at most
 $C_{\mathrm e}\varepsilon^2$ with
 $C_{\mathrm e}=b_*q_+t_0^2
 \bigl(V^2/2+V\sqrt{2t_0/\pi}\bigr)/2$, then by using 	Fubini's theorem and the uniform remainder, we have 
	\[
	\int_0^{t_0}\int_s^{t_0}
	\EE\left[
	q(W_s)\left(\frac{V^2}{2}+V|W_t|\right)
	\right]\dd t\,\dd s
	\leq q_+\frac{t_0^2}{2}
	\left(\frac{V^2}{2}+V\sqrt{\frac{2t_0}{\pi}}\right)<\infty.
	\]

	For the middle interval, the Poisson likelihood gives, for
	$0\leq z\leq1$,
	\[
	\EE_\varepsilon[z^N\mid W,\mathcal F_{t_0}^Y]
	=\exp\{\ell^\varepsilon L(1+\theta H)(z-1)\}.
	\]
	Together with the unit conditional means of the likelihood factors,
	this shows that $H$ retains its Brownian law and
	$N\mid H\sim\operatorname{Poisson}(\ell^\varepsilon L(1+\theta H))$.
	Both $\mathcal E^\varepsilon$ and $\ell^\varepsilon$ are independent
	of $g$. Thus  $
	J_\kappa(u^{\varepsilon;a,\psi,g})
	=\mathcal E^\varepsilon+\mathcal C(\ell^\varepsilon,g),$ and $ 
	\inf_{g\in\mathcal D_V}J_\kappa(u^{\varepsilon;a,\psi,g})
	=\mathcal E^\varepsilon+\mathcal R(\ell^\varepsilon).$
	
	Moreover, every $\ell^\varepsilon$ belongs to
	$I_\kappa=[\lambda_0-|\kappa|,\lambda_0+|\kappa|]\subset(0,\infty)$.
	By Lemma~\ref{lem:latercost},
	$M_\kappa=\sup_{\ell\in I_\kappa}|\mathcal R''(\ell)|<\infty$.
Then, by using	Taylor's formula and \eqref{eq:laterdecisiongap}, $$
		|\mathcal R(\ell^\varepsilon)-J_{\mathrm{base}}
		+G(\ell^\varepsilon-\lambda_0)|
		\leq\frac{M_\kappa}{2}|\ell^\varepsilon-\lambda_0|^2,$$
		and $$
		0\leq\mathcal C(\ell^\varepsilon,g_0)
		-\mathcal R(\ell^\varepsilon)
		\leq\frac{DL^2\theta^2}{2}|\ell^\varepsilon-\lambda_0|^2.$$ 
	In particular,
	\[
	|\mathcal R(\ell^\varepsilon)-J_{\mathrm{base}}
	+\varepsilon\kappa GQ(a,\psi)|
	\leq
	\left[
	GC_\ell+\frac{M_\kappa}{2}(|\kappa|t_0Z_*^2)^2
	\right]\varepsilon^2.
	\]
	Combining this with the early-cost estimate proves
	\eqref{eq:feedbackuniform}--\eqref{eq:feedbackoptimizedlate},
	with both remainders bounded by $C\varepsilon^2$, where
	\[
	C=C_{\mathrm e}+GC_\ell
	+\frac{M_\kappa+DL^2\theta^2}{2}
	(|\kappa|t_0Z_*^2)^2.
	\]
	The same $C$ works for all $(a,\psi)$ and
	$0\leq\varepsilon\leq1$; no uniform expansion for arbitrary $g$
	is needed. Since $\Gamma_{\kappa,\psi}-\Gamma_{0,\psi}
	=-\kappa Gz_\psi^2$, \eqref{eq:feedbackdifference} follows.
	
	Finally, $\bar Z_{t_0}=0$ under the baseline physical probability,
	so $(\mathcal V_R)_\mu(\bar\mu_t;y)=\phi_R(y_z)$ for $t\geq t_0$.
Based on 	Theorem~\ref{thm:differentiability}, $\chi=0$ and
	\eqref{eq:varianceactivation} , we have 
	\[
	\left.\frac{\mathrm d}{\mathrm d\varepsilon}
	\mathcal V_R(\mu_{t_0}^\varepsilon)\right|_{0+}
	=\EE[K_{t_0}(\bar Z_{t_0})^2]=Q(a,\psi).
	\]
	Because $K$ is a square-integrable martingale and
	$\bar Z_{t_0}$ is bounded and $\mathcal F_{t_0}$-measurable,
	$\EE[K_t(\bar Z_{t_0})^2]=Q(a,\psi)$ for $t\geq t_0$.
	Substitution into \eqref{eq:feedbackmeasurederivative} proves
	\eqref{eq:feedbackratevariation}.
\end{proof}

Although $\PP_\kappa^{\bar u^\psi}=\PP_0^{\bar u^\psi}$ and
$J_\kappa(\bar u^\psi)=J_0(\bar u^\psi)=J_{\mathrm{base}}$,
\eqref{eq:feedbackdifference} shows that positive $\kappa$ lowers
the activation derivative and negative $\kappa$ raises it whenever
$Q(a,\psi)>0$.

For the boundary decomposition, let
$E_{\mathrm e}=\{Y_{t_0}\geq1\}$,
$E_{\mathrm f}=\{Y_T-Y_{t_1}\geq1\}$ and
$\PP_\varepsilon=\PP_\kappa^{u^{\varepsilon;a,\psi,g_0}}$.
The strictly positive middle intensity gives
$\{\bar M_T>0\}=(E_{\mathrm e}\cup E_{\mathrm f})^c$
up to $\QQ$-null sets, whereas
$\PP_\varepsilon(E_{\mathrm f})=0$ and
\[
\PP_\varepsilon(E_{\mathrm e})
=\EE[1-e^{-\varepsilon B_{t_0}^a}]
=\varepsilon\int_0^{t_0}a(s)\dd s+O(\varepsilon^2).
\]
Proposition~\ref{prop:singular} therefore gives
$\nu^{\perp,a}(\Omega)=\int_0^{t_0}a(s)\dd s$, independently of
$\psi,\kappa$, while the cost derivative depends on both.
Theorem~\ref{thm:costderivative} and Proposition~\ref{prop:duality}
identify the same derivative as
\[
DJ_\kappa(\bar u^\psi;h)
=\EE\int_0^{t_0}\widehat H_\alpha^{\kappa,\psi}(t)a(t)\dd t
=\int_0^{t_0}a(s)\Gamma_{\kappa,\psi}(s)\dd s.
\]

\subsection{Activation gradient and the maximum principle}

Fix $\psi\in\Psi_V$ and the complete baseline policy $\bar u^\psi$.
Let $H^{\kappa,\psi}$ be the Hamiltonian evaluated along this baseline
and its adjoint processes from Proposition~\ref{prop:adjoint}, and let
$\widehat H_\alpha^{\kappa,\psi}$ denote the $\alpha$-component of
the observation-predictable projection \eqref{eq:projection}.

\begin{proposition}\label{prop:activationgradient}
	For $\mathrm ds\otimes\mathrm d\QQ$-almost every
	$(s,\omega)\in(0,t_0)\times\Omega$,
	\begin{equation}\label{eq:activationgradient}
		\widehat H_\alpha^{\kappa,\psi}(s)
		=e^{cs}\ind_{\{Y_{s-}=0\}}\Gamma_{\kappa,\psi}(s).
	\end{equation}
	In particular, if $\bar u^\psi$ is optimal over $\Uad$, then
	\begin{equation}\label{eq:activationnecessary}
		\Gamma_{\kappa,\psi}(s)\geq0
		\quad\text{for Lebesgue-almost every }s\in(0,t_0).
	\end{equation}
\end{proposition}

\begin{proof}
	Proposition~\ref{prop:adjoint} gives
	$\widehat H_\alpha^{\kappa,\psi}\in
	L^2((0,t_0)\times\Omega,\mathrm ds\otimes\mathrm d\QQ)$.
	The kernel $\Gamma_{\kappa,\psi}$ is continuous and bounded:
	$r,w,\psi$ are continuous on their compact domains, and the
	integrals in \eqref{eq:feedbackindex} preserve continuity.
	For every $a\in\mathcal A$, the policy
	$u^{\varepsilon;a,\psi,g_0}=\bar u^\psi+\varepsilon
	(a\ind_{[0,t_0]},0)$ is a feasible convex variation.
	Propositions~\ref{prop:duality} and~\ref{prop:feedbackactivation},
	together with \eqref{eq:projectionduality}, yield
	\[
	\int_0^{t_0}a(s)
	\left(\EE\bigl[\widehat H_\alpha^{\kappa,\psi}(s)\bigr]
	-\Gamma_{\kappa,\psi}(s)\right)\dd s=0.
	\]
	Fubini's theorem applies by the preceding integrability bounds.
	Taking $a$ to be indicators of Borel subsets of $(0,t_0)$ gives
	\begin{equation}\label{eq:activationmean}
		\EE\bigl[\widehat H_\alpha^{\kappa,\psi}(s)\bigr]
		=\Gamma_{\kappa,\psi}(s)
		\quad\text{for Lebesgue-almost every }s\in(0,t_0).
	\end{equation}
	
	On the early interval, $\bar\alpha=0$ and hence $\bar\lambda_\kappa=0$.
	The likelihood formula \eqref{eq:product} therefore gives
	\[
	\bar M_{s-}=e^{cs}\ind_{\{Y_{s-}=0\}}.
	\]
	The state coefficients are independent of $\alpha$, and
	$f_\alpha=k$, $(\lambda_\kappa)_\alpha=q(S_s)$ for $0<s<t_0$.
Therefore,
	\[
	H_\alpha^{\kappa,\psi}(s)
	=\bar M_{s-}\left[k+\ell_s^{\kappa,\psi}q(S_s)\right],
	\]
	where $\ell^{\kappa,\psi}$ is the counting-martingale integrand of
	the scalar adjoint in \eqref{eq:scalaradjoint}. Since $q$ and $\bar M$
	are bounded and $\ell^{\kappa,\psi}\in\mathcal J_Y^2$, then 
	this expression is square integrable.
Moreover,  note that 	the set $\{(s,\omega):Y_{s-}(\omega)=0\}$ belongs to
	$\mathcal P^Y$, thus projection preserves the vanishing of
	$H_\alpha^{\kappa,\psi}$ on $\{Y_{s-}>0\}$.
	
	For each deterministic $s>0$, the event $\{Y_{s-}=0\}$ is an atom
	of the observation-history $\sigma$-field $\mathcal F_{s-}^Y$,
	up to $\QQ$-null sets: every counting history in this event is
	identically zero before $s$. The usual augmentation only adds
	null sets to this statement. Since a predictable process has
	$\mathcal F_{s-}^Y$-measurable sections, there is a deterministic
	measurable function $\gamma$ such that $
	\widehat H_\alpha^{\kappa,\psi}(s)
	=\ind_{\{Y_{s-}=0\}}\gamma(s)$,  $\mathrm ds\otimes\mathrm d\QQ$-a.e. 
	Here a measurable version of $\gamma$ can be chosen as
	$e^{cs}\EE[\widehat H_\alpha^{\kappa,\psi}(s)]$.
	Under $\QQ$, $Y$ has constant intensity $c$, so
	$\QQ(Y_{s-}=0)=e^{-cs}$. Combining this with
	\eqref{eq:activationmean} gives
	$\gamma(s)=e^{cs}\Gamma_{\kappa,\psi}(s)$ almost everywhere,
	which proves \eqref{eq:activationgradient}.
	Its right-hand side is predictable and bounded on the finite early
	interval, so it specifies a predictable, square-integrable version of the projection.
	
	If $\bar u^\psi$ is optimal, the $\alpha$-component of
	\eqref{eq:normalineq} at $\bar\alpha=0$ requires
	$\widehat H_\alpha^{\kappa,\psi}\geq0$.
	Since $\QQ(Y_{s-}=0)=e^{-cs}>0$, by using Fubini's theorem and
	\eqref{eq:activationgradient}, \eqref{eq:activationnecessary} holds.
\end{proof}

Thus $\Gamma_{\kappa,\psi}$ gives the sign of the projected
activation gradient along histories with no early arrival.
If it is negative on a set of positive Lebesgue measure,
choosing $a=\ind_{\{\Gamma_{\kappa,\psi}<0\}}$ in
\eqref{eq:feedbackuniform} produces a strict cost decrease for all
sufficiently small $\varepsilon>0$.
The factor $e^{cs}$ in \eqref{eq:activationgradient} compensates
for the reference probability $e^{-cs}$ of a no-arrival history,
so the resulting directional cost is independent of the dominating
reference rate, consistently with Proposition~\ref{prop:referenceinvariance}.

\subsection{An explicit comparison of activation thresholds}

Set $A_s=\int_s^{t_0}w(t)^2\dd t$,
$b_s=\int_s^{t_0}w(t)\dd t$ and $d_\kappa(s)=1-2\kappa GA_s$.
Assume
\begin{equation}\label{eq:feedbackcoercivity}
	2\kappa GA_0<1.
\end{equation}
Then $d_\kappa(s)\geq1-2\kappa GA_0>0$ for $\kappa\geq0$,
and $d_\kappa(s)\geq1$ for $\kappa<0$.
Define
\begin{align}\label{eq:feedbackoptimalresponse}
	\psi_\kappa^*(t,s)=\theta r(s)\left[1+\frac{2\kappa Gb_sw(t)}{d_\kappa(s)}\right],\qquad
	B_\kappa(s)
	=\frac{\theta^2r(s)^2}{2}
	\left[t_0-s+\frac{2\kappa Gb_s^2}{d_\kappa(s)}\right]. 
\end{align}
To make both $\psi_\kappa^*$ and $\psi_0^*(t,s)=\theta r(s)$
feasible, assume also
\begin{equation}\label{eq:feedbacknonbinding}
	V>\max\{1,\|\psi_\kappa^*\|_{C(\Delta_{t_0})},
	\theta\|r\|_{C([0,t_0])}\}.
\end{equation}
Since $|g_\ell|\leq1$, $G$ is independent of $V\geq1$.
One may therefore choose $\kappa$ satisfying the model bounds and
\eqref{eq:feedbackcoercivity}, then $V$ satisfying
\eqref{eq:feedbacknonbinding}, and finally
$R>V\int_0^{t_0}w(t)\dd t$.
These additional conditions are used only for the explicit optimizer.

\begin{corollary}\label{cor:feedbackthreshold}
Under \eqref{eq:feedbackcoercivity}--\eqref{eq:feedbacknonbinding},
$\psi_\kappa^*\in\Psi_V$ minimizes the first-order continuation
cost at each arrival time, and
\begin{equation}\label{eq:feedbackfamilyvalue}
	\inf_{a\in\mathcal A,\,\psi\in\Psi_V,\,g\in\mathcal D_V}
	J_\kappa(u^{\varepsilon;a,\psi,g})
	=J_{\mathrm{base}}
	-\varepsilon\int_0^{t_0}(B_\kappa(s)-k)^+\dd s
	+O(\varepsilon^2).
\end{equation}
This first-order coefficient is achieved by taking $\psi=\psi_\kappa^*$
and $a_\kappa^*(s)=\ind_{\{B_\kappa(s)>k\}}$.
Moreover,
\begin{equation}\label{eq:feedbackstrictbenefit}
	B_\kappa(s)-B_0(s)
	=\frac{\kappa G\theta^2r(s)^2b_s^2}{1-2\kappa GA_s},
	\qquad
	B_0(s)=\frac{\theta^2r(s)^2}{2}(t_0-s).
\end{equation}
Let $k_\kappa^*=\max_{s\in[0,t_0]}B_\kappa(s)$.
The optimized first-order coefficient is strictly negative whenever $0\leq k<k_\kappa^*$,
and vanishes for $k\geq k_\kappa^*$.
In addition, $k_\kappa^*>k_0^*$ if $\kappa>0$,
while $k_\kappa^*<k_0^*$ whenever $\kappa<0$.

\end{corollary}

\begin{proof}
	Fix $s<t_0$ and define, on $L^2(s,t_0)$,
	\[
	F_s(\psi)
	=\frac12\int_s^{t_0}\psi(t)^2\dd t
	-\theta r(s)\int_s^{t_0}\psi(t)\dd t
	-\kappa G
	\left(\int_s^{t_0}w(t)\psi(t)\dd t\right)^2.
	\]
	Let $c_s=\int_s^{t_0}w(t)\psi(t)\dd t$, its derivative is
	\[
	DF_s(\psi)[\varphi]
	=\int_s^{t_0}
	[\psi(t)-\theta r(s)-2\kappa Gw(t)c_s]\varphi(t)\dd t.
	\]
Then, from the 	stationarity for every $\varphi\in L^2(s,t_0)$, we have 
	$\psi(t)=\theta r(s)+2\kappa Gw(t)c_s$ a.e.
	Multiplying by $w$ and integrating yields
	$d_\kappa(s)c_s=\theta r(s)b_s$.
	Since $d_\kappa(s)>0$, this gives
	$c_s=\theta r(s)b_s/d_\kappa(s)$ and the candidate
	\eqref{eq:feedbackoptimalresponse}.
	Conversely, this candidate has weighted integral
	$\theta r(s)b_s/d_\kappa(s)$ and then satisfies
	$DF_s(\psi_\kappa^*)[\varphi]=0$ for every $\varphi$.
	
	For any $\psi$, set $\eta=\psi-\psi_\kappa^*(\cdot,s)$.
By using Cauchy--Schwarz  inequality, we have 
	\[
	F_s(\psi)-F_s(\psi_\kappa^*)
	=\frac12\|\eta\|_2^2
	-\kappa G
	\left(\int_s^{t_0}w(t)\eta(t)\dd t\right)^2
	\geq
	\begin{cases}
		\frac12d_\kappa(s)\|\eta\|_2^2,&\kappa\geq0,\\
		\frac12\|\eta\|_2^2,&\kappa<0.
	\end{cases}
	\]
	Thus the candidate is the unique global minimizer.
	Taking $\varphi=\psi_\kappa^*$ in the stationarity identity, we have 
	\[
	F_s(\psi_\kappa^*)
	=-\frac{\theta r(s)}2
	\int_s^{t_0}\psi_\kappa^*(t,s)\dd t
	=-\frac{\theta^2r(s)^2}{2}
	\left[t_0-s+\frac{2\kappa Gb_s^2}{d_\kappa(s)}\right]
	=-B_\kappa(s).
	\]
	Since $\inf_s d_\kappa(s)>0$, the formula for $\psi_\kappa^*$
	is continuous on $\Delta_{t_0}$, and
	\eqref{eq:feedbacknonbinding} ensures $\psi_\kappa^*\in\Psi_V$.
	At $s=t_0$ the continuation cost is zero.
	
	It follows that, for every admissible $(a,\psi)$,
	\[
	\int_0^{t_0}a(s)\Gamma_{\kappa,\psi}(s)\dd s
	\geq\int_0^{t_0}a(s)(k-B_\kappa(s))\dd s
	\geq-\int_0^{t_0}(B_\kappa(s)-k)^+\dd s.
	\]
Since  $B_\kappa$ is continuous, then  $a_\kappa^*$ Borel measurable,
	and equality holds at $(a_\kappa^*,\psi_\kappa^*)$.
	
	Let $m_\kappa(k)=-\int_0^{t_0}(B_\kappa(s)-k)^+\dd s$.
	For every $(a,\psi,g)$, \eqref{eq:feedbackoptimizedlate} gives
	\[
	J_\kappa(u^{\varepsilon;a,\psi,g})
	\geq\inf_{\widetilde g\in\mathcal D_V}
	J_\kappa(u^{\varepsilon;a,\psi,\widetilde g})
	\geq J_{\mathrm{base}}
	+\varepsilon m_\kappa(k)-C\varepsilon^2.
	\]
	Here $C$ is independent of $(a,\psi,g)$ and $\varepsilon$.
	Taking the infimum gives the lower bound.
	Applying \eqref{eq:feedbackuniform} to the feasible policy
	$(a_\kappa^*,\psi_\kappa^*,g_0)$ gives the upper bound, so
	\begin{align*}
		J_{\mathrm{base}}
		-\varepsilon\int_0^{t_0}(B_\kappa(s)-k)^+\dd s
		-C\varepsilon^2
		&\leq
		\inf_{a,\psi,g}J_\kappa(u^{\varepsilon;a,\psi,g})\\
		&\leq
		J_\kappa(u^{\varepsilon;a_\kappa^*,\psi_\kappa^*,g_0})\\
		&\leq
		J_{\mathrm{base}}
		-\varepsilon\int_0^{t_0}(B_\kappa(s)-k)^+\dd s
		+C\varepsilon^2.
	\end{align*}
	This proves \eqref{eq:feedbackfamilyvalue}.
	The inequalities hold for every $0<\varepsilon\leq1$ with the same
	$C$, so they also cover policies depending on $\varepsilon$ and
	require no finite-$\varepsilon$ minimizer.
	
And, 	subtracting $B_0$ proves \eqref{eq:feedbackstrictbenefit}.
	Its right-hand side has the sign of $\kappa$ for $0<s<t_0$.
	Also, $B_\kappa(0)=B_\kappa(t_0)=0$ and $B_\kappa(s)>0$ in the
	interior. For $\kappa\geq0$ the latter is immediate; for $\kappa<0$, since 
	$b_s^2\leq(t_0-s)A_s$, then 
	\[
	t_0-s+\frac{2\kappa Gb_s^2}{d_\kappa(s)}
	\geq\frac{t_0-s}{1+2|\kappa|GA_s}>0.
	\]
	Since $B_\kappa$ is continuous, its positive maximum is attained
	in $(0,t_0)$.
	For $\kappa>0$, choose a maximizer $s_0\in(0,t_0)$ of $B_0$; then
	$k_\kappa^*\geq B_\kappa(s_0)>B_0(s_0)=k_0^*$.
	For $\kappa<0$, choose a maximizer $s_\kappa\in(0,t_0)$ of
	$B_\kappa$; then
	$k_\kappa^*=B_\kappa(s_\kappa)<B_0(s_\kappa)\leq k_0^*$.
	
	Finally, if $0\leq k<k_\kappa^*$,  then there exist an interval
	$I\subset(0,t_0)$ of positive length and $\delta>0$ such that
	$B_\kappa(s)-k\geq\delta$ on $I$. Hence
	$\int_0^{t_0}(B_\kappa(s)-k)^+\dd s\geq\delta|I|>0$.
	For $k\geq k_\kappa^*$ the integrand vanishes identically.
	This proves the threshold statement, including equality.
\end{proof}

For $\kappa>0$ and $k_0^*<k<k_\kappa^*$, $$
	\inf_{a,\psi,g}J_0(u^{\varepsilon;a,\psi,g})
	=J_{\mathrm{base}}+O(\varepsilon^2),$$
	and $$
	\inf_{a,\psi,g}J_\kappa(u^{\varepsilon;a,\psi,g})
	=J_{\mathrm{base}}
	-\varepsilon\int_0^{t_0}(B_\kappa(s)-k)^+\dd s
	+O(\varepsilon^2),$$ 
with a strictly positive integral. So for all sufficiently small $\varepsilon>0$, the policy $(a_\kappa^*,\psi_\kappa^*,g_0)$
gives a lower cost than the baseline. When $\kappa=0$, the first-order coefficient equals zero.
This does not mean we cannot still get improvements at second order.

The activation value and thresholds we found above apply only to the
policy family \eqref{eq:feedbackpolicy}. Under these policies,
the control stays zero until the first early arrival,
and the final decision only uses the count $N$.
Our analysis does not tell us the best possible value among  all admissible controls.
Different choices of $\psi$ yield different reference baselines in the reference space,
even though every one of them has the same physical baseline cost $J_{\mathrm{base}}$.

\section{Conclusion}

This paper develops a first-order variational analysis of partially
observed McKean--Vlasov jump-diffusions with bounded controlled
counting intensities that may vanish and depend on the unconditional
physical state distribution.
Under the stated regularity and moment assumptions, a deterministic upper bound on the likelihood supports one-sided state--likelihood differentiation
and a necessary stochastic maximum principle for a convex control
set, without inverse intensity or likelihood terms.
The physical probability has a one-sided total-variation derivative
on the fixed noise space, with a nonnegative singular contribution
from baseline-null histories.

A quadratic example shows how early activation changes the physical
variance of accumulated actions and thereby affects later observation
intensity. Within the specified policy family, under coercivity and
nonbinding response bounds, positive variance feedback raises the
first-order activation fee threshold relative to zero feedback,
while negative feedback lowers it.

	\bibliographystyle{plainnat}
	\bibliography{refers}
	
\end{document}